\documentclass[11pt,reqno]{amsart}
\usepackage[letterpaper,margin=1in,footskip=0.25in]{geometry}
\usepackage{mathrsfs}
\usepackage{microtype}
\usepackage{mathtools}
\usepackage{enumitem}
\PassOptionsToPackage{dvipsnames}{xcolor}
\usepackage{tikz-cd}
\PassOptionsToPackage{pdfusetitle,pagebackref,colorlinks}{hyperref}
\usepackage{bookmark}
\hypersetup{citecolor=OliveGreen,linkcolor=Mahogany,urlcolor=Plum}
\usepackage[hyphenbreaks]{breakurl}
\usepackage{amssymb}
\usepackage{bbm}
\usepackage{enumitem}
\usepackage{upgreek}
\usepackage{bm}
\usepackage{stmaryrd}
\usepackage{mathrsfs}
\usepackage{comment}
\usepackage{graphicx}

\allowdisplaybreaks

\makeatletter
\providecommand{\leftsquigarrow}{%
	\mathrel{\mathpalette\reflect@squig\relax}%
}
\newcommand{\reflect@squig}[2]{%
	\reflectbox{$\m@th#1\rightsquigarrow$}%
}
\makeatother

\newcommand{\Z}{\mathbb{Z}}
\newcommand{\N}{\mathbb{N}}

\newcommand{\cB}{\mathcal{B}}

\newcommand{\cD}{\mathcal{D}}

\newcommand{\cH}{\mathcal{H}}

\newcommand{\cL}{\mathcal{L}}
\newcommand{\cM}{\mathcal{M}}

\newcommand{\cO}{\mathcal{O}}

\newcommand{\cX}{\mathcal{X}}

\newcommand{\hR}{\widehat{R}}
\newcommand{\hK}{\widehat{K}}
\newcommand{\hX}{\widehat{X}}
\newcommand{\hB}{\widehat{B}}

\newcommand{\hH}{\widehat{H}}
\newcommand{\hV}{\widehat{V}}
\newcommand{\hv}{\widehat{v}}
\newcommand{\hF}{\widehat{F}}

\newcommand{\hY}{\widehat{Y}}
\newcommand{\hE}{\widehat{E}}
\newcommand{\hZ}{\widehat{Z}}
\newcommand{\heta}{\widehat{\eta}}
\newcommand{\hy}{\widehat{y}}

\newcommand{\fm}{\mathfrak{m}}

\newcommand{\fX}{\mathfrak{X}}

\newcommand{\fB}{\mathfrak{B}}

\newcommand{\fD}{\mathfrak{D}}

\newcommand{\bP}{\mathbb{P}}
\newcommand{\bC}{\mathbb{C}}
\newcommand{\bR}{\mathbb{R}}
\newcommand{\bA}{\mathbb{A}}
\newcommand{\bQ}{\mathbb{Q}}
\newcommand{\bZ}{\mathbb{Z}}
\newcommand{\bD}{\mathbb{D}}

\newcommand{\bG}{\mathbb{G}}
\newcommand{\bF}{\mathbb{F}}

\newcommand{\bN}{\mathbb{N}}
\newcommand{\bk}{\mathbbm{k}}

\newcommand{\bL}{\mathbb{L}}

\newcommand{\red}{\mathrm{red}}

\newcommand{\lc}{\mathrm{lc}}

\DeclareMathOperator{\Exc}{Exc}

\DeclareMathOperator{\Spec}{Spec}
\DeclareMathOperator{\vol}{vol}
\DeclareMathOperator{\mult}{mult}

\DeclareMathOperator{\gr}{gr}

\DeclareMathOperator{\id}{id}
\DeclareMathOperator{\im}{Im}
\DeclareMathOperator{\ord}{ord}

\DeclareMathOperator{\Proj}{Proj}

\DeclareMathOperator{\Supp}{Supp}

\DeclareMathOperator{\Val}{Val}
\DeclareMathOperator{\rk}{rk}
\DeclareMathOperator{\LC}{LC}

\newcommand{\Hom}{\mathrm{Hom}}
\newcommand{\bT}{\mathbb{T}}

\newcommand{\QM}{\mathrm{QM}}

\newcommand{\tX}{\widetilde{X}}

\newcommand{\tD}{\widetilde{D}}
\newcommand{\tR}{\widetilde{R}}
\newcommand{\tK}{\widetilde{K}}
\newcommand{\tV}{\widetilde{V}}
\newcommand{\tA}{\widetilde{A}}
\newcommand{\tF}{\widetilde{F}}
\newcommand{\tL}{\widetilde{L}}
\newcommand{\tG}{\widetilde{G}}

\newcommand{\tU}{\widetilde{U}}

\newcommand{\tf}{\tilde{f}}

\newcommand{\tH}{\widetilde{H}}
\newcommand{\tB}{\widetilde{B}}

\newcommand{\la}{\lambda}

\newcommand{\tcB}{\widetilde{\mathcal{B}}}

\newcommand{\tcX}{\widetilde{\cX}}

\newcommand{\trop}{\mathrm{trop}}

\newcommand{\Sk}{\mathrm{Sk}}
\newcommand{\an}{\mathrm{an}}

\numberwithin{equation}{section}

\newtheorem{prop} {Proposition} [section]
\newtheorem{thm}[prop] {Theorem} 

\newtheorem{lem}[prop] {Lemma}
\newtheorem{cor}[prop]{Corollary}
\newtheorem{prop-def}[prop]{Proposition-Definition}

\newtheorem{lemma}[prop]{Lemma}
\newtheorem{corollary}[prop]{Corollary}

\newtheorem{thm-defn}[prop]{Theorem-Definition}

\theoremstyle{definition}

\newtheorem{defn}[prop]{Definition}

\theoremstyle{remark}
\newtheorem{rem}[prop]{Remark} 

\newtheorem{exa}[prop]{Example}

\title{Valuative independence for Calabi--Yau varieties}

\author{Harold Blum}
\address{Department of Mathematics\\
	University of Utah\\
	Salt Lake City, UT 84112, USA.}
\email{blum@math.utah.edu}

\address{School of Mathematics, Georgia Institute of Technology, GA 30332, USA}
\email{haroldblum@gatech.edu}

\author{Yuchen Liu}
\address{Department of Mathematics, Northwestern University, Evanston, IL 60208, USA.}
\email{yuchenl@northwestern.edu}

\date{\today}

\begin{document}
	
	\begin{abstract}
		We construct valuatively independent bases for the space of sections of an ample line bundle on a  log Calabi--Yau pair over the field of Laurent series
		and the ring of  regular functions on an affine CY variety. 
		While the bases are not in general unique, they induce canonical functions on the respective skeletons and are expected to agree with tropicalizations of theta functions when they exist.
		The proof uses techniques from the study of higher rank degenerations in K-stability.
	\end{abstract}
	
	\maketitle
	
	\setcounter{tocdepth}{1}
	
	\tableofcontents

	\section{Introduction}

	The  Strominger--Yau--Zaslow (SYZ) conjecture \cite{SYZ96} provides a geometric explanation for mirror symmetry of Calabi--Yau varieties.  
	Roughly speaking, it predicts that, for a maximally degenerate projective family of  $n$-dimensional Calabi--Yau varieties $X \to \bD^*$ over the punctured complex disc, the fibers admit SYZ fibrations. 
	These are fibrations $X_t\to S$ over an $n$-dimensional topological manifold $S$ whose fibers over a big open subset of $S$ are special Lagrangian tori.
	Moreover, a mirror family $\check{X} \to \bD^*$ is expected to arise  by  dualizing the torus fibrations and taking a suitable compactification. 
	While the conjecture remains open in full generality, significant progress has been made using techniques from differential geometry, non-Archimedean geometry, and enumerative methods in algebraic geometry. 
	\medskip
	
	A key object in non-Archimedean approaches to the SYZ conjecture is the essential skeleton of a degeneration.
	The notion was first defined by  Kontsevich and Soibelman \cite{KS06} and was further developed using techniques from birational geometry by Musta\c{t}\u{a}, Nicaise, and Xu  \cite{MN15, NX16}.
	For a Calabi--Yau variety over $\bC(\!(t)\!)$, the essential skeleton is a subset of the Berkovich analytification and can be identified with a collection of real valued valuations of its function field.
	When the Calabi--Yau is maximally degenerate, the essential skeleton has the structure of a closed pseudomanifold \cite{NX16} and is expected to be related to the base of the SYZ fibration \cite{NXY19}.
	\medskip

	In a different direction, enumerative methods have been used to construct mirrors directly in works by  Gross--Hacking--Keel, Gross--Siebert, and Keel--White--Yu \cite{GHK15,GS26,KWY26}.
	A key feature of their approaches is the use of canonical bases. 
	In these constructions, the mirror  $\check{X}$ is defined as the $\Proj$ of a graded ring equipped with a canonical basis. 
	Thus each power of the induced ample line bundle on the mirror comes with a distinguished basis of global sections.
	These basis elements are called theta functions, because, when $X\to \bD^*$ is a degeneration of abelian varieties, they recover classical theta functions \cite[Section 2]{GS16}.
	Since $X$ should in turn be the  mirror of $\check{X}$, one expects that each power of an ample line bundle $L$ on $X$ also admits  a canonical basis of theta functions.
	\medskip
	
	An important expected property of these canonical bases  is \emph{valuative independence} \cite[Conjecture 13.6]{KWY26}.
	Roughly speaking, the  basis should be adapted to valuations in the essential skeleton so  that the pairing between basis elements and valuations resembles the pairing between monomials and toric valuations on a toric variety. 
	More precisely, valuative independence  implies that  the basis simultaneously diagonalizes all filtrations of the section ring induced by   valuations in the essential skeleton.
	Thus  the existence of a valuatively independent basis  reveals  hidden toric structure in the function theory of degenerations of CY varieties. 
	\medskip
	
	In this paper, we use techniques from birational geometry, in particular  the theory  of multi-degenerations in K-stability \cite{Xu21} (see also \cite{ABHLX20,XZ25,Che25}) to construct  valuatively independent bases for CY degenerations. 
	The constructions in this paper work in the singular setting of log canonical CY pairs and do not require any maximality hypothesis. 
	Unlike the constructions in  \cite{GHK15,GS26,KWY26} that produce canonical bases on the mirror, the construction in this paper shows the existence of such bases on the  original CY and does not single out a specific one. 
	Although valuatively independent  bases will not in general be unique, they all induce the same functions on the essential skeleton of the degeneration and hence are expected to encode the tropical data of mirror symmetry. 
	\medskip

	The results of this paper on the existence of valuatively independent bases in the case of maximally degenerate polarized CY varieties over $\mathbb{C}(\!(t)\!)$  provide a key algebro-geometric input in Yang Li’s recent work on the metric SYZ conjecture \cite{Li26}. 
	See Section \ref{ss:relmetricsyz} for further discussion.

	\subsection{Main results}
	This paper proves the existence of valuatively independent bases for affine CY varieties,  projective CY pairs with an ample line bundle, and CY degenerations. 
	These are the three settings in which canonical bases arising from mirror symmetry are expected to exist. 
	\medskip
	
	Throughout this article, $\bk$ denotes a fixed algebraically closed field of characteristic zero.
	
	\subsubsection{Affine CY varieties}
	
	\begin{thm}\label{thm:val-indepaffineCY}
		Let $U$ be an affine variety over $\bk$ with a projective compactification $X$ such that $B:=X\setminus U$ is a divisor, $(X,B)$ is an lc CY pair, and there exists an effective ample Cartier divisor $A$ on $X$ with $\Supp(A) = \Supp(B)$. 
		Then  there exists a   basis $(\theta_i)_{i \in I}$ of $H^0(U, \cO_U)$ such that 
		\[
		v\big( \textstyle \sum_{i\in I} a_i \theta_i \big) = \min \{ v(\theta_i) \, \vert\, a_i \neq 0\}
		\]
		for  all nonzero finite sums $\sum_{i \in I} a_i \theta_i$ with $a_i \in \bk$ and  $v\in \LC(X,B):=\{ v\in \Val_X\, \vert\, A_{X,B}(v)=0\}$.
	\end{thm}
	
	In the above theorem, 
	$\LC(X,B)$ denotes the set of lc places of $(X,B)$,  $\Val_X$ the set of valuations $v: K(X)^{\times} \to \bR$, and $A_{X,B}$ the log discrepancy; see Section \ref{ss:vals}.
	The set $\LC(X,B)$
	can be endowed with the  structure of a simplicial cone complex by Lemma \ref{l:LCplace} and can be thought of as a tropical CY. 
	The set is independent of the choice of compactification $X$ satisfying the hypotheses of Theorem \ref{thm:val-indepaffineCY} by \cite[Section 2.2]{Liu26}. 
	In the special case when $X$ is smooth and $B$ is a reduced snc divisor,  the valuations in $\LC(X,B)$ correspond to quasi-monomial valuations  with respect to local coordinates defining $B$ by Lemma \ref{l:LCplace}.
	\medskip
	
	We say that a basis for $H^0(U, \cO_U)$  is \emph{valuatively independent} if it satisfies the condition   in
	Theorem \ref{thm:val-indepaffineCY}. 
	The existence of such bases  was previously known to hold in several interesting cases. 
	When $(X,B)$ is a Looijenga pair, i.e. $X$ is a smooth surface and $B\in |-K_X|$ is a singular nodal curve, Mandel \cite{Man16} proved that the Gross--Hacking--Keel canonical basis constructed in \cite{GHK15,LZ25} satisfies valuative independence.
	When $U$ is a cluster variety, \cite{CMMM25} proved that the Gross--Hacking--Keel--Kontsevich canonical basis of \cite{GHKK} satisfies valuative independence. 
	Furthermore, the authors of  \cite{CMMM25} expect that their approach can be used to prove valuative independence for the theta functions constructed in \cite{KWY26} on the mirror of an affine CY variety.
	\medskip
	
	Theorem \ref{thm:val-indepaffineCY} proves the existence of valuatively independent bases, but does not single out a preferred basis. 
	In the special case when $U= \bG_m^n :=\Spec(\bk[x_1^{\pm1},\ldots, x_n^{\pm1}
	])$, the basis given by the monomials in $x_1,\ldots, x_n$ is the unique valuatively independent basis up to scaling basis elements by Proposition \ref{p:toric}. 
	In general, uniqueness fails even in the simple case when $U$ is a $\bZ/2\bZ$ quotient of a two-dimensional torus by Proposition \ref{p:cayleycubic}.
	\medskip
	
	While valuatively independent bases are not in general unique, the tropical data they encode is canonical. 
	In particular, for a basis element $\theta_i$, we define its tropicalization to be the function 
	$\theta_i^{\trop}: \LC(X,B)\to \bR$
	where $\theta_i^{\trop}(v)=v(\theta_i)$. 
	These functions are piecewise linear on $\LC(X,B)$, distinct when $(X,B)$ is maximal (Definition \ref{d:maximal}) by Lemma \ref{l:maxcythetaidistinct}, and independent of the choice of valuatively independent basis by Lemma \ref{l:tropthetaeasyprop}.
	Thus, when $U$ arises as a mirror of an affine CY pair, the tropicalization of the basis produced by Theorem \ref{thm:val-indepaffineCY}  is expected to agree with the tropicalization of the basis of theta functions from \cite{KWY26}.

	\subsubsection{Polarized CY varieties}

	\begin{thm}\label{thm:val-indlb}
		If $(X,B)$ is a projective lc CY pair over $\bk$ and $L$ is an ample line bundle on $X$, then there exists a basis $(\theta_1,\ldots, \theta_N)$ of $H^0(X,L)$ such that 
		\[
		v(a_1 \theta_1 + \cdots + a_N \theta_N)
		= 
		\min \{ v(\theta_i)\, \vert\, a_i \neq 0 \}
		\]
		for all $a_1,\ldots, a_N\in \bk$ and $v\in\LC(X,B):=\{v\in \Val_{X}\, \vert\, A_{X,B}(v)=0 \}$.
	\end{thm}
	
	The result is closely related to Theorem \ref{thm:val-indepaffineCY} by the projectivized cone construction. 
	Indeed, if $X$ is a smooth CY variety and $B=0$, then $U:= \Spec \oplus_{i \in \bN} H^0(X,L^{\otimes i})$ is an affine variety and admits a projective compactification 
	$Y:= \Proj \oplus_{m \in \bN} \oplus_{i=0}^m H^0(X,L^{\otimes i})x^{m-i}$ such that 
	$D:= Y\setminus U$ is an ample divisor and $(Y,D)$ is a projective log canonical CY pair.
	Using this relation, we prove Theorem \ref{thm:val-indlb} as a consequence of a more general version of Theorem \ref{thm:val-indepaffineCY}.
	
	\subsubsection{CY degenerations}

	\begin{thm}\label{thm:val-indep}
		If  $(X_K,B_K)$ is a projective log canonical CY pair over $K:=
		\bk(\!(t)\!)$ and $L_K$ is an ample line bundle on $X_K$ with a metric, then for every sufficiently divisible positive integer $m$, there exists a $K$-vector space basis
		$ \theta_1,\ldots, \theta_{N_m}$ for $H^0(X_K,L_K^{\otimes m})$ such that 
		\[
		v(a_1\theta_1+ \cdots + a_{N_m}\theta_{N_m}) = \min\{v(a_1 \theta_1), \cdots , v(a_{N_m}\theta_{N_m})\}
		\]
		for all $a_1,\ldots, a_{N_m}\in K$ and $v\in \Sk(X_K,B_K)$.
	\end{thm}
	
	In the above theorem, ${\rm Sk}(X_K,B_K)$ denotes the essential skeleton and is defined in Section \ref{ss:defskel}.
	The notion of a metric on $L_K$ is defined in Section \ref{sss:metrics} and is used to evaluate sections along valuations.
	\medskip

	We say that a basis for $H^0(X_K,L_K^{\otimes m})$ is \emph{valuatively independent}  if it satisfies the condition in Theorem \ref{thm:val-indep}.
	In the special cases when $(X_K,B_K)$ is the base change of a CY over the function field of an algebraic curve or satisfies a certain extension condition, there exists a valuatively independent basis for 
	$H^0(X_K,L_K^{\otimes m})$  for all integers $m\geq1$ by Theorems \ref{t:valindbasisalg} and \ref{t:semiamplevalind}.
	\medskip
	
	There are a few interesting cases for which Theorem \ref{thm:val-indep} was previously known to hold using an explicit basis. 
	Hultgren and Khalid  constructed such a basis when $X_K$ is a degeneration of the Fermat cubic \cite{HK26}.
	Even in this seemingly simple example, the basis is surprisingly complicated  to write down.
	Yang Li observed  that, for a maximally degenerate abelian variety over $K$ arising from Mumford's construction, the theta function basis is valuatively independent \cite[Example 3.21]{Li25}.
	He further observed that for certain non-maximal degenerations of a CY to a union of hypersurfaces,  the basis can be written down explicitly \cite[Example 3.19]{Li25}.
	\medskip
	
	As in the previous results, the proof only proves the existence of valuatively independent bases and does not single out a preferred such basis.
	While a basis satisfying Theorem \ref{thm:val-indep} is not in general unique, the tropical data is canonical. Indeed, each basis element $\theta_{i} \in H^0(X_K,L_K^{\otimes m})$ induces a piecewise-affine function 
	$\theta_{i}^{\rm trop}:{\rm Sk}(X_K,B_K)\to \bR$
	defined by  $\theta_{i}^{\rm trop}(v)=v(\theta_{i})$.
	The multiset of functions $(\theta_{i}^{\trop})_{i=1}^{N_m}$ is independent of the choice of valuatively independent basis up to translation by a constant $\bZ$-valued function, which corresponds to multiplying $\theta_i$ by an element in $K^\times$ (see Lemma \ref{l:valindDVRbasicprops}).
	Additionally, in the case when $\dim \Sk(X_K,B_K)= \dim X_K$, these functions are  always distinct by Lemma \ref{l:CYdegenmaxdistinct}.

	\subsection{Sketch of proof}
	We now explain the main idea behind the proof of Theorem \ref{thm:val-indep}. 
	The proofs of Theorems \ref{thm:val-indepaffineCY} and \ref{thm:val-indlb} are similar, but require fewer technical arguments.
	\medskip
	
	To prove Theorem \ref{thm:val-indep}, we relate valuative independence to the flatness of certain multigraded Rees modules, which arise naturally from degenerations of CY pairs over certain higher dimensional bases. 
	The argument is inspired by  techniques from the study of  multi-degenerations in K-stability  \cite{Xu21, ABHLX20,LXZ22,XZ25,Che25}. 
	\medskip
	
	To carry out this argument, we first use the projective cone construction to reduce the result to the case of a log canonical boundary polarized CY pair 
	$(X_K,B_K+D_K)$
	over $K:= \bk(\!(t)\!)$ and $L_K := \cO_{X_K}(D_K)$. This means that $(X_K,B_K+D_K)$ is a log canonical CY pair such that $B_K$ and $D_K$  are effective $\bQ$-divisors and  $D_K$ is $\bQ$-Cartier and ample. 
	In particular, $D_K$ captures the data of the ample line bundle.
	While this step is not strictly necessary, we find it more elegant to work with such pairs, rather than a CY pair with an ample line bundle, as it is easier to keep track of the $\bQ$-divisor $D_K$ under degeneration than the ample line bundle. 
	\medskip
	
	Next we choose an extension of $(X_K,B_K+D_K)$ to a boundary polarized CY pair $(X,B+D)$ over $R:= \bk[\![t]\!]$.
	For any $v\in{\rm Sk}(X_K,B_K +D_K)$, there exists an induced filtration $F_v$ of the free $R$-module $V_m:=H^0(X,mD)$ defined by 
	\[
	F_v^\la V_m := \{s \in H^0(X,mD)\, \vert\, v(s) \geq \la\}.
	\]
	Valuative independence can then be rephrased as saying that the collection  of filtrations $F_v$  where $v\in \Sk(X_K,B_K+D_K)$ are simultaneously diagonalizable (Definition \ref{d:simdiag}). 
	Furthermore, using properties of filtrations of modules over DVRs (Proposition \ref{p:dvrdiagfinitecol}), it suffices to verify diagonalizability for every finite  collection of  divisorial valuations in the skeleton.
	\medskip
	
	We then fix a finite collection of divisorial valuations 
	$v_1,\ldots, v_r\in\Sk(X_K,B_K+D_K)$
	with the aim of showing that $F_{v_1},\ldots, F_{v_r}$ are simultaneously diagonalizable.
	For each $i$,  we construct an  extension of $(X_K,B_K+D_K)$ to a boundary polarized CY pair 
	$(X^i,B^i+D^i)$
	over $R$ such that $E_{i}:=X^i_{0,\red}$ is a prime divisor  and $v_i = (b_i)^{-1} \ord_{E_i}$, where $b_i = {\rm mult}_{E_i}(X^i_0)$. 
	For simplicity assume that each $b_i=1$, i.e. each  $X^i_0$ is reduced. 
	In this case, we can relate these degenerations by constructing  a flat  $ \bG_m^r$-equivariant family of boundary polarized CY pairs 
	\[
	(\cX,\cB+\cD) \to {\rm S}^r:= \Spec\, R[x_0, \cdots , x_r]/(x_0 \cdots x_r-t)
	,\]
	such that
	\[
	(\cX,\cB+\cD)_{U_i }  \cong (X^i,B^i+D^i) \times_{R} U_i,
	\]
	over $U_i = D(x_0 \cdots \widehat{x}_i \cdots x_r) \subset {\rm S}^r$.
	Using  that $(\cX,\cB+\cD) \to {\rm S}^r$ is a flat family of slc pairs and vanishing theorems, we deduce that  $H^0(\cX,m\cD)$ is flat and its formation commutes with base change. We  then compute that 
	\[
	H^0(\cX,m\cD) \cong {\rm Rees}_1(F_{v_1},\ldots, F_{v_r};V_m), 
	\]
	as graded $R[x_0,\ldots, x_r]/(x_0\cdots x_r -t)$-modules, where the latter module  is the multigraded Rees  module of the collection of filtrations. 
	As a consequence of flatness, we deduce that the filtrations are simultaneously diagonalizable (Lemma \ref{l:Reesmodflat->diag}). 
	\medskip
	
	In the above discussion, we glossed over a key technical issue:
	the constructions of $(X,B+D)$ and $(\cX,\cB+\cD) \to {\rm S}^r$ rely on the lc MMP  in \cite{HX13}, which can only be applied when the original CY pair is the base change of a family of CY pairs over a curve or a DVR essentially of finite type over  $\bk$. 
	To circumvent this obstruction,  we first consider the case when $R$ is a DVR essentially of finite type over $\bk$ and $K:= {\rm Frac}(R)$ and then use an approximation argument as in \cite{MN15,NX16} to deduce the result from the algebraic case.
	\medskip
	
	When $r=1$, the quotient stack $[{\rm S}^1/ \bG_m^1]$ was previously used in \cite{AHLH23} to define S-completeness, which is used in  \emph{loc. cit.} to characterize when an algebraic stack admits a separated good moduli space. 
	Furthermore, the existence of the extension $(\cX,\cB+\cD) \to {\rm S}^1$ follows from the S-completeness of the moduli stack of boundary polarized CY pairs in \cite{ABB+}.
	For general $r> 1$, the quotient stack appears in the definition of the degeneration fan in \cite[Construction 4.1]{AHL26}.
	\medskip

	\subsection{Relation to metric SYZ}\label{ss:relmetricsyz}
	The metric SYZ conjecture predicts that, if 
	$X \to \bD^*$ is a maximally degenerate  family of Calabi--Yau varieties over the punctured complex disc and $L$ is a relatively ample line bundle on $X$, then  there are open subsets  $U_t\subset X_t$ admitting  SYZ fibrations with $\lim_{t\to 0} \vol(U_t)/\vol(X_t) =1$, where the volume is taken with respect to the CY metric in $c_1(L_t)$.
	In \cite{Li23}, Yang Li reduced the conjecture to  verifying a conjecture in non-Archimedean geometry: the comparison property for the solution to the NA Monge-Amp\`ere equation in \cite{BFJ15}. 
	In \cite{HJMM24,Li24}, the metric SYZ conjecture was solved for hypersurfaces in certain toric Fano varieties by using an optimal transport problem on the boundary of the toric polytope and its dual to prove a weak form of the comparison property.
	\medskip
	
	In \cite{Li25}, Yang Li proposed that the existence of    valuatively independent bases should  play a central role in extending this  optimal transport approach to establish the weak comparison property in general. 
	Roughly speaking, the pairing between the basis elements and valuations in the essential skeleton should  serve as a substitute for the pairing between the toric polytope and  its dual in the toric hypersurface case studied in \cite{HJMM24,Li24}.  
	Building on this,  Li proved in \cite{Li26}  that the existence of valuatively independent bases as stated in Theorem \ref{thm:val-indep}  implies the weak comparison property, and, hence, the metric SYZ conjecture. 
	Thus Theorem \ref{thm:val-indep} provides a key algebraic geometry ingredient in Li's  solution to the metric SYZ conjecture.

	\vspace{.5 cm}

	\noindent {\bf Structure of the paper.}
	The paper is structured as follows. 
	In Section \ref{s:prelims}, we discuss preliminaries on CY pairs, valuations, and essential skeletons. 
	In Section \ref{s:affineCYpairs}, we prove Theorems \ref{thm:val-indepaffineCY} and \ref{thm:val-indlb} on valuatively independent bases for affine and polarized CY pairs.
	These cases are simpler than the case of CY degenerations, which is considered in the remainder of the paper. 
	In Section \ref{s:filts}, we discuss  filtrations of free modules over DVRs and diagonalizing collections of such filtrations. 
	In Section \ref{s:bpcys}, we develop the theory of multi-degenerations of boundary polarized CY pairs. 
	In  Section \ref{s:valind}, we apply the latter results  to prove Theorem \ref{thm:val-indep} on the existence of valuatively independent bases for polarized CY degenerations.
	
	\vspace{.5 cm}

	\noindent {\bf Acknowledgments.}
	We thank Chenyang Xu, Mattias Jonsson, Sean Keel, and Fanjun Meng for helpful conversations and feedback and  Yang Li for sharing an early  draft of \cite{Li26}.
	HB is partially supported by NSF Grants DMS-2441378 and DMS-2200690.
	YL is partially supported by NSF Grant DMS-2237139 and an AT\&T Research Fellowship from Northwestern University.

	\vspace{.5 cm}

	\noindent {\bf AI Statement.} 
	The authors used AI for proofreading the paper including pointing out typos, grammatical errors, notational errors, and gaps in proofs. 
	In the process, the AI also suggested the statement and proof of Proposition \ref{p:diagordn}, which is stronger than the original version that we proved.
	The main results, proofs, and techniques are due to the authors.

	\section{Preliminaries}\label{s:prelims}
	
	Throughout, we work over an algebraically closed field $\bk$ of characteristic zero. We generally follow the notation for  singularities of pairs and birational geometry in \cite{Kol13}.
	
	\subsection{Pairs}
	\subsubsection{Definition}
	A pair $(X,B)$ consists of a $\bk$-scheme $X$ that is reduced, equidimensional, $S_2$, separated, excellent, and has a  canonical sheaf \cite[Definition 1.5]{Kol23} and an effective $\bQ$-divisor $B$ on $X$  such that no  irreducible component of $\Supp(B)$ is contained in the singular locus of $X$ and  $K_{X}+B$ is $\bQ$-Cartier. For notions of singularities of pairs that assume that $X$ is normal, e.g., lc, klt, dlt, see \cite[Definition 1.18]{Kol13}.
	For the notion of an slc pair, see \cite[Definition 5.10]{Kol13}.
	
	A pair over a scheme $S$ is a pair $(X,B)$ with a morphism $X\to S$. 
	The pair is projective over $S$ if $X\to S$ is projective.
	Unless specified otherwise, any projective pair over a field $\bk'\supset \bk$ is assumed to be geometrically connected.

	We are primarily interested in pairs $(X,B)$, where $X$ is of finite type over a field $\bk'\supset \bk$ or of finite type over a DVR $R$ containing $\bk$. In these cases, $X$ is necessarily excellent and has a canonical sheaf. 
	When $X$ is of finite type (or essentially of finite type) over $\bk'$, we can use results from the klt MMP \cite{BCHM10} and the lc MMP in \cite{HX13}.
	When $X$ is of finite type over a DVR containing $\bk$, we can use the  results in \cite{LM22} that extend the klt MMP \cite{BCHM10} to pairs that are projective over excellent schemes.

	\subsubsection{Calabi--Yau pairs}
	A projective pair $(X,B)$ over a scheme $S$ is a \emph{CY pair} if $(X,B)$ is slc and  $K_{X}+B \sim_{\bQ,S}0$.
	
	A \emph{boundary polarized CY pair}  over a field $\bk'\supset \bk$ is an slc pair $(X,B+D)$  projective over $\bk'$ such that
	$B$ and $D$ are effective $\bQ$-divisors, $K_{X}+B+D \sim_{\bQ} 0$, and $D$ is ample.
	Such pairs and their moduli are studied in \cite{ABB+,BL24}. 
	While we are primarily interested in boundary polarized CY pairs that are lc, the slc case appears naturally when considering degenerations. 
	
	\subsubsection{Families of pairs}
	Let $X\to S$ be a flat morphism of schemes with $S_2$ fibers. A \emph{relative Mumford divisor} $D$ on $X$ is a closed subscheme of $X$ such that there is an open set $U\subset X$ satisfying
	\begin{enumerate}
		\item  ${\rm codim}_{X_s}(X_s \setminus U_s) \geq 2$ for all $s\in S$, 
		\item $D\vert_U$ is an effective relative Cartier divisor,
		\item $D$ is the closure of $D\vert_U$, and
		\item $X_s$ is smooth at the generic points of $D_s$.
	\end{enumerate}
	If $T\to S$ is a morphism of schemes, then $D_T$ is the relative Mumford divisor on $X_T:= X\times_S T$ defined by pulling back $D\vert_U$ to $U\times_S T$ and taking the closure in $X_T$.
	
	A \emph{family of slc pairs} $(X,B) \to S$  over a reduced Noetherian scheme $S$ consists of a flat finite type morphism $X\to S$ together with a formal sum $B=\sum_{i=1}^rb_i B_i$, where each $b_i\in \bQ_{>0}$ and each $B_i$ is a relative Mumford divisor on $X$, such that 
	\begin{enumerate}
		\item $K_{X/S}+B$ is $\bQ$-Cartier and
		\item $(X_s,B_s)$ is slc for all $s\in S$. 
	\end{enumerate}
	A \emph{family of KSBA stable pairs} is a family of slc pairs $(X,B) \to S$ such that $X\to S$ is projective and $K_{X/S}+B$ is relatively ample over $S$.
	A \emph{family of boundary polarized CY pairs} $(X,B+D) \to S$  is a family of slc pairs such that $X\to S$ is  projective, $K_{X/S}+B+D \sim_{\bQ,S}0$, and $D$ is $\bQ$-Cartier and relatively ample over $S$. 
	
	\begin{prop}\label{p:kollarsmoothslcfam}
		Let $0 \in S$ be a local regular scheme and $H:=H_1+\ldots+H_r $ a reduced snc divisor on $S$ with $H_1 \cap \cdots \cap H_r= \{0\}$. Let $f:X\to S$ be a proper surjective morphism of schemes and $B$ an effective $\bQ$-divisor. The following are equivalent:
		\begin{enumerate}
			\item $(X,B+ f^*(H_1+\cdots +H_r))$ is slc; 
			\item $(X,B) \to S$ is a family of slc pairs. 
		\end{enumerate}
	\end{prop}
	
	\begin{proof}
		This follows from the proof of \cite[Theorem 4.54 and Corollary 4.56]{Kol23}.
	\end{proof}
	
	\begin{prop}\label{p:kollarsmoothslcfamgen}
		Let $(X,B)\to S $ be a proper family of slc pairs over a  regular local scheme $0\in S$. The following hold:
		\begin{enumerate}
			\item The pair $(X,B)$ is an slc pair and every  lc center of the pair dominates $S$.
			\item If the generic fiber of $X\to S$ is normal, then $(X,B)$ is lc.
		\end{enumerate}
	\end{prop}
	
	\begin{proof}
		Statement (1) follows from   
		Proposition \ref{p:kollarsmoothslcfam}. See \cite[Corollary 4.56]{Kol23} for details.

		For (2),  assume that  the generic fiber of $X\to S$ is normal. 
		By (1), $(X,B)$ is slc. 
		If $(X,B)$ is not lc, then there exists a codimension one point $\eta \in X$ at which $X$ is non-normal. 
		Since the closure of $\eta$ is a log canonical center of $(X,B)$, it must dominate $S$ by (1). 
		Hence $\eta \in X$ is also a non-normal point on the generic fiber, which is a contradiction. 
		Therefore $(X,B)$ is lc.  
	\end{proof}

	\subsection{Valuations}\label{ss:vals}
	\subsubsection{Valuations}
	Let $(X,B)$ be a pair such that $X$ is integral. 
	A \emph{valuation} on $X$ will mean a real valuation $v: K(X)^\times \to \bR$ that admits a center on $X$.
	A \emph{center} of $v$ on $X$ is a point $\xi \in X$ such that $v$ is $\geq 0$ on $\cO_{X,\xi}$ and $>0$ on $\fm_{\xi} \subset \cO_{X,\xi}$. 
	Since $X$ is assumed to be separated, such a point $\xi$ is unique, and we set $c_X(v):= \xi$.

	We write $\Val_{X}$ for the set of valuations on $X$. 
	We endow $\Val_X$ with the topology of pointwise convergence as in \cite{JM12}.
	
	For a valuation $v\in \Val_{X}$ and  a Cartier divisor $D$ on $X$, we write $v(D):= v(f)$, where $f\in K(X)^\times $ is a local equation of $D$ at $c_{X}(v)$. More generally, when $D$ is a $\bQ$-Cartier $\bQ$-divisor, we set $v(D) = \tfrac{1}{m}v(mD)$, where $m>0$ is sufficiently divisible so that $mD$ is a Cartier divisor. 
	For a line bundle $\cL$ on $X$ and $0\neq s\in H^0(X,\cL)$, we fix a trivialization $\cL_{c_{X}(v)} \cong \cO_{X,c_{X}(v)}$ and write $v(s)$ for the value of $v$ along the regular function corresponding to  $s$ under this trivialization.
	By convention, $v(0)=\infty$.
	
	\subsubsection{Divisorial valuations}
	Let $f: Y \to X$ be a proper birational morphism with $Y$ normal. 
	A prime divisor $E\subset Y$ induces a valuation $\ord_{E}:K(X)^\times \to \bZ$  given by order of vanishing along $E$. 
	A \emph{divisorial valuation} on $X$ is a valuation of the form $c \cdot \ord_{E}$, where $E$ is as above and $c\in \bR_{\geq0}$. In particular, we use the convention that the trivial valuation is divisorial.

	\subsubsection{Quasi-monomial valuations}
	Let $f: Y \to X$ be a proper birational morphism. 
	Fix a not necessarily closed point $\eta \in Y$ at which $Y$ is regular and a regular system of parameters $y_1,\ldots, y_r\in \cO_{Y, \eta}$.
	Given $\alpha \in \bR_{\geq0}^r$, we write $v_{\alpha}$ for the valuation on $X$ satisfying the following. 
	For $h\in \cO_{X,f(\eta)}$, we can write $f^*h$  in $\widehat{\cO}_{Y, \eta} \cong k(\eta) \llbracket y_1,\ldots, y_r\rrbracket $ as $\sum_{\beta \in \bN^r} c_\beta y^{\beta}$, where $c_\beta \in k(\eta)$ and set 
	\[
	v_\alpha(h) := \min \{ \langle \alpha, \beta \rangle \, \vert\, c_\beta \neq 0 \}.
	\]
	A valuation of the form $v_\alpha$ as above is called quasi-monomial. 
	The valuation $v_{\alpha}$ is divisorial if and only if  $\alpha  = c \alpha'$ for some  $c\in \bR_{>0}$ and $\alpha' \in \bZ_{\geq0}^r$ \cite[Remark 3.9]{JM12}.
	
	Let $E= \sum_{i \in I} E_i$ be a reduced divisor on $Y$ such that either $E$ is snc or, more generally, for each $J\subset I$ and a generic point $\eta$ of $E_J: = \cap_{i \in J} E_i$, $E$ is snc at $\eta$.
	For such $\eta $ and $J$, fix a regular system of parameters  $(y_i)_{i \in J}$ of $\cO_{Y,\eta}$ such that each $y_i$ locally defines $E_i$. 
	We write ${\rm QM}_{\eta}(Y,E)$ for the quasi-monomial valuations 
	that can be described at $\eta$ with respect to $(y_{i})_{i \in J}$. Note that ${\rm QM}_{\eta}(Y,E) \cong \bR_{\geq 0}^{J}$ as topological spaces by \cite[Lemma 4.5]{JM12}. 
	We write ${\rm QM}(Y,E) = \cup_{\eta} {\rm QM}_\eta(Y,E)$, which has the structure of a simplicial cone  complex.

	\subsubsection{Log discrepancy}
	Let $(X,B)$ be a pair such that $X$ is normal. 
	There is a  log discrepancy function $A_{X,B}:\Val_{X} \to \bR\cup \{+\infty\}$, which was defined when $X$ is smooth and $B=0$ in \cite{JM12}, when $B=0$ in \cite{BdFFU15}, and for normal pairs in \cite{Xu25}.
	While the latter assumes that $X$ is of finite type over a field, the definition and basic properties extend verbatim.

	For a divisorial valuation $c \cdot \ord_{E}$ arising from a proper birational morphism $f:Y \to X$ with $Y$ normal and $E$ a prime divisor on $Y$, the log discrepancy is
	\[
	A_{X,B}(c\cdot \ord_E ) = c (1+ {\rm coeff}_{E} (K_Y- f^*(K_X+B)))
	.\]
	If $f:Y \to X$ is a log resolution of $(X,B)$ and $E:= {\rm Exc}(f)+ \Supp(f_*^{-1}(B))$, then $A_{X,B}$ is linear on cones in ${\rm QM}(Y,E)$. 
	If $(X,B)$ is an lc pair, then the set of lc places of the pair is 
	\[
	{\rm LC}(X,B):= \{ v\in \Val_{X}\, \vert\, A_{X,B}(v) =0 \}
	.
	\]
	
	\begin{lem}\label{l:LCplace}
		Let $(X,B)$ be an lc pair and $f:Y \to X$ a proper birational morphism with $Y$ normal.  Write $B_Y$ for the $\bQ$-divisor on $Y$ such that $K_Y+ B_Y = f^*(K_X+B)$. 
		If either $f$ is a log resolution or $(Y,B_Y)$ is dlt, then
		\[
		{\rm LC}(X,B)
		=
		{\rm QM}(Y,(B_Y)^{=1}) 
		,\]
		where $(B_{Y})^{=1}$ denotes the coefficient 1 part of $B_Y$.
	\end{lem}
	
	\begin{proof}
		If $f$ is a log resolution, then the result holds by \cite[Proof of Lemma 2.3]{BLX22}. 
		If $(Y,B_Y)$ is dlt, then  \cite{Tem08} implies that there exists a log resolution
		$g:Z\to Y$ of $(Y,B_Y)$ that is an isomorphism over the locus where $(Y,\Supp(B_Y))$ is snc. 
		Let $B_Z$ be the $\bQ$-divisor on $Z$ such that $K_Z+B_Z = g^*(K_Y+B_Y)$.
		Now 
		\[
		{\rm LC}(X,B)
		=
		{\rm LC}(Y,B_Y)
		=
		{\rm QM}(Z,(B_Z)^{=1})
		=
		{\rm QM}(Y,(B_Y)^{=1}),
		\]
		where the first equality holds as $(Y,B_Y) \to (X,B)$ is crepant, the second holds by the log resolution case applied to $Z\to Y$, and the last holds as $(Y, B_Y)$ is dlt and $g$ is an isomorphism at the generic points of the strata of $(B_Z)^{=1}$.
	\end{proof}
	
	\subsubsection{Dlt modification}
	Let $(X,B)$ be an lc pair. A \emph{$\bQ$-factorial dlt modification} $f:(Y, B_Y)\to (X,B)$ is a proper birational morphism $f: Y \to X$ such that
	\begin{enumerate}
		\item $Y$ is $\bQ$-factorial,
		\item $K_Y+B_Y = f^*(K_X+B)$,
		\item $(Y, B_Y)$ is dlt, and 
		\item every   prime divisor  $E \subset \Exc(f)$ satisfies $A_{X,B}(\ord_E)=0$.
	\end{enumerate}
	
	\subsection{Skeleton}\label{ss:defskel}
	Let $R$ be a DVR containing $\bk$ with fraction field $K:= {\rm Frac}(R)$ and a choice of  uniformizer $\pi\in R$.\footnote{We are primarily interested in the case when $R= \bk\llbracket t\rrbracket $ or $R$ is the local ring at a smooth point on a curve over $\bk$.}
	We write $v_K: K  \to \bZ\cup \{ +\infty\}$ for the induced discrete valuation  on $K$ such that $v_K(\pi)=1$. 
	
	\subsubsection{Analytification}
	If $X_K$ is a finite type integral scheme over $K$, we write $X_K^{\rm an}$ for the set of pairs $(x, v_x )$ such that $x\in X_K$ is a scheme theoretic point and $v_x: k(x)^\times \to \bR$ is an $\bR$-valued  valuation extending $v_K$. By convention, $v_x(0)= +\infty$.
	The topology on $X_K^{\rm an}$ is the weakest topology such that 
	\begin{enumerate}
		\item[(i)] the map $i:X_K^{\an} \to X_K$ defined by  $(x, v_x )\mapsto x$ is continuous and 
		\item[(ii)] for every open set $U_K\subset X_K$ and $f\in \cO_{X_K}(U_K)$, the map $	i^{-1}(U_K) \to \bR\cup \{+\infty\}$ defined by $(x, v_x )\mapsto v_x(f(x))$ is continuous.  
	\end{enumerate}
	When $R$ is complete, $X_K^{\rm an}$  agrees with the Berkovich analytification of $X_K$ in \cite{Ber90}.

	\subsubsection{Dual complexes}
	Let $X$ be a flat integral scheme of finite type over $R$  and  $X_{0,\red} = \cup_{i \in I} E_i$ denote the decomposition into irreducible components.
	We assume that either  $X_{0,\red}$ is an snc divisor in $X$ or, more generally, 
	$X_{0,\red}$ is snc at each generic point  of an irreducible component of $ E_J:= \cap_{i\in J} E_i$ for any subset $J\subset I$.

	The \emph{dual complex} of $X_0$ is the regular $\Delta$-complex $\mathcal{D}(X_0)$ such  that the $n$-simplices are in bijection with the irreducible components of $E_J:=\cap_{j \in J} E_j$ such that $J\subset I$  is a subset of cardinality $n+1$.
	In particular, the $0$-simplices are in bijection with irreducible components of $X_0$.
	There is an inclusion
	\[
	\mathcal{D}(X_0) \hookrightarrow  X_{K}^{\rm an}
	\]
	that maps a $0$-simplex corresponding to a prime divisor $E_i$ to $v:= (b_i)^{-1} \ord_{E_i}$, where $b_i = {\rm mult}_{E_i}(X_0)$. 
	Furthermore, for an irreducible component $Z\subset E_J$ and a point in the corresponding simplex $\alpha_J=(\alpha_j)_{j \in J} \in \bR_{ \geq 0}^J$ with  $\sum_{j\in J} \alpha_jb_j = 1$, the point gets mapped to the quasi-monomial valuation 
	$v_{\alpha_J} \in \QM(X,X_{0,\red})$.
	Note that 
	\[
	\mathcal{D} (X_0) =  {\rm QM}(X,X_{0,\red}) \cap \{ v(\pi)=1\}
	\]
	as subsets of $X_K^{\rm an}$.

	\subsubsection{Log canonical model}
	
	\begin{defn}
		Let $(X_K,B_K)$ be an lc CY pair over $K$. An \emph{lc model} of $(X_K,B_K)$ is a  normal scheme $X$ projective over $R$ and a $\bQ$-divisor $B$ on $X$, whose support contains no irreducible component of $X_0$, and an isomorphism
		\[
		(X_K,B_K) \cong (X,B)\times_R K
		\]
		over $K$ such that 
		$(X,B+X_{0,\red})$ is lc and $K_{X}+B +X_{0,\red} \sim_{\bQ} 0$.
	\end{defn}

	\begin{prop}\label{lem:semistable-reduction}
		If $(X_K,B_K)$ is an lc CY pair over $K$, then the following hold.
		\begin{enumerate}
			\item There exists an lc  model $(X,B)$ of $(X_K,B_K)$.
			\item If $(X,B)$ and $(X',B')$ are  two lc  models of $(X_K,B_K)$, then $(X,B+X_{0,\red})$ and $(X',B'+X'_{0,\red})$ are crepant birational. 
		\end{enumerate}
		Furthermore, if $L_K$ is an ample line bundle on $X_K$, then there exists an lc model $(X,B)$ such that  $L_K^{\otimes m}$ extends to a relatively ample line bundle $L$ on $X$ for some integer $m> 0$.
	\end{prop}

	\begin{proof}
		Both statements are well known to experts. See for example \cite[Theorem 2.3.6]{NX16} for similar statements in a  less general setting. 
		
		We now prove (1). Since $X_K$ is projective, there exists an ample line bundle $L_K$ on $X_K$. 
		Since $(X_K,B_K)$ is lc,  Bertini's theorem  implies that there exists  an effective $\bQ$-divisor $H_K$ on $X_K$ such that $(X_K,B_K+H_K)$ is lc and $H_K$ is in the $\bQ$-linear equivalence class of $L_K$.
		For $\varepsilon \in (0,1] \cap \bQ$, $(X_K,B_K+\varepsilon H_K)$ is a KSBA stable pair, and we write $\cM_\varepsilon$ for the irreducible component of the moduli stack of KSBA stable pairs containing the pair as a $K$-point.
		By \cite{Kol23}, $\cM_{\varepsilon}$ is a proper Deligne-Mumford stack and admits a  coarse moduli space. 
		Thus \cite[Theorem 1.1]{BIS25} implies that there exists an integer $d>0$ such that the corresponding morphism
		$\phi_K:\Spec(K) \to \cM_{\varepsilon}$ extends to a morphism 
		\[
		\phi_{\tR} : [\Spec(\tR) / \mu_d ] \to \cM_{\varepsilon} 
		,\]
		where $\tR:=R[\pi^{1/d}] $ and $\tK= K[\pi^{1/d}]$.
		The latter morphism corresponds to a $\mu_d$-equivariant family of KSBA pairs $(\tX,\tB+\varepsilon\tH) \to \Spec\,\tR$ that is an extension of $(X_K,B_K+\varepsilon H_K)\times_K \tK$. 
		Furthermore, \cite{KX20} implies that if $0<\varepsilon \ll1$, then $K_{\tX} +\tB +\tX_0 \sim_{\bQ}0$.
		Since $(\tX,\tB+\varepsilon\tH) \to \Spec(\tR)$ is a family of slc pairs,  $(\tX, \tB+ \varepsilon \tH + \tX_0)$ is slc  by Proposition \ref{p:kollarsmoothslcfam}.
		Since the generic fiber is normal and the latter pair is slc, $\tX$ is normal in codimension one.
		As $\tX$ is deminormal, this implies that  $\tX$ is normal.
		Thus $(\tX,\tB+\tX_0)$ is a projective lc CY pair over $\tR$. 
		Thus its $\mu_d$-quotient $(X,B+X_{0,\red})$ is a projective lc CY pair over $R$ that is an extension of $(X_K,B_K)$.
		Furthermore, $H:= \tH/\mu_d$ is an extension of $H_K$ to a $\bQ$-Cartier relatively ample $\bQ$-divisor. Thus there exists an integer $m>0$ such that $mH$ is a Cartier divisor and $L_K^{\otimes m} \cong \cO_{X_K}(mH_K)$. 
		Now $\cO_{X}(mH)$ is an extension of  $L_K^{\otimes m}$ to a relatively ample line bundle on $X$.

		For statement (2), let $Z$ denote the normalization of the graph of $X\dashrightarrow X'$ with morphisms $f: Z\to X$ and $f': Z \to X'$. 
		Let $G$ and $G'$ be the $\bQ$-divisors on $Z$ defined by 
		\[
		K_{Z}+ f_*^{-1}(B)+G + Z_{0,\red}=
		f^*(K_{X}+B + X_{0,\red})
		\]
		and 
		\[
		K_{Z}+ f'^{-1}_*(B')+G' + Z_{0,\red}=
		f'^*(K_{X'}+B' + X'_{0,\red}).
		\]
		Note that $G\leq 0$ and $G'\leq 0$ by the  definition of an lc model. 
		In addition, since $X_K\dashrightarrow X'_K$ is an isomorphism, $Z_K \to X_K $ and $Z_K\to X'_K$ are both isomorphisms.
		Thus $G$ and $G'$ are both vertical.
		We now compute that
		\[
		G-G'
		=
		f^*(K_{X}+B + X_{0,\red})
		-f'^*(K_{X'}+B' + X'_{0,\red})
		\sim_{\bQ} 0 
		\]
		and $f_*(G-G')  = -f_* G' \geq 0$.
		Therefore the negativity lemma \cite[Lemma 5.15]{LM22} implies that $G-G'\geq 0$. By symmetry $G'-G \geq  0$ as well. 
		Thus $G =G'$, which implies that the pairs are crepant birational.
	\end{proof}

	\begin{prop}\label{p:dltmodexist}
		If $(X,B)$ is an lc model of a projective lc CY pair $(X_K,B_K)$ over $K$, then there exists a $\bQ$-factorial dlt modification
		\[
		(Y,B_Y+Y_{0,\red})\to (X,B+X_{0,\red})
		.\]
		Furthermore, if $E_1,\ldots, E_r$ are prime divisors over $X$ such that $A_{X,B+X_{0,\red}}(\ord_{E_i})=0$ for all $1\leq i \leq r$, then $Y\to X$ can be chosen so that the $E_i$ are prime divisors on $Y$.
	\end{prop}

	\begin{proof}
		When $R$ is essentially of finite type over a field, this is a consequence of  \cite[Theorem 4.1]{Fuj11}, which uses the klt MMP in \cite{BCHM10}.  
		In the general case, the same argument works with the inputs from    \cite{BCHM10} replaced   by their generalizations in \cite{LM22}.
	\end{proof}
	
	\subsubsection{Essential skeleton}
	
	\begin{defn}[Essential Skeleton]\label{d:skeleton}
		If $(X_K,B_K)$ is a projective lc  CY pair over $K$, then we set 
		\[
		\Sk(X_K,B_K) = \{ v\in \Val_{X} \, \vert\,  A_{X,B+X_{0,\red}}(v)=0 \text{ and } v(\pi)=1\}
		,\]
		where $(X,B)$ is an lc model of $(X_K,B_K)$.
	\end{defn}
	
	\begin{prop}
		If $(X_K,B_K)$ is a projective lc  CY pair over $K$, then ${\rm Sk}(X_K,B_K)$ is independent of the  choice of lc model $(X,B)$. 
		Furthermore, if $(X_K,B_K)$ is klt, then
		\[
		{\rm Sk}(X_K,B_K) = \mathcal{D}(Y_{0})
		\]
		for any dlt modification $(Y,B_Y+Y_{0,\red})$ of $(X,B+X_{0,\red})$. 
	\end{prop}
	
	\begin{proof}
		Fix any two lc models $(X,B)$ and $(X',B')$ of $(X_K,B_K)$. 
		Let $Z$ be a common log resolution of $(X,B+X_{0,\red})$ and $(X',B'+X'_{0,\red})$ with morphisms $f:Z\to X$ and $f':Z\to X'$. 
		Let $B_Z$ be the $\bQ$-divisor on $Z$ defined by 
		\[
		K_{Z}+B_Z = f^*(K_X+B+X_{0,\red}) = f'^*(K_{X'}+B'+X'_{0,\red})
		,\]
		where the second equality holds by Proposition \ref{lem:semistable-reduction}.2.
		Now Lemma \ref{l:LCplace} applied twice implies that
		\[
		{\rm LC}(X,B+X_{0,\red}) = {\rm QM}(Z,(B_Z)^{=1}) = 
		{\rm LC}(X',B'+X'_{0,\red})
		.\]
		Thus ${\rm Sk}(X_K,B_K)$ is independent of the choice of lc model.
		Furthermore, if 
		\[
		(Y,B_Y+Y_{0,\red})\to (X,B+X_{0,\red})
		\]
		is a dlt modification, then Lemma \ref{l:LCplace} implies that 
		\[
		{\rm LC}(X,B+X_{0,\red})=
		{\rm QM}(Y,(B_Y+Y_{0,\red})^{=1}) .
		\]
		If $(X_K,B_K)$ is klt, then $B_Y$ has coefficients $<1$ since $B_Y$ is horizontal and $(Y_K,B_{Y_K})\to (X_K,B_K)$ is crepant. 
		Therefore  $(B_Y+Y_{0,\red})^{=1} = Y_{0,\red}$. Thus
		\[
		\Sk(X_K,B_K) =
		\QM(Y,Y_{0,\red}) \cap \{v(\pi)=1\} 
		= \mathcal{D}(Y_0),
		\]
		as desired.
	\end{proof}

	\begin{rem}[Relation to literature]
		The essential skeleton   was first introduced for smooth CY varieties over $\bC(\!(t)\!)$ in \cite{KS06}. The notion was then generalized to all smooth varieties over the fraction field of a complete DVR in \cite{MN15} and reinterpreted  using dlt models in \cite{NX16}. 
		In \cite{BM19}, the definition was generalized to the setting of dlt pairs.
		In the case when $R$ is complete, $B_K=0$, and $X_K$ is regular,  Definition \ref{d:skeleton} agrees with the original definition by \cite[Theorem 3.3.3]{NX16}.
	\end{rem}

	\subsubsection{Metrics}\label{sss:metrics}
	Let $(X_K,B_K)$ be a projective lc CY pair over $K$ and $L_K$ an ample line bundle on $X_K$. The following notion is needed to evaluate sections of $L_K$ along valuations.

	A \emph{metric} $(f,Y,M)$ for $L_K$ consists of a projective flat integral scheme $Y$ over $R$  with a proper birational morphism $f:Y_K \to X_K$ and a line bundle $M$ on $Y$ with a fixed isomorphism $M_K \cong f^*L_K$. 
	Two metrics  $(f,Y,M)$ and $(f',Y',M')$ of $L_K$ are \emph{equivalent} if there exists a projective flat integral scheme $Z$ over $R$ with proper birational morphisms $g:Z\to Y$ and $g':Z\to Y'$ such that $f\circ g_K = f'\circ g'_K$ and the induced isomorphism $(g^*M)_K  \cong (g'^*M')_K$ extends to an isomorphism $g^*M \cong g'^*M'$.

	Fix a metric $(f,Y,M)$ of $L_K$. 
	For any integer $m\geq0$, $0\neq s\in H^0(X_K,L_K^{\otimes m})$,  and $v\in \Sk(X_K,B_K)$, we define $v(s) \in \bR$ by viewing $v\in \Val_Y$ using the properness of $Y$ over $R$ and that $K(Y) \cong K(X_K)$ and then taking the evaluation  of $v$ along $f^*s$ viewed as a rational section of $M^{\otimes m}$. 
	
	\subsection{Orbifold cones}
	To deduce results on  CY pairs polarized by an ample line bundle or $\bQ$-Cartier divisor from the boundary polarized CY pairs case, we will use the orbifold cone construction. See  \cite[Section 2.6]{ABB+} for more details. 
	
	Let $X$ be a normal scheme that is projective over a Noetherian ring $A$ (we are primarily interested in the case when $A=\bk$ or is a DVR containing $\bk$) and $H$ a $\bQ$-Cartier $\bQ$-divisor on $X$. 
	We write
	\[
	f:X_H := \Spec_X \bigoplus_{i \in \bZ} \cO_{X}(\lfloor iH\rfloor )  \to X
	.\]
	The above $\bZ$-grading induces a $\bG_m$-action on $X_H$ that is equivariant with respect to the trivial action on $X$.
	If $H$ is Cartier, then $X_H\to X$ is a $\bG_m$-bundle.
	In general, $X_H\to X$ is a Seifert  $\bG_m$-bundle \cite[Definition 9.50]{Kol13} by \cite[Theorem 9.51]{Kol13}.
	This in particular implies that, for any $x\in X$, the reduced fiber
	$f^{-1}(x)_{\red}$ is  equivariantly isomorphic
	to the quotient $\bG_m / \mu_{m(x)}$ for some positive integer $m(x)$.
	
	If $H$ is relatively ample, we define the \emph{affine and projective orbifold cones} of $H$ as  
	\[
	C_a(X,H) : = \Spec \bigoplus_{i \in \bN} H^0(X,\lfloor iH\rfloor ) \quad \text{ and } C_p(X,H) := \Proj \bigoplus_{m \in \bN} \bigoplus_{i =0}^m H^0(X,\lfloor iH\rfloor ) x^{m-i}
	.\]
	Both schemes admit $\bG_m$-actions induced by the $\bZ$-grading by $i$.
	Note that 
	$V(x)$ is an ample divisor on $C_p(X,H)$. 
	As well, there  are equivariant  open inclusions 
	\[
	X_H \hookrightarrow C_a(X,H) \hookrightarrow C_p(X,H)
	\]
	under which $C_p(X,H) \setminus C_a(X,H)= V( x)$
	and $C_a(X,H) \setminus X_H = V\big( \bigoplus_{i>0} H^0(X,\lfloor iH\rfloor )\big)$.
	
	\section{Valuative independence for CY pairs}\label{s:affineCYpairs}
	
	In this section, we prove the existence of valuatively independent bases for affine CY pairs (Theorem \ref{thm:val-indepaffineCY}) and for ample line bundles on CY pairs (Theorem \ref{thm:val-indlb}). As the proof is  simpler than in the CY degeneration case, this section will provide a roadmap for the remainder of the paper.
	\medskip
	
	Throughout, all CY pairs are projective over $\bk$ unless specified otherwise.

	\subsection{Filtrations}\label{ss:vsfilts}
	We begin by discussing filtrations of vector spaces and diagonalizing bases.
	
	\medskip
	
	Throughout, we fix a vector space $V$ over a field $\bk$. We do not assume that $V$ is necessarily  finite-dimensional.
	
	\subsubsection{Definition}
	
	\begin{defn}\label{d:filtrationvs}
		An $\bR$-filtration $F$ of $V$ consists of vector subspaces $F^\la V\subset V$ for each $\la \in \bR$ satisfying 
		\begin{enumerate}
			\item $F^\la V \subset F^\mu V$ for $\la \geq \mu$,
			\item $F^{\la} V = \bigcap_{\mu< \la} F^\mu V$, 
			\item $\bigcup_{\la \in \bR} F^\la  V = V$, and 
			\item $\bigcap_{\la \in \bR} F^\la V = 0$.
		\end{enumerate}
		A \emph{$\bZ$-filtration} of $V$ is an $\bR$-filtration $F$ of $V$ such that $F^{\la} V = F^{\lceil \la\rceil} V$ for all $\la \in \bR$.
	\end{defn}
	
	\begin{rem}
		Given a filtration $F$ of $V$ and an element $0\neq s\in V$, we set 
		\[
		\ord_{F}(s):= \max\{ \la \,\vert\, s\in F^\la V\}
		\]
		and $\ord_{F}(0)=+\infty$.
		It is also convenient to define 
		\[
		\| s\| = {\rm exp}(- \ord_{F}(s))
		.\]
		This norm-like  function satisfies 
		\begin{enumerate}
			\item[(1)]  $\| s +s' \| \leq \max\left\{ \| s\|  , \|s'\| 
			\right\}$, 
			\item[(2)] $\| as\|  = \| s\|$ , and
			\item[(3)] $\|s\| =0$ if and only if $s=0$
		\end{enumerate}
		for all $s, s'\in V$ and $0\neq a\in \bk  $. 
	\end{rem}
	
	\begin{exa}
		We are primarily interested in the following example of a filtration.
		Let $(X,B)$ be a projective lc CY pair over $\bk$,  and let $V\subset K(X)$ be a $\bk$-vector subspace.
		A valuation $v\in \Val_X$ induces a filtration $F_v$ of $V$ defined by 
		\[
		F_v^\la V:= \{s \in V\, \vert\, v(s)\geq \la \}.
		\]
		
	\end{exa}

	\subsubsection{Diagonalization}
	
	\begin{defn}\label{d:diagtriv}
		Let $F$ be an $\bR$-filtration of $V$. A basis $(s_i)_{i \in I}$  of $V$ \emph{diagonalizes} $F$ if 
		\[
		F^\la V = {\rm span} ( s_i \, \vert\, \ord_{F}(s_i) \geq \la) 
		\]
		for all $\la \in \bR$.
	\end{defn}
	
	\begin{rem}\label{r:diagfiltsequiv}
		It is straightforward to see that  the condition in  Definition \ref{d:diagtriv} is equivalent to each of the following statements:
		\begin{enumerate}
			\item For each $\la \in \bR$, 
			$F^\la V$ is a span of some elements in $(s_i)_{i \in I}$.
			\item For each finite sum $\sum_{i \in I} a_i s_i$ with $a_i \in \bk$, 
			\[
			\ord_{F}\big( \textstyle \sum_{i\in I} a_i s_i \big)  = \min \{ \ord_{F}(s_i) \, \vert\, a_i \neq 0 \}
			.\]
		\end{enumerate}
	\end{rem}
	
	\begin{rem}\label{r:flag}
		For an $\bR$-filtration $F$ of $V$, set 
		\[
		{\rm gr}_{F}^\la V := F^\la V/ F^{>\la}V,\quad \text{ where }\quad 
		F^{>\la}V := \cup_{\mu >\la} F^{\mu}V 
		\]
		for each $\la \in \bR$.
		If $V$ is finite dimensional, then
		\[
		C:=\{ \la \in \bR \, \vert\, {\rm gr}_F^\la V \neq 0 \} \subset \bR
		\]
		is finite.
		If we write $C:= \{ \la_1,\ldots, \la_r\}$  with $ \la_{i} < \la_{i+1}$, then we can construct a basis of $V$ that diagonalizes $F$ by choosing  a basis compatible with the flag of vector spaces 
		$ F^{\la_1} V  \supset \cdots \supset F^{\la_r} V$.
	\end{rem}

	\subsubsection{Multigraded modules}
	
	\begin{defn}\label{d:assocgrtriv}
		Let $F=(F_j)_{j \in J}$ be a collection of $\bR$-filtrations of $V$. 
		\begin{enumerate}
			\item For $\la =(\la_j)_{j\in J}  \in \bR^J$, we set 
			\[
			F^{\la} V =  \bigcap_{j \in J} F^{\la_j}_j V
			\quad \text{ and } \quad 
			F^{>\la }V= \sum_{\mu> \la} F^{\mu } V
			.\]
			For $\la,\mu\in \bR^J$, we write $\mu \geq  \la$ if $\mu_j \geq \la_j$ for all $j\in J$. We write $\mu>\la$ if $\mu\geq \la$ and $\mu\neq \la$.
			
			\item For $0\neq s \in V$, we set 
			\[
			\ord_F(s) = (\ord_{F_j}(s))_{j\in J} \in \bR^J
			.\]
			
			\item 
			The \emph{associated multigraded vector space} of $(F_j)_{j\in J}$ is 
			\[
			{\rm gr}_{F} (V)
			:= 
			\bigoplus_{\la \in \bR^J} {\rm gr}_{F}^\la (V)
			:= 
			\bigoplus_{\la \in \bR^J} 
			F^{\la }V/ F^{>\la} V
			.\]
			We will often consider the case when $J$ is the finite set $\{1,2,\ldots, r \}$. 
			In this case, we will often write ${\rm gr}_{F_1,\ldots,F_r}(V)$  for ${\rm gr}_{F}(V)$. 
		\end{enumerate}
	\end{defn}

	\begin{defn}\label{d:multreestriv}
		Let $F_1,\ldots,F_r$ be a collection of $\bZ$-filtrations of $V$.
		The \emph{multigraded Rees module} of  $F_1,\ldots, F_r$  is 
		\[
		{\rm Rees}(F_1,\ldots, F_r;V):=
		\bigoplus_{\la \in \bZ^r} F^{\la} V x^{-\la}
		\subset V[x_1^{\pm1},\ldots,x_r^{\pm1}]
		,\]
		where $x^{-\la}:=x_1^{-\la_1}\cdots x_r^{-\la_r}$.
		Note that ${\rm Rees}(F_1,\ldots, F_r;V)$  has the structure of a $\bZ^r$-graded $\bk[x_1,\ldots, x_r]$-module, where $\deg(x_i) := e_i \in \bZ^r$. 
	\end{defn}

	\begin{lemma}
		\label{l:Reestquotient=graded}
		If $F_1,\ldots, F_r$ are $\bZ$-filtrations of $V$, then there is an isomorphism of  $\bk$-vector spaces
		\[
		{\rm Rees}(F_1,\ldots, F_r;V)/ (x_1,\ldots, x_r){\rm Rees}(F_1,\ldots, F_r;V) \overset{\cong}{\longrightarrow} {\rm gr}_{F_1,\ldots, F_r} V
		\]
		that sends the class of  $sx^{-\la}$ in degree $-\la $ to the class of  $s$ in degree $\la$.
	\end{lemma}
	
	\begin{proof}
		We compute that
		\[
		(x_1,\ldots, x_r) {\rm Rees}(F_1,\ldots, F_r;V) = \bigoplus_{\la \in \bZ^r} \sum_{i=1}^r F^{\la+e_i} V x^{-\la}
		=
		\bigoplus_{\la \in \bZ^r} F^{>\la} V x^{-\la},
		\]
		where $e_i \in \bZ^r$ is the $i$-th standard basis vector. 
		Above, the second equality uses the assumption  that $F_1,\ldots, F_r$ are $\bZ$-filtrations.
		Therefore the desired map is well defined and is an isomorphism.
	\end{proof}
	
	\subsubsection{Simultaneous diagonalization}
	
	\begin{defn}
		A collection of $\bR$-filtrations $(F_j)_{j\in J}$ of $V$ is \emph{simultaneously diagonalizable} if there exists a single basis $(s_i)_{i \in I}$ of $V$ that diagonalizes $F_j$ for all $j\in J$.
	\end{defn}
	
	While a simultaneously diagonalizing basis is often non-unique (even up to scaling),  the values of $\ord_{F}$ for the basis elements  are  unique. 
	
	\begin{lem}\label{l:diagbasisords}
		Let $F=(F_j)_{j\in J}$ be a collection of filtrations of $V$ that is simultaneously  diagonalized by a basis $(s_i)_{i \in I}$ of $V$.
		For any  $\la \in \bR^J$, the following holds 
		\[
		{\rm card} \{i \in I \vert\,  \la= \ord_{F}(s_i) \in \bR^J \}  = \dim {\rm gr}_F^\la (V).
		\]
		In particular, the values $\ord_{F}(s_i)$ are independent of the choice of simultaneously diagonalizing basis.
	\end{lem}
	
	\begin{proof}
		Since
		\[
		F^{\la} V= {\rm span}( s_i \vert\, \ord_{F}(s_i) \geq \la )
		\quad \text{ and } \quad 
		F^{>\la} V= {\rm span}( s_i \vert\, \ord_{F}(s_i) > \la )
		,\]
		the collection 
		$(\overline{s}_i \, \vert\, \la = \ord_{F}(s_i) )$ is a basis for ${\rm gr}^\la_{F}(V)$, where
		$\overline{s}_i \in {\rm gr}_{F}^{\la^i}(V)$
		and
		$\la^i := \ord_{F}(s_i) \in \bR^J$.
		Therefore $\dim {\rm gr}_{F}^\la(V)$ is the cardinality of the set 
		$\{ i \in I \, \vert\, \la = \ord_{F}(s_i) \}$.
	\end{proof}

	Under additional assumptions, we characterize simultaneous diagonalizability in terms of the associated graded vector space and the Rees module.

	\begin{prop}\label{p:Reesflat=diagtriv}
		If  $N:=\dim V<\infty$ and $F_1,\ldots, F_r$ are $\bZ$-filtrations of $V$, then the following conditions are equivalent. 
		\begin{enumerate}  
			\item The multigraded Rees module ${\rm Rees}(F_1,\ldots, F_r;V)$ is flat as a $\bk[x_1,\ldots,x_r]$-module.
			\item $\dim {\gr}_{F_1,\ldots, F_r}(V) = \dim V$. 
			\item The filtrations $F_1,\ldots, F_r$ are simultaneously diagonalizable.
		\end{enumerate}
		Furthermore, a collection of elements $(s_i)_{i=1}^N$ of $V$
		is a basis for $V$ that simultaneously diagonalizes $F_1,\ldots, F_r$ if and only if the elements $\overline{s_i} \in {\rm gr}_{F_1,\ldots, F_r}^{\la^i}(V)$ form a basis for ${\rm gr}_{F_1,\ldots, F_r}(V)$, where $\la^i =  \ord_{F}(s_i)\in \bR^r$. 
	\end{prop}
	
	\begin{proof}
		Since  $V$ is finite dimensional, 
		\[
		V = F_i^{-\la} V \supset F_i^\la V=0
		\]
		for $\la \gg 0$ and each $1\leq i \leq r $. 
		Thus $M:={\rm Rees}(F_1,\ldots, F_r;V)$ is a finite $\bk[x_1,\ldots, x_r]$-module. 
		By Lemma \ref{l:Reestquotient=graded}
		we have $M/ (x_1,\ldots, x_r) M \cong {\rm gr}_{F_1,\ldots, F_r}(V)$, which we will repeatedly use below.
		
		We now show that $(1)\implies (2)$ holds.
		Observe that the localization satisfies
		\[
		M_{x_1\cdots x_r}\cong V[x_1^{\pm1},\ldots, x_r^{\pm1}]
		\]
		and, in particular, is a rank $\dim V$ free module over $\bk[x_{1},\ldots, x_r]_{x_1\cdots x_r} \cong \bk[x_1^{\pm1},\ldots, x_r^{\pm1}]$. 
		If $M$ is flat, this implies that $M/(x_1,\ldots, x_r) M$
		is a $\dim V$-dimensional $\bk$-vector space. 
		Thus $\dim {\rm gr}_{F_1,\ldots, F_r}(V) =\dim V$.
		
		Next, we prove $(2)\implies (3)$, as well as the reverse implication of the final statement. 
		Assume that $N:=\dim {\rm gr}_{F_1,\ldots, F_r}(V) = \dim V$. 
		Thus there exist elements $s_1,\ldots, s_N \in V$ such that $\overline{s}_1,\ldots,  \overline{s}_N$ form a basis for $\gr_{F_1,\ldots, F_r}(V)$.
		Consider the homogeneous elements 
		\[
		s_1x^{-\la^1},\ldots, s_N x^{-\la^N} \in M,
		\]
		where $\la^i := \ord_{F}(s_i)$. 
		By our assumption, these elements generate $M/ (x_1,\ldots, x_r) M$ as a $\bk$-vector space. 
		By the graded version of Nakayama's Lemma \cite[\href{https://stacks.math.columbia.edu/tag/0EKB}{Lemma 0EKB}]{stacks} 
		with respect to the $\bN$-grading on $\bk[x_1,\ldots, x_r]$ with $\deg(x_i) = 1$ and the $\bZ$-grading on $M$ with $\deg (s x^{-\la} ) = -\sum_{j=1}^r \la_j$,
		these elements generate $M$ as an $\bk[x_1,\ldots, x_r]$-module. 
		Thus 
		\[
		F^\la V= {\rm span} (s_i \, \vert \, \la^i \geq \la )
		\]
		for all $\la \in \bZ^r$. 
		In particular  $s_1,\ldots, s_N$ must span $V$. 
		As $N = \dim V$, this implies that $(s_1,\ldots, s_N)$ is a basis for $V$.
		Thus this basis diagonalizes each of $F_1,\ldots, F_r$ by Remark \ref{r:diagfiltsequiv}.1.
		
		Finally, we verify $(3)\implies (1)$ and the forward implication of the last statement. 
		If there exists a basis $s_1,\ldots, s_N$ of $V$ that diagonalizes  $F_1,\ldots, F_r$, then 
		\[
		F^\la V= {\rm span} (s_i \, \vert\,  \la^i\geq \la)
		,
		\]
		where $\la^i:= \ord_{F}(s_i)$. 
		Therefore $s_1 x^{-\la^1},\ldots, s_N x^{-\la^N}$ form a free $\bk[x_1,\ldots, x_r]$-module basis for $M$. 
		In particular $M$ is a flat $\bk[x_1,\ldots, x_r]$-module and ${\rm gr}_{F_1,\ldots, F_r}(V)$ is spanned by $\overline{s}_1,\ldots, \overline{s}_N$.
	\end{proof}

	The next proposition allows us to reduce diagonalizing a collection of filtrations to proving simultaneous diagonalizability for every finite subcollection.

	\begin{prop}\label{p:vsfinitediminfinitecol}
		Assume that $V$ has finite dimension and $(F_j)_{j\in J}$ is a collection of $\bR$-filtrations of $V$.
		If  the filtrations $(F_j)_{j\in I}$ are simultaneously diagonalizable for every finite subset $I\subset J$, then the following hold:
		\begin{enumerate}
			\item The filtrations $(F_j)_{j\in J}$ are simultaneously diagonalizable. 
			\item There exists a finite subset $J^* \subset J$ such that if a basis $(s_1,\ldots, s_N)$ of $V$ diagonalizes $F_j$ for all $j\in J^*$, then the basis diagonalizes $F_j$ for all $j\in J$. 
		\end{enumerate}
	\end{prop}
	
	\begin{proof}
		Let $N:= \dim V < \infty$ and $\bL(V)$ denote the set of all subspaces of $V$. 
		For any subset $I\subset J$, consider the subset 
		\[
		\bF_I:= \{ W\subset V \, \vert\, W = F_j^\la V \text{ for some } j\in I \text{ and } \la \in \bR \}
		\subset \bL(V),
		\]
		which is the set of subspaces occurring in the filtrations $(F_j)_{j\in I}$.
		
		We now show that $\bF_J$ is a finite set.
		Fix a  finite subset $\bF \subset \bF_J$.
		By the finiteness, there exists a finite set $I\subset J$ such that $\bF \subset \bF_{I}$. 
		Since $(F_j)_{j \in I}$ is simultaneously diagonalized by some  basis $(s_1,\ldots, s_N)$, each element contained in $\bF_I$ is a span of some subset of the basis elements. Thus 
		\[
		|\bF|\leq | \bF_I| \leq 2^N
		.\]
		Therefore every finite subset of $\bF_J$ has at most $2^N$ elements, and so $|\bF_J| \leq 2^N<\infty$.
		
		Since $\bF_J$ is finite, the subspaces in this set occur in finitely many filtrations.
		Thus we can choose a finite subset $J^*\subset J $ such that 
		$\bF_{J^*}= \bF_J$.
		Now choose any basis $(s_i)_{i=1}^N$ that diagonalizes $F_j$ for each $j\in J^*$ (such a basis exists as $J^*$ is finite).
		Thus every element of $\bF_{J^*}$ is spanned by some subset of the basis elements. 
		Since $\bF_{J^*}= \bF_{J}$, every element of $\bF_J$ is also spanned by some  subset of the basis elements. 
		Therefore $(s_i)_{i=1}^N$ diagonalizes $F_j$ for each $j\in J$ using Remark \ref{r:diagfiltsequiv}.1.
	\end{proof}

	\subsection{Test configurations}
	
	We now discuss test configurations of boundary polarized CY pairs.
	These objects provide a geometric way to understand filtrations induced by  lc places.

	\begin{defn}
		A \emph{test configuration} of a boundary polarized CY pair $(X,B+D)$ is the data of
		\begin{enumerate}
			\item a  family of boundary polarized CY pairs $(\cX,\cB+\cD) \to \bA^1$, 
			\item a $\bG_m$-action on $(\cX,\cB+\cD)$ extending the standard $\bG_m$-action on $\bA^1$, and 
			\item an isomorphism $(\cX_1,\cB_1+\cD_1) \cong (X,B+D)$. 
		\end{enumerate}
	\end{defn}
	
	Test configurations of boundary polarized CY pairs are studied in detail in \cite[Section 4]{ABB+}.
	
	\begin{rem}
		Let $(\cX,\cB+\cD)$ be a test configuration of a boundary polarized CY pair $(X,B+D)$. 
		\begin{enumerate}
			\item The data of the $\bG_m$-action and  the isomorphism in the definition  induce a natural $\bG_m$-equivariant isomorphism 
			\[
			(\cX,\cB+\cD)_{\bA^1 \setminus 0} \cong (X,B+D) \times (\bA^1\setminus 0)
			,\]
			where $\bG_m$ acts on the right-hand side as the product of the trivial action and the standard action. 
			
			\item If $(X,B+D)$ is lc, then $(\cX,\cB+\cD+\cX_0)$ is  lc.
			Indeed, Proposition \ref{p:kollarsmoothslcfam} implies that $(\cX,\cB+\cD+\cX_0)$ is slc.
			Since $X$ is normal and 
			\[
			\cX_{\bA^1\setminus 0} \cong X\times (\bA^1\setminus 0),
			\]  $\cX_{\bA^1\setminus 0}$ is also  normal.
			Thus  Proposition \ref{p:kollarsmoothslcfamgen} implies that $\cX$ is normal, and so  the pair $(\cX,\cB+\cD+\cX_0)$ is lc. 
			
			\item If $(X,B+D)$ is lc and $\cX_0$ is irreducible,  then $\cX$ is normal by (2) and so 
			\[
			\ord_{\cX_0}:K(\cX)^\times  \to \bZ
			\]
			is a well-defined valuation. 
			We write 
			$v_{\cX_0} : K(X)^\times  \to \bZ$ for the valuation induced by restriction via the inclusion $K(X) \hookrightarrow K(X)(t) \cong K(\cX)$. 
		\end{enumerate}
	\end{rem}
	
	\begin{prop}\label{p:dival->tc}
		If $(X,B+D)$ is an lc boundary polarized CY pair and $v:K(X)^\times \to \bZ$ is a divisorial valuation   such that $A_{X,B+D}(v)=0$,  then there exists a test configuration $(\cX,\cB+\cD)$ of $(X,B+D)$ such that $\cX_0$ is irreducible and $v_{\cX_0} = v$. 
	\end{prop}
	
	\begin{proof}
		See \cite[Theorem 4.8]{ABB+}.
	\end{proof}

	\begin{defn}\label{d:testconfigfilt}
		Let $(\cX,\cB+\cD)$ be a test configuration of an lc boundary polarized CY pair $(X,B+D)$. 
		For each non-negative integer $m$ such that $mD$ is a $\bZ$-divisor, there is a $\bZ$-filtration $F$ of $H^0(X,mD)$ defined by 
		\[
		F^\la H^0(X,mD) = \{ s\in H^0(X,mD) \, \vert\, \overline{s} x^{-\lceil \la\rceil} \in H^0(\cX,m\cD)  \},
		\]
		where $\overline{s}$ is the unique $\bG_m$-invariant element of $H^0(\cX\setminus \cX_0 , m\cD)$ such that $\overline{s}\vert_{\cX_1} = s$, and $x$ is the parameter for $\bA^1$. 
	\end{defn}
	
	\begin{lem}\label{l:filttc}
		If $(\cX,\cB+\cD)$ is a test configuration of an lc boundary polarized CY pair $(X,B+D)$ and $\cX_0$ is irreducible, then
		\[
		F^\la H^0(X,mD) = \{ s\in H^0(X,mD) \, \vert\, v_{\cX_0}(s) \geq \la \}.
		\]
	\end{lem}
	
	Above, we view $H^0(X,mD)\subset K(X)$, and $v_{\cX_0}(s)$ denotes the value of the rational function $s$  along the valuation.

	\begin{proof} 
		Fix $s\in H^0(X,mD)$.
		Since $\overline{s} \in H^0(\cX\setminus \cX_0, m\cD)$,  $s\in F^\la H^0(X,mD)$ if and only if $\ord_{\cX_0} ( \overline{s} x^{-\lceil \la\rceil }) \geq 0$. 
		Since 
		\[
		\ord_{\cX_0} (\overline{s} x^{-\lceil \la\rceil }) = \ord_{\cX_0} (\overline{s}) - \lceil \la\rceil =  v_{\cX_0}  (s) -\lceil \la\rceil 
		\]
		and $v_{\cX_0}(s)$ is an integer,  it follows  that $s\in F^\la H^0(X,mD)$ if and only if $v_{\cX_0}(s) \geq \la$.  
	\end{proof}
	
	\begin{rem}
		The above filtration differs from the filtration in \cite[Proposition 4.9]{ABB+} by a shift, because the cited reference considers filtrations of $H^0(X,-m(K_{X}+B))$ with its natural linearization.
	\end{rem}

	\subsection{Higher rank degenerations}\label{ss:higherkdeg}
	We now construct higher rank degenerations of a boundary polarized CY pair that relate finite collections of test configurations via a family over $\bA^r$. 
	A related construction appears in \cite{Xu21} using a different argument that is inductive. 
	
	Below, let $\bA^r := \bA^r_{x_1,\ldots, x_r}:= \Spec \bk[x_1,\ldots, x_r]$ denote affine $r$-space with open sets 
	\[
	U:= D(x_1\cdots x_r) \cong \bG_m^{r}
	\quad
	\text{ and } 
	\quad
	U_i:=  D(x_1\cdots \widehat{x}_i \cdots x_r) \cong \bA_{x_i}^1\times \bG_m^{r-1}
	\]
	for $1\leq i \leq r$.
	
	\begin{prop}\label{p:highranktheta}
		If $(X,B+D)$ is an lc boundary polarized CY pair and 
		\[
		(\cX^1, \cB^1+\cD^1), \ldots, (\cX^r, \cB^r+\cD^r)
		\]
		are test configurations of $(X, B+D)$, 
		then there exists a $\bT:= \bG_m^r$-equivariant family of boundary polarized CY pairs 
		$(\fX , \fB+ \fD) \to \bA^r$
		such that
		\begin{enumerate}
			\item  there is a $\bT$-equivariant isomorphism 
			$(\fX,\fB+\fD)_U \cong (X,B+D) \times U$ over $U$ and
			\item  the  composition of  birational maps
			\[
			(\fX , \fB+ \fD)_{U_i}\dashrightarrow  (X,B+D) \times U_i \dashrightarrow (\cX^i, \cB^i+\cD^i)\times \bG_m^{r-1}
			\]
			is an isomorphism for each $1\leq i \leq r$.
		\end{enumerate}
	\end{prop}
	
	In statement (2), the first birational map is induced by the birational map in (1) and the second is given by the product of $(X,B+D) \times \bA^1 \dashrightarrow (\cX^i,\cB^i+\cD^i)$ with $\bG_{m}^{r-1}$.

	\begin{proof}
		To simplify notation, let $S:=\bA^r$.
		Now consider the trivial family of boundary polarized CY pairs
		\[
		(X_{S},B_{S}+D_{S}) := (X,B+D) \times S
		\to S
		.
		\]
		Note that over $U:= D(x_1\cdots x_r) \subset S$, the family has the desired form.
		We will construct $\fX$ via a sequence of birational maps appearing in the following diagram:
		\[
		\begin{tikzcd}
			&Y\arrow[ld,"f",swap]\arrow[r,dashed,"g"]& Z\arrow[rd,dashed,"h"]&\\
			X_S&&& \fX
		\end{tikzcd}
		.\]

		\emph{Step 1: We construct the morphism $f$}.
		Let $\cX_{01}^{i},\ldots, \cX^{i}_{0n_i}$ denote the irreducible components of $\cX_0^{i}$. 
		These are prime divisors on $\cX^i$. Let 
		$E_{ij}:= \cX_{0j}^{i} \times (\bG_m^{r-1})$, which is a prime divisor on $\cX^i \times \bG_m^{r-1}$.
		Using the birational maps
		\[
		\cX^i \times \bG_m^{r-1} \dashrightarrow 
		(X\times \bA^1_{x_i}) \times \bG_m^{r-1} \cong X\times U_i \hookrightarrow X_S
		,\]
		we may view each $E_{ij}$ as a prime divisor over $X_{S}$.
		Note that the center of $E_{ij}$ on $X_S$  dominates $\{x_i =0 \} \subset S$. 
		Observe that 
		\[
		A_{X_S,B_S+D_S+ V_{X_S}(x_1\cdots x_r)}(E_{ij})
		=
		A_{X\times \bA^1, B\times \bA^1+ D\times \bA^1+ X\times 0}(\cX_{0j}^i) 
		=
		A_{\cX^i,\cB^i+\cD^i+\cX^i_0}(\cX_{0j}^i)
		=
		0
		,\]
		where the second to last equality holds by \cite[Lemma 2.10]{ABB+}.
		
		Since $(X_S,B_S+D_S)\to S$ is a family of slc pairs, Proposition \ref{p:kollarsmoothslcfam} implies that 
		\[
		(X_S,B_S+D_S+V(x_1\cdots x_r))
		\]
		is lc and the  lc centers of $(X_S,B_S+D_S)$ dominate $S$. 
		Now  \cite[Corollary 1.38]{Kol13} implies that there exists a $\bQ$-factorial dlt modification 
		\[
		f:(Y,B_Y+D_Y+V_Y(x_1\cdots x_r)_{\red}) \to (X_S,B_S+D_S+V_{X_S}(x_1\cdots x_r))
		\]
		such that $B_Y = f_*^{-1}(B_S)$,  $D_Y$ is the unique $\bQ$-divisor such that the above  map is crepant, and each $E_{ij}$ is a prime divisor on $Y$. 
		Let $E:= \sum_{i=1}^r \sum_{j=1}^{n_i} E_{ij}$, which we view as a divisor on $Y$. 
		Note that 
		\begin{equation}\label{e:D_Yform}
			D_Y = f^*D_S  +G- \sum_{i,j} a_{ij}E_{ij} 
			,
		\end{equation}
		where $a_{ij}:=A_{X_S,B_S+V_{X_S}(x_1\cdots x_r)}(E_{ij})$ and $\Supp(G)$ is a union of $f$-exceptional divisors that either dominate $S$ or are contained in the support of $V_Y(x_1\cdots x_r)_{\red}-E$. 
		Since the $E_{ij}$ are $\bT$-invariant divisors over $X_S$, by choosing the log resolution in the proof of \cite[Corollary 1.38]{Kol13} to be $\bT$-equivariant and using the fact that $\bT$ is connected, the MMP in the proof is necessarily $\bT$-equivariant. 
		Thus we may assume that the $\bT$-action on $X_S$ extends to a $\bT$-action on $Y$. 
		\medskip
		
		\emph{Step 2: We now construct the birational contraction $g$}.
		Note that 
		\[
		(Y,B_Y+D_Y+V_{Y}(x_1\cdots x_r)_{\red})
		\]
		is dlt,
		\[
		K_{Y}+ B_Y+ D_Y + V_Y(x_1\cdots x_r)_{\red}\sim_{\bQ,S}0,\]
		and 
		\[
		K_{Y}+ B_Y+ D_Y + V_Y(x_1\cdots x_r)_{\red} -E
		\sim_{\bQ,S}-E 
		\sim_{\bQ,S} V_{Y}(x_1\ldots x_r)-E
		,\]
		which is effective.
		Thus \cite[Theorem 1.6]{HX13} implies that there exists a $-E$-MMP over $S$ that terminates with a good minimal model. 
		Write $g:Y\dashrightarrow Z$ for the birational contraction to the minimal model, $B_Z:= g_* B_Y$, $D_Z:= g_* D_Y$, and $E_Z := g_* E$. 
		Thus $-E_Z$ is nef over $S$. 
		Since 
		\[
		(Y,B_Y+D_Y+V_Y(x_1\cdots x_r)_{\red})
		\]
		is lc and CY over $S$, 
		\[
		(Z,B_Z+D_Z+V_Z(x_1\cdots x_r)_{\red})
		\]
		is lc and CY over $S$. 
		
		We claim that $g$ contracts precisely the prime divisors in the support of $V_{Y}(x_1\cdots x_r)-E$. 
		Since $-E\sim_{\bQ} V_{Y}(x_1\cdots x_r)-E$ and the latter divisor is effective,  $g$ only contracts divisors contained in $V_{Y}(x_1\cdots x_r)- E$.
		Since $-E_Z$ is nef over $S$ and $\im(\Supp(E_Z) \to S) = V_{S}(x_1\cdots x_r) $, all such divisors must get contracted. 
		Indeed, as $Z\to S$ has connected fibers, if $E_Z \subsetneq V_{Z}(x_1\ldots x_r)$, then there would exist a curve $C$ contracted by $Z\to S$ such that $C$ is not contained in $E_Z$, but intersects it and hence satisfies $C \cdot E_Z>0$.
		As a consequence 
		\[E_Z =V_{Z}(x_1\cdots x_r)_{\red}= V_Z(x_1\cdots x_r),
		\]
		where the second equality uses that $\ord_{E_{ij}}(x_1\cdots x_r) = 
		\ord_{\cX_{0j}^{i}}(t) = 1$.

		\emph{Step 3: We now construct a birational contraction $h$ to a canonical model.}
		Since $D$ is ample and $(X,B+D)$ is lc, there exists an effective $\bQ$-divisor $ H \sim_{\bQ} D $ such that $(X,B+D+H)$ is lc and has the same lc centers as  $(X,B+D)$. 
		Thus
		\[
		(X_S,B_S+D_S+H_S) \to S
		\]
		is a family of lc pairs, where $H_S := H\times S$. 
		Thus Proposition \ref{p:kollarsmoothslcfam} implies  that 
		\[
		(X_S,B_S+D_S+H_S + V_{X_S}(x_1\cdots x_r))
		\]
		is lc. In particular, $\Supp(H_S)$ contains no lc centers of 
		\[
		(X_S,B_S+D_S + V_{X_S}(x_1\cdots x_r)).
		\]
		Let $H_Y := f^*H_S$ and $H_Z:=g_* H_Y=g_* f^* H_S$. 
		Fix a rational number $a>0$ such that 
		\[
		A_Z:= H_Z + a E_Z- \textstyle \sum_{i,j} a_{ij} g_*(E_{ij})
		\]
		is effective. 
		Using \eqref{e:D_Yform} and that $E_Z = V_Z(x_1\cdots x_r)\sim0$, we compute that 
		\[
		A_{Z}\sim_{\bQ} 
		g_*(f^*D_S) + a E_Z- \textstyle \sum_{ij} a_{ij} g_*(E_{ij})
		\sim_{\bQ}
		D_{Z}-g_* G
		\]
		and
		\[
		(D_{Z}-g_*G)_U = g_* (f^*D_{S})\vert_U
		.\]
		
		We claim that, for $0<\varepsilon \ll1$, the pair
		\[
		(Z,B_Z+D_Z+\varepsilon A_Z)
		\]
		is (i) a $\bQ$-factorial dlt pair with no lc centers contained in $V_{Z}(x_1\cdots x_r)$ and (ii) its restriction to $U\subset S$ is a good minimal model. 
		For (i), note that 
		\[
		(Z,B_Z+D_Z+\varepsilon A_Z)_U \cong (Y,B_Y+D_Y+\varepsilon H_Y)_U,
		\]
		which is a dlt pair for $0<\varepsilon<1$, 
		since $(Y,B_Y+D_Y)_U$ is dlt and  $(Y,B_Y+D_Y+H_Y)_U$ is lc.
		Since $(Z,B_Z+D_Z+ E_Z)$ is lc and $\bQ$-factorial, 
		$(Z,B_Z+D_Z)$ is lc and has no lc centers contained in $E_Z$ or $\Supp(A_Z)$. 
		Therefore  (i) holds.
		For (ii), note that the birational map
		\[
		(Z,B_Z+D_Z+\varepsilon A_Z)_U \dashrightarrow 
		(X,B+D+\varepsilon H) \times U
		\]
		is a morphism, crepant, and $K_X+B+D+\varepsilon H$ is ample.
		
		Using the claim, \cite[Theorem 1.1]{HX13} now implies that 
		\[
		(Z,B_Z+D_Z+\varepsilon A_Z )
		\]
		admits a canonical model over $S$ for any such $0<\varepsilon \ll1 $.
		Let 
		\[
		h: Z \dashrightarrow \fX
		\]
		denote the birational contraction to the canonical model, $\fB:= h_* B_Z$, and $\fD:= h_*D_Z$.
		By the justification for (ii), $\fX \dashrightarrow X\times U$ is an isomorphism over $U$. As a consequence, $h_* g_* G =0$.
		Now we compute that
		\[
		K_Z+B_Z+D_Z +\varepsilon A_Z 
		\sim_{\bQ,S}
		\varepsilon A_{Z}
		\sim_{\bQ}
		\varepsilon(D_Z-g_* G)
		\]
		and that $h_*(D_Z-g_* G) = \fD$.
		Thus $\fX$ is the ample model of $D_Z-g_*G$ and  $\fD$ is ample over $S$.
		\medskip
		
		\emph{Step 4: We now verify that $(\fX,\fB+\fD) \to S$ is a $\bT$-equivariant family of boundary polarized CY pairs satisfying (1) and (2).}
		Since $\bT$ is a connected group, the MMP steps $Y \dashrightarrow Z$ and $Z\dashrightarrow \fX$ are $\bT$-equivariant. 
		Thus there exists a  $\bT$-action on $\fX $  extending the $\bT$-action on $S$ and fixing $\fB$ and $\fD$. 
		Next, since 
		\[
		(Z,B_Z+D_Z + \varepsilon A_Z)_U \dashrightarrow  (X,B+D+\varepsilon H) \times U
		\]
		is the morphism to the canonical model, the induced birational map 
		\[
		(\fX,\fB+\fD) \dashrightarrow (X_S ,B_S+D_S)
		\]
		is an isomorphism over $U$. This defines the isomorphism  in (1). 
		For (2),  consider the birational map
		\[
		h_i:Z_{U_i} \dashrightarrow \cX^i \times \bG_m^{r-1} 
		.\]
		Note that $h_i\vert_U$ is a morphism and $h_i$ is a birational contraction that only contracts divisors that intersect $Z_U$. 
		Additionally,
		\[
		(D_Z-g_*G)_{U_i} = h_i^* (\cD^i \times \bG_m^{r-1})
		,\]
		where $h_i^*$ denotes the birational pullback,
		as both $\bQ$-divisors agree with 
		$g_* (f^*D_{S})$ over $U$ and have no components of their support contained in $V_{Z}(x_i)$.
		Thus the relative section rings of $(D_Z-g_*G)_{U_i}$ and $\cD^i\times \bG_{m}^{r-1}$ over $U_i$ are isomorphic.
		Since $\cD^i\times \bG_{m}^{r-1}$ is ample over $U_i$, we conclude that the induced birational map
		\[
		\fX_{U_i} \dashrightarrow \cX^i\times \bG_m^{r-1}
		\]
		is an isomorphism, which proves (2).

		Finally, we verify that $(\fX,\fB+\fD)\to S$ is a family of boundary polarized CY pairs. 
		Since 
		\[
		(Z,B_Z+D_Z+ V_{Z}(x_1\cdots x_r))
		\]
		is lc and CY over $S$ and $Z\dashrightarrow \fX$ is a birational contraction, 
		\[
		(\fX,\fB+\fD+ V_{\fX}(x_1\cdots x_r))
		\]
		is lc and CY over $S$. 
		Thus $(\fX, \fB+\fD) \to S$ is a family of slc pairs in an open neighborhood of $0\in S$ by Proposition \ref{p:kollarsmoothslcfam}. 
		Let $V\subset S$ be the largest open set over which  $(\fX,\fB+\fD)\to S$ is a family of slc pairs. 
		Since $V$ is $\bT$-invariant, contains $0$, and contains $D_S(x_1\cdots x_r)$,  $V=S$. 
		Thus $(\fX,\fB+\fD) \to S$ is a family of slc pairs. 
		Since $K_{\fX}+\fB+\fD \sim_{\bQ,S}0$ and $\fD$ is  ample over $S$, it is a family of boundary polarized CY pairs. 
	\end{proof}
	
	We now analyze the relative section ring of the family constructed in Proposition \ref{p:highranktheta}.
	
	\begin{prop}\label{p:familyArRees}
		Keep the assumptions and notation from  Proposition \ref{p:highranktheta}, 
		and let  
		\[
		(\fX,\fB+\fD)\to \bA^r
		\]
		be a $\bT$-equivariant family of boundary polarized CY pairs satisfying the conclusion of the  proposition.  
		
		If $m$ is a non-negative integer such that $mD$ is a  $\bZ$-divisor, then the following hold:
		\begin{enumerate}
			\item The $\bk[x_1, \ldots, x_r]$-module $H^0(\fX, m\fD)$ is flat and the natural map 
			\[
			H^0(\fX,m\fD)\otimes k(s) \longrightarrow H^0(\fX_s, m\fD_s) 
			\]
			is an isomorphism for all $s\in \bA^r$. 
			\item  There is a commutative diagram of $\bZ^r$-graded $\bk[x_1,\ldots, x_r]$-modules 
			\[
			\begin{tikzcd}
				{\rm Rees}(F_1,\ldots,F_r; H^0(X,mD))\arrow[r,hook] \arrow[d,"\cong"]  & \ H^0(X,mD)[x_1^{\pm1},\ldots, x_r^{\pm 1}]\arrow[d,"\cong"] \\
				H^0(\fX,m\fD) \arrow[r,hook] & H^0(\fX_{D(x_1\cdots x_r)}, m\fD_{D(x_1\cdots x_r)})  
			\end{tikzcd}
			,\]
		\end{enumerate}
		where $F_i$ is the filtration of $H^0(X,mD)$ induced by $(\cX^i,\cB^i+\cD^i)$,  the horizontal maps are the natural inclusions, the vertical maps are isomorphisms, and the
		right vertical map is induced by pulling back via the isomorphism $\fX_{D(x_1 \cdots x_r)} \to X\times \bT$.
	\end{prop}
	
	\begin{proof}
		Since $(\fX,\fB+\fD) \to \bA^r$ is a family of slc pairs and $\fD$ is $\bQ$-Cartier, $\cO_{\fX}(m\fD)$ is flat over $\bA^r$ and the natural map
		\[\cO_{\fX}(m\fD)\vert_{\fX_s} \to \cO_{\fX_s}(m\fD_s)
		\]
		is an isomorphism for all $s\in \bA^r$  by \cite[Corollary 4.33]{Kol23}. 
		Thus, for all $i\geq 1$ and $s\in \bA^r$, 
		\[
		H^i(\fX_s, \cO_{\fX}(m\fD)\vert_{\fX_s} ) \cong H^i(\fX_s, \cO_{\fX_s}(m\fD_s) ) 
		= 
		0
		,\]
		where the last equality holds by the  slc version of Kodaira Vanishing in \cite[Theorem 11.34]{Kol23}, which applies as  $ (\fX_s,\fB_s)$ is slc and 
		\[
		m\fD_s - (K_{\fX_s} +\fB_s)  \sim_{\bQ}(m+1)\fD_s
		\]
		is ample.
		Thus the cohomology and base change theorem \cite[Theorem III.12.11]{Har77} implies  that (1) holds.

		To prove (2), let $V_m:=H^0(X,mD)$, $U := D(x_1\cdots x_r)$ and $U_i := D(x_1 \cdots \widehat{x}_i \cdots x_r)$. 
		Observe that the isomorphism $(\fX_{U},\fD_{U})\to (X,D) \times \bT$ induces an isomorphism
		\[
		\phi:
		H^0(\fX_{U}, m\fD_{U})
		\to
		V_m [x_1^{\pm1},\ldots, x_r^{\pm1 }] 
		.
		\]
		Note that $(\fX,\fB+\fD)$ is an lc pair by Proposition \ref{p:kollarsmoothslcfamgen}, since $\fX_{U} \cong X\times \bT$, which is normal. 
		
		To finish the proof, it suffices to show that 
		\[
		\phi(H^0(\fX,m\fD))={\rm Rees}(F_1,\ldots, F_r;V_m).
		\]		
		Fix an element $s\in H^0(\fX_{U},m\fD_{U})$. 
		By the isomorphism $\phi$, we can write
		$s=\sum_{\la\in \bZ^r}s_{\la} x^{-\la}$ for some $s_\la \in H^0(X,mD)$ with only finitely many $s_\la$ nonzero.
		Since
		$\fX\to \bA^r$ is flat and
		\[
		{\rm codim}_{\bA^r}(\bA^r \setminus (U_1 \cup \cdots \cup U_r))\geq 2,
		\]
		it follows that 
		\[
		{\rm codim}_{\fX}(\fX\setminus (\fX_{U_1}\cup \cdots \cup \fX_{U_r})) \geq 2
		.\]
		Thus  $s\in H^0(\fX,m\fD)$ if and only if 
		$s \in H^0(\fX_{U_i}, m\fD_{U_i})$ for all $1\leq i \leq r$.
		Using the isomorphism $\fX_{U_i} \cong  \cX^i \times \bG_m^{r-1}$, we see that the latter holds if  and only if
		\[
		\sum_{\la \in \bZ^r}  s_\la  x^{-\la} \in H^0(\cX^i\times \bG_m^{r-1}, m(\cD^i  \times \bG_m^{r-1})) 
		\]
		for all $1\leq i \leq r$.
		This is equivalent to the condition that 
		\[
		s_{\la} x_i^{-\la_i} \in H^0(\cX^i,m\cD^i)
		\]
		for all $\la \in \bZ^r$ and $1\leq i \leq r$.
		The latter holds if and only if   $s_{\la} \in F_i^{\la_i}V_m$ for all $\la \in \bZ^r$ and $1\leq i \leq r$, which is equivalent to the condition that  $\phi(s)\in {\rm Rees}(F_1,\ldots, F_r;V_m)$ as desired.
	\end{proof}
	
	We now combine the previous result with Proposition \ref{p:dival->tc}  to deduce the simultaneous diagonalizability of filtrations induced by divisorial valuations in the skeleton.

	\begin{thm}\label{t:skfiltsdiag}
		Let $(X,B+D)$ be an lc boundary polarized CY pair and $b$ be a positive integer such that $bD$ is an integral divisor.
		Let $v_1,\ldots, v_r$ be $\bZ$-valued divisorial valuations in $\LC(X,B+D)$, and write $F_i$ for the filtration of $V_m:=H^0(X,mbD)$ defined by
		\[
		F_i^\la V_m := \{s \in V_m\, \vert\, v_i(s) \geq \la \}.
		\]
		The following hold:
		\begin{enumerate}
			\item For each  integer $m\geq 0$,   ${\rm Rees}(F_1,\ldots, F_r;V_m)$ is  a flat $\bk[x_1,\ldots, x_r]$-module.
			\item The graded ring 
			$\bigoplus_{m \in \bN} {\rm gr}_{F_1,\ldots, F_r} (V_m)$
			is finitely generated and reduced. 
			\item The natural map ${\rm gr}_{F_1,\ldots, F_r}(V_{m-1}) \to {\rm gr}_{F_1,\ldots, F_r}(V_{m})$  induced by the inclusion $V_{m-1} \hookrightarrow V_{m}$ is injective for each $m \geq 1$.
		\end{enumerate}
	\end{thm}

	\begin{proof}
		By Proposition \ref{p:dival->tc},
		for each $1 \leq i \leq r$, there exists a test configuration $(\cX^i,\cB^i+\cD^i)$ of $(X,B+D)$ such that $\cX^i_0$ is irreducible and $v_{\cX_0^{i}} =v_i$.
		For the collection of test configurations
		\[
		(\cX^1,\cB^1+\cD^1),\ldots, (\cX^r,\cB^r+\cD^r),
		\]
		consider the $\bT:=\bG_m^r$-equivariant family of boundary polarized CY pairs 
		\[
		(\fX,\fB+\fD) \to \bA^r
		\]
		constructed in Proposition  \ref{p:highranktheta}.
		Since the filtration $F_i$ of $V_m$ agrees with the filtration induced by  the test configuration $(\cX^i,\cB^i+\cD^i)$ by Lemma \ref{l:filttc},  Proposition \ref{p:familyArRees} implies that   (1) holds.
		Furthermore, by Proposition \ref{p:familyArRees}.2, 
		there is a natural graded
		$\bk[x_1,\ldots, x_r]$-module isomorphism 
		\[
		\bigoplus_{m \in \bN}
		{\rm Rees}(F_1,\ldots, F_r;V_m)
		\overset{\simeq}{\longrightarrow}
		\bigoplus_{m \in \bN} H^0(\fX, mb\fD) 
		.\]
		As this isomorphism is compatible with multiplication, it is also an isomorphism of graded $\bk[x_1,\ldots, x_r]$-algebras. 
		By Lemma \ref{l:Reestquotient=graded} and
		Proposition \ref{p:familyArRees}.1, the base change via $\bk[x_1,\ldots, x_r ] \to \bk$ that sends $x_i\mapsto 0$ induces a graded algebra isomorphism 
		\begin{align}\label{e:isomgrings}
			\bigoplus_{m \in \N}
			{\rm gr}_{F_1,\ldots, F_r} (V_m)
			\overset{\simeq}{\longrightarrow}
			\bigoplus_{m \in \N}
			H^0(\fX_0, m b\fD_0) 
			.
		\end{align}
		Since $\fX_0$ is reduced and $\fD_0$ is ample, 
		the $\bk$-algebra on the right-hand side is   reduced and finitely generated. 
		Therefore  (2) holds. 
		
		To verify that (3) holds, note that multiplication by $\overline{1}\in {\rm gr}_{F_1,\ldots, F_r}(V_{1})$ induces the natural map 
		\[
		{\rm gr}_{F_1,\ldots, F_r}(V_{m-1})\to {\rm gr}_{F_1,\ldots, F_r}(V_{m})
		.\]
		Hence it suffices to show that $\overline{1}$ in degree $m=1$ is a nonzero divisor in the above graded ring. 
		To verify this, observe that  $\overline{1}$ corresponds to the restriction of $1\in H^0(\fX, b\fD)$ to $H^0(\fX_{0}, b \fD_{0})$. 
		As the section satisfies ${\rm div}_{\fX}(1) = b\fD$ and $\Supp(\fD)$ does not contain an irreducible component of $\fX_{0}$, the section $1\in H^0(\fX,\cO_{\fX}(b\fD))$ does not vanish on any irreducible component of $\fX_{0}$. 
		Thus $1 \vert_{\fX_{0}} \in H^0(\fX_{0}, \cO_{\fX_0}( b \fD_{0}))$ is a nonzero divisor in the graded ring in 
		\eqref{e:isomgrings}.
		Hence $\overline{1}\in {\rm gr}_{F_1,\ldots, F_r}(V_1)$ is also a  nonzero divisor in \eqref{e:isomgrings} as desired.
	\end{proof}
	
	We now combine the previous theorem with Proposition \ref{p:vsfinitediminfinitecol} to deduce the following result.

	\begin{cor}\label{c:diagfiltlcp}
		Let $(X,B+D)$ be an lc boundary polarized CY pair and  $b$ be a positive integer such that $bD$ is an integral divisor.
		For each $v\in \LC(X,B+D)$ and $m\geq0$,  let $F_v$ be the filtration of   $V_m:=H^0(X,mbD)$ defined by
		\[
		F_v^\la V_m := \{s \in V_m\, \vert\, v(s) \geq \la \}
		.\]
		The following hold: 
		\begin{enumerate}
			\item The filtrations $(F_v)_{v\in \LC(X,B+D)}$ of $V_m$ are simultaneously diagonalizable. 
			\item There exists a finite collection of $\bZ$-valued divisorial valuations $v_1,\ldots, v_r\in \LC(X,B+D)$ such that if  a basis $(s_1,\ldots, s_{N_m})$ of $V_m$ diagonalizes $F_{v_1},\ldots, F_{v_r}$, then the basis diagonalizes $F_v$ for all $v\in \LC(X,B+D)$.
		\end{enumerate}
	\end{cor}
	
	\begin{proof}
		Consider the set of valuations 
		\[
		J:= \{ v\in \LC(X,B+D) \, \vert\, \text{$v$ is divisorial and $\bZ$-valued}\}
		.\]
		By Theorem \ref{t:skfiltsdiag}.1, the filtrations $(F_v)_{v\in I}$ are  simultaneously diagonalizable for every finite set  $I\subset J$.
		Proposition \ref{p:vsfinitediminfinitecol} now implies the following:
		\begin{enumerate}
			\item[(i)] The filtrations $(F_v)_{v\in J}$ are simultaneously diagonalizable. 
			\item[(ii)] There exist finitely many valuations $v_1,\ldots, v_r \in J$ such that if a basis  $(s_1,\ldots, s_{N_m})$ of $V_m$ diagonalizes $F_{v_j}$ for all $1\leq j \leq r$, then it diagonalizes $F_v$ for all $v\in J$.
		\end{enumerate}
		Let $(s_1,\ldots, s_{N_m})$ be any basis of $V_m$ that diagonalizes $F_{v_i}$ for all $1\leq i \leq r$. 
		It remains to show that this basis diagonalizes $F_v$ for all $v\in \LC(X,B+D)$. 
		
		If $v\in J$, then the basis diagonalizes $F_v$ by (ii). 
		Next  assume that $v\in \LC(X,B+D)$ is divisorial, but not necessarily $\bZ$-valued. 
		Thus $v= c  v'$ for some $c\in \bR_{\geq 0}$ and $v'\in J$. 
		If $c=0$, then $F_v$  is the trivial filtration and so the basis diagonalizes $F_v$. 
		If $c>0$, then $F_{v}^\la V_m = F^{\la/c}_{v'} V_m$. 
		As the basis diagonalizes $F_{v'}$, it also diagonalizes $F_v$. 
		
		Finally, fix an arbitrary valuation $v\in \LC(X,B+D)$. 
		There exists a sequence of divisorial valuations $(v_j)_{j\geq 1}$ in $\LC(X,B+D)$ such that $v= \lim_{j \to \infty} v_j$. 
		Thus for any sum $s= \sum_{i=1}^{N_m} a_i s_i$ with $a_1,\ldots, a_{N_m}\in \bk$, we have that 
		\begin{multline*}
			\ord_{F_v}( s) = v( s)
			=
			\lim_{j} v_j (s) 
			=
			\lim_{j} \min\{ v_j(s_i)\, \vert\, a_i \neq 0 \}\\
			=
			\min\{  \lim_j v_j(s_i)\, \vert\, a_i \neq 0 \}
			=
			\min \{ v(s_i) \, \vert\, a_i \neq 0 \}
			= 
			\min\{ \ord_{F_v}(s_i) \, \vert\, a_i \neq 0 \}.
		\end{multline*}
		Thus the basis diagonalizes $F_{v}$ as desired. 
	\end{proof}

	\subsection{Valuative independence}
	We will now study and construct  valuatively independent bases for CY pairs.

	\subsubsection{Definition}
	
	\begin{defn}[Valuative Independence]
		Let $(X,B)$ be a projective lc CY pair and $W\subset K(X)$ a $\bk$-vector subspace. A basis $(\theta_i)_{i \in I}$ for $W$ is called \emph{valuatively independent} if 
		\[
		v\big( 
		\textstyle \sum_{i \in I} a_i \theta_i
		\big)
		= \min \left\{ v(\theta_i) \, \vert\, a_i \neq 0 \right\}
		\]
		for all $v\in {\rm LC}(X,B)$ and finite  sums 
		$\textstyle \sum_{i \in I} a_i \theta_i$  with $a_i \in \bk$.
	\end{defn}
	
	We will be primarily interested in the case when either $W = H^0(U, \cO_U)$ for some open subset $U\subset X$ or $W= H^0(X,G)$ for some $\bZ$-divisor $G$ on $X$. 
	Thus both finite-dimensional and infinite-dimensional cases of $W$ arise.

	A  valuatively independent basis need not be unique by the example in Section \ref{ss:Cayleycubic}. 
	However, we will see that the induced functions on the  set of lc places are independent of the choice of such a basis.
	
	\begin{defn}[Tropicalization]
		Let $(X,B)$ be a projective lc CY pair. The \emph{tropicalization} of a function  $0\neq s \in K(X)$ is the function 
		\[
		s^{\rm trop}:{\rm LC}(X,B) \to \bR \quad \text{ defined by } \quad 
		v\longmapsto v(s). 
		\]
	\end{defn}
	
	The term tropicalization is not used in the technical sense, but to highlight that $s^{\rm trop}$ encodes  the combinatorial data of a  piecewise linear function on a simplicial cone complex.

	\begin{lem}\label{l:tropthetaeasyprop}
		Let $(X,B)$ be a projective lc CY pair and $W\subset K(X)$ a $\bk$-vector subspace. 
		If $(\theta_i)_{i \in I}$ is a valuatively independent basis  for $W$, then the multiset of functions  
		\[
		\{ \theta_i^{\rm trop}\, \vert\, i \in I \}
		\]
		is independent  of  the choice of valuatively independent basis $(\theta_i)_{i \in I}$ of $W$.
	\end{lem}

	\begin{proof}
		For each $v\in {\rm LC}(X,B)$, let $F_v$ be the filtration of $W$ defined by 
		\[
		F_v^\la W = \{ s\in W\, \vert\, v(s) \geq \la \}.
		\] 
		Since $\ord_{F_v}(\theta_i) = v(\theta_i)= \theta^{\trop}_{i}(v)$, $(\theta_i)_{i\in I}$ is a diagonalizing basis for  $F_v$.

		Now let $\varphi : \LC(X,B) \to \bR$ be any function.
		By  Lemma \ref{l:diagbasisords}, applied to the collection of filtrations  $(F_v)_{v\in \LC(X,B)}$, the cardinality of the set
		\[
		\{ i \in I \, \vert\, \theta_i^{\rm trop} = \varphi \}
		\]
		does not depend on the choice of valuatively independent basis. 
		Thus the multiset of functions $(\theta_i^{\rm trop})_{i \in I}$  does not depend on the choice of valuatively independent basis.
	\end{proof}
	
	The valuatively independent bases we construct will satisfy the following useful property.

	\begin{defn}
		Let $(X,B)$ be a projective lc CY pair and $W\subset K(X)$ a $\bk$-vector subspace. A basis $(\theta_i)_{i \in I}$ for $W$ is  \emph{closed under tropical scaling} if for each $i\in I$ and $k\in \bZ_{>0}$, there exists $j\in I$ such that 
		$\theta_j^{\trop}  = k \theta_i^{\trop}$.
	\end{defn}

	By Lemma \ref{l:tropthetaeasyprop}, this property does not depend on the choice of a valuatively independent basis.

	\subsubsection{Boundary polarized CY pairs}
	We now  proceed to construct valuatively independent bases for boundary polarized CY pairs. 
	This is the key step in the proof of  Theorem \ref{thm:val-indepaffineCY}.

	\begin{thm}\label{t:valindbpcy}
		If $(X,B+D)$ is an lc boundary polarized CY pair and  $U := X\setminus \Supp(D)$, then there exists a valuatively independent basis  $(\theta_i)_{i \in I}$ of $H^0(U, \cO_U)$.
	\end{thm}

	\begin{proof}
		We may assume that $\dim X>0$, since otherwise the result holds trivially.
		Fix a positive integer $b$ such that $bD$ is an integral divisor. 
		For $m\geq 0$, let 
		\[
		V_m:= H^0(X,mb D)
		\quad \text{ and } \quad 
		N_m: = \dim V_m
		.\]
		Since $D$ is effective,  there are inclusions $V_0 \subset V_1 \subset V_2\subset \cdots $ 
		and  
		\[
		H^0(U, \cO_U)  = \textstyle \bigcup_{m \geq 0}V_m
		.\]
		To prove the result, we will first construct valuatively independent bases for each $V_m$.
		\medskip
		
		\noindent \emph{Claim:
			There exists a sequence $(\theta_i)_{i \geq 1}$  of elements in $H^0(U, \cO_U)$ 
			such that  $(\theta_{i})_{i=1}^{N_m}$ is a valuatively independent basis for $V_m$ for each $m\geq0$.
		}
		\medskip 
		
		\noindent \emph{Proof of claim}:
		We will construct the sequence inductively.
		When $m=0$, we have $V_0= H^0(X, \cO_X) = \bk$. 
		We set $\theta_1 = 1$, which trivially is a valuatively independent basis for $V_0$.
		
		Assume that there exists a sequence $(\theta_{i})_{i=1}^{N_{m-1}}$ of elements in $H^0(U,\cO_U)$ and that $(\theta_i)_{i=1}^{N_l}$ is a valuatively independent basis for $V_{l}$ for all $0 \leq l <m$.
		We now aim to define the next $N_{m}- N_{m-1}$ elements in the sequence.
		
		For each $0\leq l \leq m$ and $v\in \LC(X,B+D)$, consider the filtration
		\[
		F_{v}^\la V_l:= \{ s\in V_l \, \vert\, v(s) \geq \la\}.
		\]
		By Corollary \ref{c:diagfiltlcp}, the following hold:
		\begin{enumerate}
			\item[(i)] The filtrations $(F_v)_{v\in \LC(X,B+D)}$ of $V_m$ are simultaneously diagonalizable.
			\item[(ii)] There exist finitely many $\bZ$-valued divisorial valuations $v_1,\ldots, v_r\in \LC(X,B+D)$ such that if a basis of $V_m$ diagonalizes $F_{v_j}$ for all $1\leq j \leq r$, then the basis diagonalizes $F_v$ for all $v\in \LC(X,B+D)$. 
		\end{enumerate}
		Let  $F_j = F_{v_j}$ for $1\leq j \leq r$.

		Since the basis $(\theta_i)_{i=1}^{N_{m-1}}$ of $V_{m-1}$ is valuatively independent and $\ord_{F_j} = v_j$, Remark \ref{r:diagfiltsequiv} implies that  the basis $(\theta_i)_{i=1}^{N_{m-1}}$  of $V_{m-1}$ diagonalizes $F_j$ for each $1\leq j \leq r$. 
		By  Proposition  \ref{p:Reesflat=diagtriv}, the elements
		\[
		\overline{\theta}_{1},\ldots, \overline{\theta}_{N_{m-1}}
		\in 
		{\rm gr}_{F_1,\ldots, F_r}(V_{m-1})
		\]
		form a basis.
		By Theorem \ref{t:skfiltsdiag}, the natural map 
		\[
		{\rm gr}_{F_1,\ldots, F_r}(V_{m-1}) \to {\rm gr}_{F_1,\ldots, F_r}(V_m)
		\]
		is injective.
		Thus the corresponding elements 
		\[
		\overline{\theta}_{1},\ldots, \overline{\theta}_{N_{m-1}}
		\in {\rm gr}_{F_1,\ldots, F_r}(V_m)
		\]
		are linearly independent. 
		Thus the sequence of elements 
		$(\theta_{i})_{i=1}^{N_{m-1}} $ in $V_{m-1}$
		extends to a sequence of elements 
		$(\theta_i)_{i=1}^{N_m}$ in $V_m$
		such that
		\[
		\overline{\theta}_{1},\ldots, \overline{\theta}_{N_{m}}
		\in
		{\rm gr}_{F_1,\ldots, F_r}(V_m)
		\]
		form a basis. 
		By  Proposition  \ref{p:Reesflat=diagtriv}, $(\theta_{i})_{i=1}^{N_m}$  is a basis for $V_m$ that  diagonalizes $F_j$ for each $1\leq j \leq r$.
		By (ii), $(\theta_{i})_{i=1}^{N_m}$ is a basis for $V_m$ that diagonalizes $F_v$ for  all $v\in \LC(X,B+D)$.
		By Remark \ref{r:diagfiltsequiv}, the basis $(\theta_{i})_{i=1}^{N_m}$ of $V_m$ satisfies valuative independence. 
		\hfill \qed
		\medskip
		
		Consider a sequence of elements $(\theta_i)_{i \geq 1}$ produced by the claim. 
		Note that  $(\theta_i)_{i=1}^{N_m}$ is a valuatively independent basis for each $V_m$ and $H^0(U, \cO_U) = \cup_{m \geq 0} V_m$.
		Thus $(\theta_i)_{i\geq 1}$ is a basis for $H^0(U,\cO_U)$. 
		Since the definition of valuative independence involves only finite sums, $(\theta_i)_{i\geq 1}$ satisfies valuative independence.
	\end{proof}
	
	We now prove that the above basis is closed under tropical scaling. 
	
	\begin{prop}\label{p:basisclosedsclaing}
		Let $(X,B+D)$ be a boundary polarized lc CY pair  and $U := X \setminus \Supp(D)$.
		If $(\theta_i)_{i \in I}$ is a valuatively independent basis of $H^0(U, \cO_U)$, then it is closed under tropical scaling. 
	\end{prop}
	
	The result is a consequence of the reducedness statement in Theorem \ref{t:skfiltsdiag}.2.
	
	\begin{proof}
		We may assume that $\dim X>0$, since otherwise the result is trivial.
		Fix an integer $b>0$ such that $bD$ is an integral divisor. 
		Let 
		\[
		V_{m} := H^0(X,mb D) \subset K(X)\quad \text{ and }\quad  N_m:=\dim V_m
		.\]
		Since the tropical theta functions are independent of the choice of valuatively independent basis by Lemma \ref{l:tropthetaeasyprop}, we may assume that the basis is the one constructed in  the proof of Theorem \ref{t:valindbpcy}. 
		In particular,
		$I =\bZ_{>0}$ and $(\theta_i)_{i=1}^{N_m}$ is a valuatively independent basis for $V_m$ for each $m\geq 0$.
		
		Fix any $i' \in I$ and $k\in \bZ_{>0}$. 
		Choose $m\geq 0$ such that $1\leq i'\leq N_m$. 
		Aiming for a contradiction, 
		assume that 
		\[
		k\theta_{i'}^{\trop}\neq \theta_{j}^{\trop}
		\]
		for all $1\leq j \leq N_{mk}$.
		Let $r:=N_{mk}$.
		Since divisorial valuations are dense in $\LC(X,B+D)$ and
		the tropicalization of an element in $ K(X)^\times $ is continuous and homogeneous of degree 1, there exist $\bZ$-valued divisorial valuations $v_{1},\ldots, v_{r} \in \LC(X,B+D)$ such that 
		\begin{equation}\label{eq:tropscal}
			k  \theta_{i'}^\trop (v_j) \neq \theta_{j}^{\trop}(v_j)
		\end{equation}
		for all $1\leq j \leq r$.
		
		For each $1\leq j \leq r$ and $l\geq 0$, let $F_j:= F_{v_j}$ denote the filtration of $V_l$ induced by $v_j$.
		For $l\geq 0$, the basis  $(\theta_i)_{i=1}^{N_l}$ of $V_l$ satisfies valuative independence. 
		Using Remark \ref{r:diagfiltsequiv},  $(\theta_i)_{i=1}^{N_l}$ diagonalizes  the filtrations $F_1,\ldots, F_r$ of $V_l$.
		In particular the images of $(\theta_i)_{i=1}^{N_l}$ in ${\rm gr}_{F_1,\ldots, F_r} V_{l}$
		are  a homogeneous  basis  for the graded vector space by Proposition \ref{p:Reesflat=diagtriv}.
		Thus  
		\[
		0\neq 
		\overline{\theta}_{i'}\in 
		{\rm gr}_{F_1,\ldots, F_r}(V_{ m})
		\]
		and is in degree $\ord_{F}(\theta_{i'}) :=(\ord_{F_1}(\theta_{i'}),\ldots , \ord_{F_r}(\theta_{i'})) \in \bZ^r$.
		Since the graded ring in Theorem \ref{t:skfiltsdiag} is reduced, 
		\[
		0\neq 
		\overline{\theta}_{i'}^k
		\in
		{\rm gr}_{F_1,\ldots, F_r}(V_{mk})
		\]
		and is in degree $k\cdot \ord_{F}(\theta_{i'}) \in \bZ^r$.
		Since the images of $(\theta_i)_{i=1}^{N_{mk}}$ in 
		${\rm gr}_{F_1,\ldots, F_r}(V_{mk})$ form a  basis,  there exists $1\leq j \leq N_{mk}$ such  that the image of $\theta_{j}$ in ${\rm gr}_{F_1,\ldots, F_r}(V_{mk})$ is in degree $k\ord_{F}(\theta_{i'})$.
		Thus
		\[
		k \ord_{F}(\theta_{i'})=  \ord_{F}(\theta_{j}) \in \bZ^r
		.\]
		Taking the value at the $j$-th entry gives that
		\[
		k v_{j}(\theta_{i'})=
		k \ord_{F_{j}} (\theta_{i'}) = \ord_{F_{j}}(\theta_{j}) = v_{j}(\theta_{j})
		.\]
		Thus 
		\[
		k \theta_{i'}^{\trop}(v_{j}) = \theta_{j}^{\trop}(v_{j})
		,\]
		which contradicts \eqref{eq:tropscal}. 
		Thus $k \theta_{i'}^{\trop} = \theta_j^{\trop}$ for some $1\leq j \leq N_{mk}$. 
		Therefore the basis is closed under tropical scaling.
	\end{proof}
	
	Below, we will apply Theorem \ref{t:valindbpcy} to construct valuatively independent bases on other classes of CY pairs.

	\subsubsection{Positive boundary}

	\begin{thm}\label{t:affineCYposbound}
		If $(X,B)$ is a projective lc CY pair
		and  there exists an effective big and semiample  Cartier divisor $A$ on $X$ with $\Supp(A) \subset \Supp(B)$, then there exists a valuatively independent basis  $(\theta_i)_{i \in I}$ for  $H^0(U, \cO_U)$ with  
		$U := X\setminus \Supp(A)$.
		
		Furthermore, any valuatively independent basis of $H^0(U, \cO_U)$ is closed under tropical scaling.
	\end{thm}
	
	\begin{proof}
		We first prove the result when $A$ is ample.
		Since $\Supp(A) \subset \Supp(B)$, there exists $c\in  (0,1) \cap\bQ$ such that 
		$B' := B- cA \geq 0$. Let $D'= cA$. 
		Note that $(X,B'+D')$ is an lc boundary polarized CY pair satisfying  
		\[
		B'+D'= B
		\quad \text{ and} \quad 
		U := X\setminus \Supp(A) =X\setminus  \Supp(D').
		\]
		Thus Theorem \ref{t:valindbpcy} and Proposition \ref{p:basisclosedsclaing} imply that there exists a 
		valuatively independent basis   for  $H^0(U, \cO_U)$ that is closed under tropical scaling.
		
		Now we consider the case when $A$ is big and semiample. 
		By definition, there is a proper birational morphism  $f:X\to Y$ to a normal projective variety such that $A\sim_{\bQ} f^*H$ for some ample $\bQ$-Cartier $\bQ$-divisor $H$ on $Y$. 
		Since 
		\[
		f_* A  \sim_{\bQ} f_* f^* H = H
		,\]
		$f_* A$ is ample and $A \sim_{\bQ} f^* f_*A$. 
		Furthermore, the negativity lemma applied to $A-f^* f_*A$ and its negative implies that $A= f^* f_* A$. 
		
		Now we construct a valuatively independent basis on $Y$.
		Since $(X,B)$ is an lc CY pair, 
		\[
		f:(X,B) \to (Y,B_Y)
		\]
		is crepant, where $B_Y:= f_* B$.
		In particular, $(Y,B_Y)$ is an lc CY pair.
		Since $f_*A$ is an effective ample divisor satisfying $\Supp(f_* A) \subset \Supp(B_Y)$, the ample case implies that there exists a valuatively independent basis $(\theta_i)_{i \in I}$ for $H^0(U_Y,\cO_{U_Y})$ that is closed under tropical scaling, where $U_Y:= Y\setminus \Supp(f_* A)$. 
		
		Finally, we construct the basis on $X$. 
		Since $f^{*} f_* A = A$, $U=f^{-1}(U_Y)$ holds. 
		Using that $f_*\cO_X = \cO_Y$, the natural map
		\[
		H^0(U_Y, \cO_{U_Y}) \overset{\simeq}{\longrightarrow} H^0(U, \cO_U)
		\]
		is an isomorphism.
		Thus the basis for $H^0(U_Y, \cO_{U_Y})$ maps to a  basis for $H^0(U, \cO_U)$.
		By the crepant property,   $\LC(X,B)= \LC(Y,B_Y)$ under the identification of $K(X)$ and $K(Y)$.
		Hence the basis  for $H^0(U, \cO_U)$ satisfies valuative independence and is closed under tropical scaling.
		By Lemma \ref{l:tropthetaeasyprop}, any other such valuatively independent basis must also be closed under tropical scaling.
	\end{proof}
	
	\begin{proof}[Proof of Theorem \ref{thm:val-indepaffineCY}]
		The result follows from Theorem \ref{t:affineCYposbound}.
	\end{proof}
	
	\begin{rem}[Relation to Gross--Hacking--Keel basis]
		The construction in  \cite{GHK15} produces a mirror family for an lc CY pair $(X,B)$ such that $\dim X=2$, $X$ is smooth, and $B\in |-K_X|$ is a singular nodal  curve.
		Using that the relevant pairs are self-mirror, 
		\cite{GHK15,LZ25} show that there exists a canonical basis for
		$H^0(U, \cO_U)$, where $U:= X\setminus \Supp(B)$, when the pair is \emph{positive} \cite[Definition 6.10]{GHK15}.
		This basis  satisfies valuative independence by \cite{Man16}.
		
		The condition that $(X,B)$ is \emph{positive} is equivalent to requiring that  there exists an effective big and nef $\bQ$-divisor $A$ on $X$ such that $\Supp(A)=\Supp(B)$ by \cite[Lemma 6.9.1.4]{GHK15}. 
		By the base point free theorem \cite[Theorem 3.3]{KM98} applied to $A$ and the klt pair $(X,B-\varepsilon A)$ with $0<\varepsilon \ll1$,  $A$ is semiample and, hence,  Theorem \ref{t:affineCYposbound} implies the existence of  a valuatively independent basis for $H^0(U, \cO_U)$. 
		By Lemma \ref{l:tropthetaeasyprop}, tropicalizations of a basis produced by Theorem \ref{t:affineCYposbound} and of the Gross--Hacking--Keel basis agree.
	\end{rem}

	\subsubsection{Maximal boundary}

	\begin{defn}\label{d:maximal}
		A projective lc CY pair $(X,B)$ has \emph{maximal boundary} if there exists  a dlt modification $(Y,B_Y) \to (X,B)$ such that  there exist distinct prime divisors $E_1,\ldots, E_{r} \subset \Supp((B_Y)^{=1})$  such that $r= \dim X$ and $\cap_{i=1}^{r} E_i\neq \emptyset$. 
	\end{defn} 
	
	\begin{rem}
		If the condition in Definition \ref{d:maximal} is satisfied for one dlt modification, it is satisfied for all dlt modifications by \cite[Proposition 11]{dFKX17}.  
		Furthermore, the maximal boundary condition is sometimes  phrased as saying that $(B_Y)^{=1}$ contains a $0$-stratum or, in \cite{FMM25}, that $(X,B)$ has coregularity zero.
	\end{rem}

	\begin{thm}\label{t:affinecymaxb}
		If $(X,B)$ is a projective lc CY pair  with maximal boundary, then there exists an open set $U\subset X$ in the klt locus of $(X,B)$
		and a valuatively independent basis  $(\theta_i)_{i \in I}$ of $H^0(U, \cO_U)$ that is closed under tropical scaling.
	\end{thm}

	To prove the theorem, we will use the key technical result in \cite{KX16} to reduce to the positive boundary case in Theorem   \ref{t:affineCYposbound}.
	
	\begin{proof}
		Fix a $\bQ$-factorial dlt modification $f:(Y, B_Y)\to (X,B)$.
		Since $(Y,B_Y)$ is a dlt CY pair and has a $0$-dimensional lc center, \cite[Theorem 49]{KX16}  implies  there exists a crepant birational map 
		\[
		g: (Y, B_Y) \dashrightarrow (Z,B_Z)
		\]
		to an lc CY pair $(Z,B_Z)$ such that
		\begin{enumerate}
			\item[(i)] there exists an effective ample divisor $A$ on $Z$ with  $\Supp(A) = (B_Z)^{=1}$, 
			\item[(ii)] the prime divisors contracted by $g^{-1}$ are contained in $(B_{Z})^{=1}$,  and
			\item[(iii)]  $g^{-1}$ is an isomorphism over $Z \setminus (B_Z)^{=1}$. 
		\end{enumerate} 
		(Here we are using that $(B_Y)^{=1}$ contains a $0$-stratum to deduce that the target of the morphism in the statement of \cite[Theorem 49]{KX16} is a point.) 
		Let  $U_Z:= Z \setminus (B_Z)^{=1}$.
		By Theorem \ref{t:affineCYposbound}, there exists a valuatively independent basis $(\theta_i)_{i \in I}$ for $H^0(U_Z, \cO_{U_Z})$ with respect to ${\rm LC}(Z,B_Z)$ and it is closed under tropical scaling.
		
		We seek to show that this gives the desired valuatively independent basis. 
		First, note that birational maps induce isomorphisms $K(X) \cong K(Y) \cong K(Z)$.
		Since $f$ and $g$ are crepant birational,  we also have
		\[
		{\rm LC}(X,B) = {\rm LC}(Y,B_Y)= {\rm LC}(Z,B_Z)
		\]
		under the above  isomorphisms.
		Using (iii), $g^{-1}$ restricts to an open embedding $U_Z\hookrightarrow  Y$. Let $U_Y$ denote its image. 
		Since $(Y,B_Y) \dashrightarrow (Z,B_Z)$ is crepant, $U_Y \subset Y\setminus(B_Y)^{=1}$. 
		Let $U:= X\setminus f(Y\setminus U_Y)$. 
		Since $f$ is proper, $U$ is an open subset of $X$.
		Since the exceptional divisors of $Y\to X$ are contained in the support of $(B_Y)^{=1}$ and $f$ is an isomorphism over the codimension one points of $X$, $U_Y \dashrightarrow U$ is an isomorphism in codimension one.
		Thus  $U_Z\dashrightarrow U $ is an isomorphism in codimension one.
		Since $U$ and $U_Z$ are in addition normal, the  isomorphism $K(U_Z) \cong K(U)$ restricts to an isomorphism
		\[
		H^0(U_Z,\cO_{U_Z}) \overset{\simeq}{\longrightarrow}
		H^0(U, \cO_{U}) 
		. \]
		Thus the image of   $(\theta_i)_{i\in I}$  under this map is  a valuatively independent basis for $H^0(U,\cO_U)$ that is closed under tropical scaling. 
		Furthermore, since $U_Y \subset Y\setminus (B_Y)^{=1}$, $U_Y$ is in the klt locus of $(Y,B_Y)$. 
		As $f$ is crepant,  $U$ is   in the klt locus of $(X,B)$.
	\end{proof}
	
	The next lemma shows that  the tropical theta functions are distinct in the maximal case.
	
	\begin{lem}\label{l:maxcythetaidistinct}
		Let $(X,B)$ be a projective lc CY pair and $W\subset K(X)$ a $\bk$-vector subspace. 
		If $(X,B)$ has maximal boundary and $(\theta_i)_{i\in I}$ is a basis for $W$ satisfying valuative independence, then the functions $(\theta_{i}^{\rm trop})_{i \in I}$ are distinct.
	\end{lem}
	
	\begin{proof}
		Fix a dlt modification $f:(Y,B_Y) \to (X,B)$.
		Since $(X,B)$ has maximal boundary, there exist $r:=\dim(X)$ distinct prime divisors  $E_1, \ldots, E_r $ in $\Supp( (B_Y)^{=1})$ and a closed point $y \in E_1 \cap \cdots \cap E_r$.
		Since $\bk$ is algebraically closed, $y$ is a $\bk$-point.
		
		Choose local coordinates $y_1,\ldots, y_r \in \cO_{Y,y}$ such that $y_i$ is a defining equation for $E_i$ at $y$.  
		By Cohen's Structure Theorem, $\widehat{\cO}_{Y,y} \cong \bk\llbracket y_1,\ldots, y_r\rrbracket $.
		Fix $\alpha \in \bR^r_{> 0}$ such that $\alpha_1,\ldots, \alpha_r$ are linearly independent over $\bQ$ and consider the valuation 
		\[
		v:=v_{\alpha} \in \QM_{y}(Y,E_1+ \cdots + E_r) \subset {\rm LC}(X,B)
		.\]
		
		Now assume that there exist $i,j \in I$  such that $\theta_i^{\rm trop} = \theta_{j}^{\trop}$.
		We seek to show that $i=j$.
		Write $\theta_i / \theta_j = g/h$, where $g,h \in \cO_{Y,y}$. 
		Since $\theta_i^{\rm trop} = \theta_j^{\rm trop}$, $v(g)= v(h) $.
		By the linear independence of  $\alpha_1,\ldots, \alpha_r$, there exists a unique $\beta \in \bN^r$ such that 
		\[
		v(g) = \langle \alpha, \beta \rangle  = v(h)
		.\]
		By the definition of $v$, $\beta$ must be  the exponent of the lowest-order term in the Taylor series expansion of both $g$ and $h$ in $\bk\llbracket y_1,\ldots, y_r\rrbracket$ with respect to the order on the monomials  induced by $\alpha$.
		Thus there exists $c \in \bk^\times$
		such that $g -c h$ has no degree $\beta$ term in the Taylor expansion. Thus
		\[
		v(g-ch) > v(g) = v(h).
		\]
		Now we compute that 
		\begin{multline*}
			v(\theta_i - c \theta_j) 
			= 
			v(\theta_i/\theta_j - c) +v(\theta_j)
			=
			v(g/h- c)+ v(\theta_j)
			=
			v(g-ch) -v(h) + v(\theta_j)\\
			>
			v(g) - v(h) + v(\theta_j) 
			=
			v(\theta_j)
			=
			v(\theta_i).
		\end{multline*}
		Since $\theta_i$ and $\theta_j$ are contained in a valuatively independent basis, we must have that $i=j$.
	\end{proof}

	\subsubsection{Polarized CY pairs}
	We now prove a version of the above results for sections of an ample line bundle on a CY pair. 
	Using the cone construction, the proof reduces to the boundary polarized CY pair result in Theorem \ref{t:valindbpcy}.
	
	\begin{thm}\label{t:valindlb}
		Let $(X,B)$ be a projective lc CY pair and $L$ be an ample line bundle on $X$.
		If $m$ is a non-negative integer, then there exists a basis 
		$\theta_1,\ldots, \theta_{N_m}$ of $H^0(X,L^{\otimes m})$ such that 
		\begin{equation}\label{eq:valindpol}
			v(a_1 \theta_1 + \cdots + a_{N_m} \theta_{N_m}) 
			=
			\min \{ v(\theta_i) \, \vert\, a_i \neq 0 \}
		\end{equation}
		for all $a_1,\ldots, a_{N_m} \in \bk$ and $v\in {\rm LC}(X,B)$.
	\end{thm}
	
	\begin{proof}
		We may assume that $\dim X >0$ as otherwise the statement holds trivially.
		Consider the affine and projectivized cones
		\[ 
		C_a(X,L) := \Spec\, \bigoplus_{i \in \bN} H^0(X,L^{\otimes i})
		\quad \text{ and } \quad 
		C_p(X,L) := \Proj \, \bigoplus_{m \in \bN} \bigoplus_{i=0}^m H^0(X,L^{\otimes i}) x^{m-i}
		,\]
		where the Proj is taken with respect to the $m$-grading.
		Observe that $V(x) \subset C_p(X,L)$ is an ample Cartier divisor and $C_p(X,L)\setminus V(x) \cong C_a(X,L)$.
		Write $0 \in C_a(X,L)$ for the cone point and $C_a(B,L)$ for the $\bQ$-divisor on $C_a(X,L)$  that is the cone over $B$ and  $C_p(B,L)$ for its closure in $C_p(X,L)$; see \cite[Section 3.1]{Kol13}.
		Let
		\[
		Y:= C_p(X,L), \quad G:=  C_p(B,L), \quad \text{ and } \quad 
		D:= V(x) 
		.\]
		By \cite[Proposition 2.20.3]{ABB+},  $(Y,G+D)$ is an lc boundary polarized CY pair. 
		Thus the proof of Theorem \ref{t:valindbpcy} implies that there exists a basis $\theta_1,\ldots,\theta_{M_m}$ of $H^0(Y,mD)$  satisfying valuative independence. 
		
		Observe that for $f\in H^0(Y\setminus D, \cO_Y) \cong \bigoplus_{i \in \bN} H^0(X,L^{\otimes i})$,
		\[
		\ord_{D}(f)= \min \{ -i \, \vert\, f_i 
		\neq 0\}
		\quad \text{ and }\quad 
		\ord_0(f) := \min \{ i  \,\vert\, f_i \neq 0 \}
		,\]
		where $f= \sum_{i \geq 0} f_i$ with 
		$f_i \in H^0(X,L^{\otimes i})$.
		Therefore
		\[
		H^0(Y, mD) =  H^0(X,L^{\otimes 0}) \oplus H^0(X,L^{\otimes 1}) \oplus \cdots \oplus  H^0(X,L^{\otimes m})
		,
		\]
		and 
		\[
		F_{\ord_0}^{m}H^0(Y,mD)=H^0(X,L^{\otimes m}).
		\]
		Since $\ord_{0} \in {\rm LC}(Y,G+D)$ by  \cite[Proof of Lemma 3.1]{Kol13},
		\[
		F^{m}_{\ord_{0}}H^0(Y,mD) =\langle \theta_i \, \vert\, \ord_{0}(\theta_i) \geq m \rangle
		.\]
		Thus $(\theta_i \, \vert\, \ord_{0}(\theta_i) \geq m)$ is a basis for $H^0(X,L^{\otimes m})$.
		After reordering $(\theta_{i})_{i=1}^{M_m}$, we may assume that $(\theta_i \, \vert\, \ord_{0}(\theta_i) \geq m) = (\theta_{i})_{i=1}^{N_m}$ for some $N_m \leq M_m.$
		
		It remains to verify that \eqref{eq:valindpol} holds.
		We claim that for  any $v\in \LC(X,B)$, there exists a valuation  $w\in {\rm LC}(Y,G+D)$ such that 
		\begin{equation}\label{e:coneval}
			w(f) = \min \{ v(f_i) \, \vert\, f_i \neq 0 \}
		\end{equation}
		for $f= \sum_{i \geq 0} f_i$ with $f_i \in H^0(X,L^{\otimes i})$.
		To verify this, fix a dlt modification  $g:Z \to X $ of $(X,B)$ and write $B_Z$ for the $\bQ$-divisor on $Z$ such that $K_Z+B_Z =g^*(K_X+B)$. 
		By Lemma \ref{l:LCplace},  ${\rm LC}(X,B)= {\rm QM}(Z,(B_Z)^{=1})$.
		Now consider the $\bG_m$-bundles 
		\[
		\pi_Z: Z_{g^*L} := \Spec_Z\, \bigoplus_{i \in \bZ} (g^*L)^{\otimes i} \to Z
		\quad \text{ and } \quad 
		\pi_X:X_{L}:= \Spec_X \, \bigoplus_{i \in \bZ} L^{\otimes i} \to X.
		\]
		Note that we have a crepant proper birational morphism  and an open embedding
		\[
		(Z_{g^*L}, \pi_Z^{-1}(B_Z)) 
		\to 
		(X_{L}, \pi_X^{-1}(B))
		\hookrightarrow
		(Y, G+D).
		\]
		By trivializing the $\bG_m$-bundles on an open cover, we see that there is a set theoretic bijective map 
		\[
		\phi:{\rm LC}(X,B)
		=
		{\rm QM}(Z,(B_Z)^{=1}) 
		\to 
		{\rm QM}(Z_{g^*L}, (\pi_Z^{-1}(B_Z))^{=1})
		\subset {\rm LC}(Y,G+D)
		\]
		such that $\phi$ maps $v$ to a valuation $w$ satisfying \eqref{e:coneval}, which proves the claim. 
		Since $(\theta_i)_{i=1}^{M_m}$ satisfies valuative independence with respect to $\LC(Y,G+D)$, the claim implies that   \eqref{eq:valindpol} holds.
	\end{proof}
	
	\begin{proof}[Proof of Theorem  \ref{thm:val-indlb}]
		The result is a restatement of Theorem \ref{t:valindlb}.
	\end{proof}
	
	\subsection{Examples}
	
	We now give two examples. The valuatively independent basis is unique up to scaling in the first one, but not in the second. 
	
	\subsubsection{Toric CY pairs}
	We follow the notation for toric varieties in \cite{CLS11}. 
	Fix a lattice $N\cong \bZ^n$, and let $M:= \Hom (N,\bZ) \cong \bZ^n$. 
	Let $X:= X_{\Sigma}$ be the projective toric variety associated to a complete fan $\Sigma$ in $N$ and $B$ denote the sum of the toric invariant prime divisors.
	By \cite[Theorem 8.2.3 and Corollary 11.4.25]{CLS11}, $(X,B)$ is an lc CY pair. 
	
	Let $U:= X\setminus \Supp(B)$, which is the dense torus in $X$.
	Thus 
	\[
	H^0(U, \cO_U) = \bigoplus_{\beta \in M} \bk \cdot \chi^\beta.
	\]
	For each $\alpha \in N_{\bR}:= N\otimes_{\bZ}\bR$, there exists a valuation $v_{\alpha} \in \Val_X$ such that  
	\[
	v_{\alpha}\big(\textstyle f \big)  
	=
	\min \{ \langle \alpha,\beta \rangle \, \vert\, c_\beta \neq 0 \}
	\]
	for any finite sum $f=\sum_{\beta\in M} c_\beta \chi^\beta$ with each $c_\beta \in \bk$.

	\begin{prop}\label{p:toric}
		For a projective toric lc CY pair $(X,B)$ as above, $(\chi^\beta)_{\beta\in M}$ is a valuatively independent basis  for $H^0(U, \cO_U)$ and is the unique such basis up to scaling elements.
	\end{prop}
	
	\begin{proof}
		We first show that 
		$\LC(X,B) = \{ v_{\alpha} \in \Val_X \, \vert\, \alpha \in N_{\bR} \}$. 
		By \cite[Theorem 11.1.9]{CLS11}, there exists  a smooth fan $\Sigma'$ of $N$ that is a refinement of $\Sigma$.
		Let $X':= X(\Sigma')$ and $B'$ denote the reduced toric boundary divisor. 
		The refinement induces a proper birational morphism $(X',B') \to (X,B)$. The morphism is crepant, since $(X',B')$ is also an lc CY pair. 
		Since $(X',B')$ is log smooth, 
		\[
		\LC(X,B) 
		= \LC(X',B')
		= \{ v_\alpha \in \Val_X \, \vert\, \alpha \in N_{\bR} 
		\}
		,\]
		where the first equality is Lemma \ref{l:LCplace} and the second holds by explicit computation on the open sets of $X'$ corresponding to the cones of $\Sigma'$.
		By the definition of $v_\alpha$,  $(\chi^\beta)_{\beta\in M}$ is a valuatively independent basis for $V:=H^0(U, \cO_U)$.
		
		Fix any other valuatively independent basis $(\theta_i)_{i \in I}$ for $V$. 
		For any $\beta'\in M$, we will show that a scalar multiple of $\chi^{\beta'}$ is an element of $(\theta_i)_{i \in I}$.
		For $\alpha \in N_{\bR}$ and $\la \in \bR$, observe that
		\[
		F_{\alpha}^\la V := F^\la _{v_\alpha}
		V 
		={\rm span}( \chi^\beta \, \vert\, \langle \alpha,\beta\rangle \geq \la ) 
		.\]
		Thus if we set $\la'_\alpha:= \langle \alpha,\beta'\rangle $, then
		\[
		W:=
		\bigcap_{\alpha \in N_{\bR}} F_{\alpha}^{\la'_{\alpha}}V 
		= 
		{\rm span}( \chi^\beta \, \vert\, \langle \alpha, \beta \rangle \geq \langle \alpha, \beta' \rangle  \text{ for all } \alpha \in N_{\bR}
		) 
		=\bk  \chi^{\beta'}
		.\]
		Since $(\theta_i)_{i\in I}$ is a valuatively independent basis and $\ord_{F_\alpha} = v_\alpha$ for $\alpha\in N_{\bR}$, this basis simultaneously diagonalizes $(F_\alpha)_{\alpha\in N_{\bR}}$. 
		Hence $W$ is spanned by a subset of the elements in  $(\theta_i)_{i\in I}$, and so  a scalar multiple of $\chi^{\beta'}$ occurs among the $\theta_i$.
		As this holds for any $\beta' \in M$, it follows that $(\theta_i)_{i\in I}$ consists of  scalar multiples of the monomials of $V$.
	\end{proof}

	\subsubsection{Cayley cubic surface with a triangle}\label{ss:Cayleycubic}
	The following example, which originated in \cite[Remarks 1.1 and 3.7]{GHKS22}, shows that valuatively independent bases need not be unique up to scaling elements.
	
	Fix the lattice $N:= \bZ^2$, and let $M:=\Hom(N,\bZ) \cong \bZ^2$. Let $\Sigma$ be the complete fan in $N$ whose rays, listed in cyclic order, are generated by 
	\[
	(1,0),\quad (0,1),\quad (-1,1),\quad (-1,0),\quad (0,-1),\quad (1,-1).
	\]
	Let $\tX:= X_{\Sigma}$ be the associated projective toric surface, and $\tB$ the sum of the toric invariant prime divisors. Then $\tX$ is a smooth del Pezzo surface of degree $6$ and $\tB$ is a cycle of six copies of $\bP^1$. Denote by $\tU := \tX \setminus \tB\cong\bG_m^2$ the dense algebraic torus with coordinates $(x_1,x_2)$.  
	
	Next, consider the lattice involution $-\id_N: N\to N$ that sends $\alpha$ to $-\alpha$. Since $-\id_N$ preserves $\Sigma$, it induces an involution $\iota$ on $(\tX, \tB)$. It is clear that $\iota|_{\tU}(x_1,x_2) = (x_1^{-1}, x_2^{-1})$. Moreover, if we write $\tB = \sum_{i=1}^6 \tB_i$ according to the cyclic order of the rays, then $\iota|_{\tB}(\tB_i) = \tB_{i+3}$ where the indices are taken modulo $6$. Hence $\iota|_{\tB}$ is fixed point free. Therefore, the fixed set of $\iota$ is a $2$-torsion subgroup $\tU[2]$ which consists of four points in $\tU$.
	
	Next, let $(X,B):= (\tX, \tB)/\iota$ with the quotient map $\pi: \tX \to X$. Since the fixed set of $\iota$ is finite, we know that $\pi$ is quasi-\'etale and is \'etale away from the four points $\tU[2]$. Thus we have 
	\[
	0\sim K_{\tX}+ \tB = \pi^*(K_X+B) \quad \textrm{ and }\quad \tB = \pi^*B.
	\]
	Thus $(X,B)$ is a dlt CY pair with $B$ ample. Thus $U:= X\setminus B = \tU/\iota$ is an affine CY variety which has four $A_1$-singularities and is smooth elsewhere. By \cite[Remark 1.1]{GHKS22}  the map $U\to \bA^3$ sending  $(x_1,x_2) \mapsto (x_1+x_1^{-1}, x_2+x_2^{-1}, x_1x_2+x_1^{-1} x_2^{-1})$ 
	induces an isomorphism
	\[
	U \cong V(x^2+y^2 +z^2 - xyz-4) \subset \bA^3_{x,y,z}.
	\]
	Projectivizing  gives an isomorphism to the Cayley cubic surface
	\[
	X\cong V((x^2+y^2+z^2)w - xyz-4w^3)\subset \bP^3_{x,y,z,w}. 
	\]
	Under this identification, we have that $B =  X\cap V(w) = V(xyz, w)$ is a triangle.
	
	Next, for each $\beta \in M$ we define 
	\[
	\tilde{\theta}_{\beta}:= \chi^{\beta}+ \chi^{-\beta} \in H^0(\tU, \cO_{\tU}). 
	\]
	Since each $\tilde{\theta}_{\beta}$ is $\iota$-invariant, it descends to a function $\theta_{\beta}\in H^0(U,\cO_U)$. 
	
	\begin{prop}\label{p:cayleycubic}
		Let $(X,B)$ be the Cayley cubic surface with the triangle of lines and $U=X\setminus B$ as above. Then 
		$({\theta}_{\beta})_{\beta\in M/\{\pm 1\}}$ is a valuatively independent basis for $H^0(U,\cO_U)$. However, this basis is not unique up to scaling elements among valuatively independent bases of $H^0(U,\cO_U)$.
	\end{prop}

	\begin{proof}
		First,  $(\theta_{\beta})_{\beta\in M/\{\pm 1\}}$ form a basis since
		\[
		H^0(U,\cO_U) \cong H^0(\tU, \cO_{\tU})^{\iota} = \bk[x_1^{\pm1}, x_2^{\pm1}]^\iota  = \bigoplus_{\beta\in M/\{\pm 1\}} \bk\cdot \tilde{\theta}_{\beta}.
		\]
		Next, since $\pi:(\tX,\tB)\to (X,B)$ is \'etale over an open neighborhood of $B$, we may identify $\LC(X,B)\cong \LC(\tX,\tB)/\iota_*$ under $\pi_*$. Thus for every $v\in \LC(X,B)$ there exists a lift $v_{\alpha}\in \LC(\tX,\tB)$ for some $\alpha \in N_{\bR}$ such that $v = \pi_* v_{\alpha}$. 
		
		Let $f\in H^0(U, \cO_U)\setminus \{0\}$ be an arbitrary non-zero function. Then we can write $f = \sum_{\beta\in M/\{\pm 1\}} c_{\beta}\theta_{\beta}$ as a finite sum with $c_\beta \in \bk$. 
		Thus we have 
		\begin{align*}
			v(f)& = v_{\alpha} (\sum_{\beta\in M/\{\pm 1\}} c_{\beta}\tilde{\theta}_{\beta}) = v_{\alpha}(\sum_{\beta\in M/\{\pm 1\}} c_{\beta}(\chi^\beta + \chi^{-\beta})) \\
			& = \min \{-|\langle \alpha, \beta\rangle| \mid c_{\beta}\neq 0 \}  = \min \{ v_{\alpha}(\tilde{\theta}_{\beta})\mid c_{\beta}\neq 0\} = \min \{ v(\theta_{\beta})\mid c_{\beta}\neq 0\}.
		\end{align*}
		This shows that  $(\theta_{\beta})_{\beta\in M/\{\pm 1\}}$ is valuatively independent. 
		
		Next, we construct a different valuatively independent basis that is not a scaling  of the basis $(\theta_{\beta})$. Let $\theta_{\beta}':= \theta_{\beta} + 1$ for $\beta\in (M\setminus \{0\})/\{\pm 1\}$ and $\theta_{0}':=\theta_0 = 2$. Clearly, $(\theta_{\beta}')_{\beta\in M/\{\pm 1\}}$ form a basis of $H^0(U, \cO_U)$ as it is simply obtained by an invertible shear from $(\theta_{\beta})_{\beta\in M/\{\pm 1\}}$ via $\theta_{\beta}'= \theta_{\beta}+\frac{1}{2}\theta_0$ for $\beta\neq 0$. Moreover, for $\beta\neq 0$ we have 
		\[
		v(\theta_{\beta}') = v(\theta_{\beta}+ 1) = v_\alpha (\chi^\beta+\chi^{-\beta}+1) = \min \{-|\langle \alpha, \beta\rangle|, 0\} = -|\langle \alpha, \beta\rangle| =v(\theta_{\beta}) 
		\]
		and so 
		\begin{equation}\label{eq:cayley-cubic}
			v(\theta_{\beta}') = v(\theta_{\beta}) \leq 0 = v(\theta_0) = v(\theta_0').
		\end{equation}
		For $f = \sum_{\beta\in M/\{\pm 1\}} c_\beta  \theta_\beta$ as before, we can write
		$f = \sum_{\beta\in M/\{\pm 1\}} c_{\beta}' \theta_{\beta}'$, where  $c_{\beta}' = c_{\beta}$ whenever $\beta \neq 0$ and $c'_0 \in \bk$.
		If $f$ is a non-constant function, then there exists $\beta\neq 0$ such that $c_{\beta}= c_{\beta}'\neq 0$ and hence 
		\begin{multline*}
			v(f)  = \min \{v(\theta_{\beta})\mid c_{\beta}\neq 0 \} = \min \{v(\theta_{\beta})\mid c_{\beta}\neq 0\textrm{ and }\beta \neq 0 \}\\
			= \min \{v(\theta_{\beta}')\mid c_{\beta}'\neq 0\textrm{ and }\beta \neq 0 \} = \min \{v(\theta_{\beta}')\mid c_{\beta}'\neq 0\},
		\end{multline*}
		where the second, third, and fourth equalities use \eqref{eq:cayley-cubic}.
		If $f = c_0' \theta_0'$ is a non-zero constant, then  $v(f) = 0 = v(\theta_0')$. 
		This shows that $(\theta_{\beta}')_{\beta\in M/\{\pm 1\}}$ is also a valuatively independent basis of $H^0(U, \cO_U)$.
		Since $ \theta_{\beta}' = \theta_{\beta} + \tfrac{1}{2} \theta_0$ for $\beta \neq 0$, these two bases are not the same up to scaling elements.
	\end{proof}

	\section{Filtrations}\label{s:filts}
	In this section, we discuss filtrations of free modules over a DVR and their diagonalizability.
	The theory mirrors the vector space case in Section \ref{ss:vsfilts}.
	\medskip

	Throughout, let $R$ be a DVR with fraction field $K$, residue field $\kappa$, and a uniformizer $\pi$. 
	Let $V$ be a free $R$-module of rank $N:= {\rm rk}_R(V) < \infty$.

	\subsection{Definition}
	
	\begin{defn}\label{d:filtration}
		An $\bR$-\emph{filtration} $F$ of $V$ is the data of $R$-submodules $F^\la V \subset V$ for each $\la \in \bR$ satisfying 
		\begin{enumerate}
			\item $F^{\la} V\subset F^\mu V$ if $\la\geq \mu$,
			\item $F^\la V = \bigcap_{\mu < \la } F^\mu V$,
			\item  $F^{-\la} V =  V$ for $\la \gg0$, 
			\item $\bigcap_{\la \in \bR} F^\la V = 0$, and 
			\item $F^{\la+1} V \cap \pi V = \pi F^\la  V $.\footnote{This condition is not always used in the K-stability literature and is needed here for  Remark \ref{r:filttonorm}.2 to hold.}
		\end{enumerate}
		For $0 \neq s\in V$, we set $\ord_{F}(s) := \max \{ \la \, \vert\, s\in F^\la V\} \in \bR$ and $\ord_{F}(0)= \infty$.
	\end{defn}

	A \emph{$\bZ$-filtration}  of $V$ is an $\bR$-filtration $F$ of $V$ such that $F^\la V = F^{\lceil \la \rceil }V$ for all $\la \in \bR$. 
	More generally, for a positive integer $d$,  a   $\tfrac{1}{d} \bZ$-filtration $F$ of $V$ is an $\bR$-filtration such that  $F^{\la } V= F^{\lceil \la d\rceil /d } V$ for all $\la \in \bR$. 
	This is equivalent to the condition that $\ord_{F}(s)\in \tfrac{1}{d}\bZ$ for all $0 \neq  s\in V$.

	An $\bR$-filtration $F$ is  \emph{bounded} if $F^\la V\subset \pi V$ for $\la \gg0$. 
	Using condition (5), this is equivalent to the condition that $F^{\la+1} V = \pi F^{\la}V$ for $\la \gg 0 $.
	
	\begin{lem}\label{l:completebounded}
		If $F$ is an $\bR$-filtration of $V$ and the DVR $R$ is complete, then $F$ is bounded. 
	\end{lem}
	
	\begin{proof}
		By the definition of an $\bR$-filtration, $ (F^n V)_{n \geq 0}$ is a non-increasing chain of submodules of $V$  and $\cap_{n \geq 0} F^n V =0 $. 
		Since $R$ is complete and $V$ is finitely generated, Chevalley's Lemma (see e.g. \cite[Theorem 14.3]{24local})  implies that $F^n V \subset \pi V$ for some $n\gg 0$. Thus $F$ is bounded.
	\end{proof}
	
	\begin{lem}\label{l:boundedfiltofvs}
		If $F$ is an $\bR$-filtration of $V$ and is bounded, then setting
		\[
		\overline{F}^\la (V/\pi V) := \im (F^\la V\to V/\pi V) 
		\]
		for $\la \in \bR$
		defines a filtration of the vector space $\overline{V}:=V/\pi V$ in the sense of Definition \ref{d:filtrationvs}. 
	\end{lem}
	
	\begin{proof}
		Since $F$ is bounded, $\overline{F}^\la \overline{V}=0$ for $\la\gg0$, which is Definition \ref{d:filtrationvs}.4.
		To verify left continuity, fix $\la\in \bR$.
		By (3) and (5) in Definition \ref{d:filtration}, there exists $n>0$ such that  $\pi^n V\subset F^\la V$. 
		Since $V/\pi^n V$ is Artinian, Definition \ref{d:filtration}.2 implies that $F^{\la-\varepsilon} V = F^\la V$ for $0<\varepsilon \ll1$. 
		Therefore $\overline{F}^{\la-\varepsilon}  \overline{V}= \overline{F}^\la \overline{V}$ for $0<\varepsilon \ll1$, and so Definition \ref{d:filtrationvs}.2 is satisfied. The remaining properties in Definition \ref{d:filtrationvs} are immediate.
	\end{proof}

	\begin{rem}
		\label{r:filttonorm}
		The data of an $\bR$-filtration  $F$ of $V$ induces a  norm $\left\Vert \, \cdot \right\Vert : V \to \bR_{\geq 0}$  defined by 
		$\| s\| = {\rm exp}(- \ord_{F}(s))$,
		which satisfies 
		\begin{enumerate}
			\item[(1)]  $\| s +s' \| \leq \max\left\{ \| s\|  , \|s'\| 
			\right\}$,
			\item[(2)] $\| as\| = \| a\| \cdot \| s\|$, and
			\item[(3)] $\|s\| =0$ if and only if $s=0$
		\end{enumerate}
		for all $s, s'\in V$ and $0\neq a\in R$.
		Above, $\|a \| := {\rm exp}(- \ord_\pi(a))$, where $\ord_{\pi}(a) $ is the unique integer $n$ such that $a = u\pi^n$ for some $u\in R^\times$.
	\end{rem}
	
	\begin{exa}\label{e:filtvalDVR}
		We are primarily interested in filtrations arising from the following setup.
		Let $(X_K,B_K)$ be a projective lc CY pair  over $K$ with an lc model $(X,B)$, and 
		let $L$ be a relatively ample line bundle on $X$. 
		For any  integer $m\geq 0$, 
		\[
		V_m:= H^0(X,L^{\otimes m})
		\]
		is a finitely generated torsion-free $R$-module.
		Since $R$ is a DVR, $V_m$ is a finite rank free $R$-module. 
		
		A valuation $v \in {\rm Sk}(X_K,B_K)$ induces an $\bR$-filtration of $V_m$ defined by 
		\[
		F^\la V_m := \{s \in V_m \, \vert\, v(s) \geq \la \}.
		\]
		Note that   Definition \ref{d:filtration}.5  is satisfied, since $v(\pi)=1$ by the definition of the skeleton.
	\end{exa}

	\subsection{Diagonalization}
	
	\begin{defn}
		Let $F$ be an $\bR$-filtration of $V$. 
		A free $R$-module basis $s_1,\ldots, s_N$ for $V$
		\emph{diagonalizes}  $F$ if
		\[
		\ord_{F}(a_1 s_1 + \cdots + a_N s_N) = \min \{ \ord_{F} (a_1 s_1),\cdots , \ord_F(a_N s_N)
		\}
		\]
		for all $a_1,\ldots, a_N \in R$.
	\end{defn}
	
	\begin{rem}\label{r:diag}
		It is straightforward to check that the condition 
		that a basis $(s_1,\ldots, s_N)$ of $V$ diagonalizes $F$ is  equivalent to each of the following statements:
		\begin{enumerate}
			\item the associated norm $\| \, \cdot \, \| := {\rm exp}( - \ord_F(\, \cdot\, ))$ satisfies 
			\[
			\|  a_1 s_1 + \cdots + a_N s_N \| = \max \{ \|a_1\| \|s_1\|, \ldots, \|a_N \| \|s_N\| \}
			\]
			for any $a_1,\ldots, a_N \in R$.
			\item for each $\la \in \bR$
			\[
			F^\la V = \pi^{a_1(\la)} R s_1 + \cdots + \pi^{a_{N}(\la)}R s_N
			,\]
			where $a_i(\la):=\max\{ 0 , \lceil \la-\ord_F(s_i)\rceil \}$.
		\end{enumerate}
	\end{rem}
	
	\begin{lem}\label{l:restrictiondiagbasis}
		Let $F$ be a bounded filtration of $V$ and $\overline{F}$ be the filtration of $\overline{V} = V/\pi V $ defined in Lemma \ref{l:boundedfiltofvs}.
		If $s_1,\ldots, s_N \in V$ are elements such that their images  $\overline{s}_1,\ldots, \overline{s}_N\in \overline{V}$  form a basis for $\overline{V}$ diagonalizing $\overline{F}$ and  $\ord_{F}(s_i) \geq \ord_{\overline{F}}(\overline{s}_i)$, then $(s_1,\ldots,s_N)$ is a  basis for $V$ that diagonalizes $F$.
	\end{lem}
	
	Note that any $s\in V$ trivially satisfies that  $\ord_{F}(s) \leq \ord_{\overline{F}}(\overline{s})$.
	Thus the above hypothesis ensures that equality holds. 
	
	\begin{proof}
		Since   $(\overline{s}_1,\ldots, \overline{s}_N)$ is a basis for $\overline{V}= V/\pi V$, Nakayama's Lemma implies that $s_1,\ldots, s_N$ generate $V$ as a free $R$-module. As $V$ has rank $N$ and $R$ is a DVR, the elements form a basis. 
		
		Now consider the $\bR$-filtration $G$ of $V$ defined by 
		\[
		G^\la V := R \pi^{b_1(\la)} s_1 \oplus \cdots \oplus  R \pi^{b_N(\la)} s_N,
		\]
		where $b_i(\la) = \max\{ 0,\lceil \la-\ord_{\overline{F}}(\overline{s}_i)\rceil \}$. 
		Note that $(s_1,\ldots,s_N)$ diagonalizes $G$ by construction.
		
		We seek to show that  
		$F^\la V= G^\la V$ for all $\la \in \bR$.
		The inclusion $\supseteq $ holds as $\ord_{F}(s_i) \geq \ord_{\overline{F}}(\overline{s}_i) = \ord_{G}(s_i)$. 
		For the reverse inclusion, note that $\overline{F}^\la \overline{V} =\overline{G}^\la \overline{V}$ and $ F^\la V\supset G^\la V$.
		Thus 
		\[
		F^\la V \subset G^\la V + (F^\la V \cap \pi V)
		=
		G^\la V + \pi F^{\la-1} V.
		\]
		Applying the previous inclusion with $\la $ replaced by $\la-1$, we get
		\[
		G^\la V + \pi F^{\la-1} V
		\subset 
		G^\la V + \pi (G^{\la-1} V+ \pi F^{\la-2}V)
		=
		G^\la V + \pi^2  F^{\la-2} V.
		\]
		By iterating this $d$ times, we see that 
		\[
		F^\la V \subset G^\la V + \pi F^{\la-1} V
		\subset \cdots \subset G^\la V + \pi^d F^{\la-d}V
		.
		\]
		For $d\gg0$,  $\pi^d V \subset G^{\la}V$ and so
		\[
		F^\la V
		\subset 
		G^\la V +  \pi^d V
		\subset 
		G^\la V,
		\]
		as desired. 
		Thus $F=G$ as filtrations. 
		Since $(s_1,\ldots, s_N)$ is a diagonalizing basis for $G$, it is also a diagonalizing basis for $F$.
	\end{proof}
	
	\begin{prop}
		An $\bR$-filtration $F$ of $V$ is bounded if and only if it is diagonalizable. 
	\end{prop}
	
	\begin{proof}
		We first show the ``only if'' direction.
		By Lemma \ref{l:boundedfiltofvs}, there is an induced $\bR$-filtration $\overline{F}$ of $\overline{V}=V/\pi V$ defined by 
		\[
		\overline{F}^\la \overline{V} = \im (F^\la V\to V/\pi V)
		.
		\]
		By Remark \ref{r:flag}, there exists a basis $(t_1,\ldots, t_N)$ of $\overline{V}$ diagonalizing $\overline{F}$. 
		Let $\la^i := \ord_{\overline{F}}(t_i) \in \bR$.
		Since $t_i\in \overline{F}^{\la^i} \overline{V}$, there exists $s_i \in F^{\la^i}V $ such that $\overline{s}_i=t_i$. 
		Since $\ord_{F}(s_i) \geq \la^i = \ord_{\overline{F}}(\overline{s}_i)$,
		Lemma \ref{l:restrictiondiagbasis} implies that  $(s_1,\ldots, s_N)$ is a basis for $V$ that diagonalizes $F$.
		
		Next, we show the ``if'' direction. If $\lambda > \max_{1\leq i\leq N}\ord_F(s_i)$, then by Remark \ref{r:diag}(2) we have $F^\lambda V \subset \pi V$ as each $a_i(\lambda)>0$. Thus $F$ is bounded.
	\end{proof}

	\subsection{Multigraded modules}
	
	\begin{defn}\label{d:multgradedDVR}
		Let $F=(F_j)_{j\in J}$ be a collection of $\bR$-filtrations of $V$.
		\begin{enumerate}
			
			\item For $\la=(\la_j)_{j\in J} \in \bR^J$, we set 
			\[
			F^{\la} V := \bigcap_{j \in J} F_j^{\la_j} V \quad \text{ and } \quad F^{>\la}V := \sum_{\mu >\la }F^{\mu}V
			,\]
			where, for $\mu,\la \in \bR^J$, we write that $\mu >\la$ if $\mu_j \geq \la_j$ for all $j\in J$ and $\mu \neq \la$.
			
			\item For $0\neq s\in V$, we write 
			\[
			\ord_{F}(s) = (\ord_{F_j}(s))_{j\in J } \in \bR^J.
			\]
			
			\item The \emph{associated multigraded vector space} of $F$ is 
			\[
			{\rm gr}_{F} (V)
			:= 
			\bigoplus_{\la \in \bR^J} {\rm gr}_{F}^\la (V)
			:= 
			\bigoplus_{\la \in \bR^J} 
			F^{\la }V/ (F^{>\la} V +  ( F^\la V \cap \pi V))
			.\]
			When $J= \{1,\ldots, r\}$, we  write ${\rm gr}_{F_1,\ldots, F_r}(V)$ for ${\rm gr}_{F}(V)$.
			Using that $\pi F^\la V \subset F^{>\la} V +  ( F^\la V \cap \pi V)$, we see that ${\rm gr}_{F}(V)$ has the structure of a $\kappa$-vector space.
			
		\end{enumerate}
	\end{defn}

	\begin{defn}\label{d:multrees}
		Let $F_1,\ldots, F_r$ be a collection of $\bR$-filtrations of $V$ and $d$ be a positive integer. 
		The \emph{$d$-th multigraded Rees module} of  $F_1,\ldots, F_r$  is 
		\[
		{\rm Rees}_d(F_1,\ldots, F_r;V):=
		\bigoplus_{\la \in \bZ^r} F^{\la/d} V x^{-\la}
		\subset V[x_1^{\pm1},\ldots,x_r^{\pm1}]
		,\]
		where $x^{-\la}:=x_1^{-\la_1}\cdots x_r^{-\la_r}$ for $\la = (\la_1,\ldots, \la_r)\in \bZ^r$.
	\end{defn}
	
	\begin{rem}
		Observe that ${\rm Rees}_d(F_1,\ldots, F_r;V)$ has the structure of a 
		\[
		R[x_0,x_1,\ldots, x_r]/(x_0x_1^d\cdots x_r^d-\pi)
		\]
		module where $x_0,\ldots, x_r$ act via
		\[
		x_0 \cdot  (s x^{-\la}) = \pi s x^{-\la-(d,\ldots, d)} \quad \text{ and } \quad 
		x_1^{\mu_1}\cdots x_r^{\mu_r}  \cdot (s x^{-\la} )= sx^{-\la+\mu}
		\]
		for $\mu=(\mu_1,\ldots, \mu_r) \in \bN^r$.
	\end{rem}
	
	The $d$-th multigraded Rees module is most relevant  when $F_1,\ldots, F_r$ are all $\tfrac{1}{d}\bZ$-filtrations. 
	In this case, we can relate the above two constructions.

	\begin{lemma}
		\label{l:Reestquotient=gradedDVR}
		If $F_1,\ldots, F_r$ are $\frac{1}{d}\bZ$-filtrations of $V$, then there is an isomorphism of  $\kappa$-vector spaces
		\[
		{\rm Rees}_d(F_1,\ldots, F_r;V)/ (x_0,x_1,\ldots, x_r){\rm Rees}_d(F_1,\ldots, F_r;V) \overset{\cong}{\longrightarrow} {\rm gr}_{F_1,\ldots, F_r} V
		\]
		that sends the class of $sx^{-\la}$ in degree $-\la $ to the class of  $s$ in degree $\tfrac{1}{d}\la$.
	\end{lemma}
	
	\begin{proof}
		Observe that 
		\begin{align*}
			(x_0,x_1,\ldots, x_r) {\rm Rees}_d(F_1,\ldots, F_r;V) &= \bigoplus_{\la \in \bZ^r} \Big(\sum_{i=1}^r F^{(\la+e_i)/d}V + \pi F^{\la/d -(1,\ldots, 1)}V \Big)x^{-\la}\\
			&= 
			\bigoplus_{\la \in \bZ^r} \Big(F^{>\la/d}V + (F^{\la/d} V \cap \pi V)\Big)x^{-\la},
		\end{align*}
		where the second equality uses our assumption that each $F_i$ is a $\tfrac{1}{d} \bZ$-filtration and Definition \ref{d:filtration}.5.
		Therefore the map is well defined and an isomorphism.
	\end{proof}

	\subsection{Simultaneous Diagonalization}

	\begin{defn}\label{d:simdiag}
		A collection  of $\bR$-filtrations $(F_j)_{j \in J}$ of $V$ is said to be \emph{simultaneously diagonalizable} if there exists a free $R$-module basis $(s_1,\ldots, s_N)$ of $V$  that  diagonalizes  $F_j$ for all $j \in J$.
	\end{defn}
	
	We now give a criterion for simultaneous diagonalizability in terms of the previous multigraded modules.

	\begin{lem}\label{l:Reesmodflat->diag}
		Let $F_1,\ldots, F_r$ be  bounded $\tfrac{1}{d}\bZ$-filtrations of $V$. 
		The following are equivalent:
		\begin{enumerate}
			\item  The  $d$-th multigraded Rees module
			${\rm Rees}_d(F_1,\ldots, F_r;V)$
			is  flat  as a module over 
			\[
			A:=	R[x_0,x_1,\ldots, x_r]/(x_0x_1^d\cdots x_r^d-\pi)
			.\]
			\item We have  $\dim_{\kappa} {\rm gr}_{F_1,\ldots, F_r}(V) =\rk_R V$.
			\item  The filtrations $F_1,\ldots, F_r$ are simultaneously diagonalizable.
		\end{enumerate}
		Furthermore, if the above equivalent conditions hold, then  elements $s_1,\ldots, s_N$ of $V$ are a diagonalizing basis for $F_1,\ldots, F_r$ if and only if the classes of $s_i$ in $\gr^{\la^i}_{F_1,\ldots, F_r}(V)$ in degree $\la^i:= \ord_{F}(s_i) \in \bR^r$ are a basis for ${\rm gr}_{F_1,\ldots, F_r}(V)$.
	\end{lem}
	
	\begin{proof}
		We first verify that $M:={\rm Rees}_d(F_1,\ldots, F_r;V)$ is a finite $A$-module.
		Observe that 
		\[
		F_i^{-\la} V = V
		\quad \text{ and } \quad 
		F_i^\la V = \pi F_i^{\la-1}V
		\]for $\la \gg 0$,
		where the second equality uses the assumption that  each filtration $F_i$ is bounded.
		Since 
		\[
		x_0 \cdot  F^{\la/d}Vx^{-\la}= \pi F^{\la/d} V x^{-\la - d{\bf1}}
		\quad \text{ and } \quad 
		x_j\cdot ( F^{\la/d}V x^{-\la})= F^{\la/d} V x^{-\la+e_j}
		\]
		for all $1\leq j \leq r$,
		the $A$-module $M$ is generated by finitely many of its summands. 
		As each summand  is a submodule of a finite rank free $R$-module $V$, each summand is a finite rank free $R$-module. 
		Therefore $M$ is finitely generated as an $A$-module. 
		Also note that $M/(x_0,\ldots, x_r)M \cong {\rm gr}_{F_1,\ldots, F_r}(V)$ by Lemma \ref{l:Reestquotient=gradedDVR}.

		We now show that $(1)\implies (2)$ holds. 
		Assume that the $A$-module $M$ is flat.
		The localization satisfies
		\[
		M_{x_1\cdots x_r} 
		\cong 
		V[x_1^{\pm1},\ldots, x_r^{\pm1}],
		\]
		and hence is a free module over $A_{x_1\cdots x_r} \cong R[x_1^{\pm1},\ldots x_r^{\pm1}]$ of rank $\rk_{R} V$.
		If $M$ is flat, then $M/(x_0,\ldots,x_r) M$ is a $\kappa$-vector space of dimension $\rk_R V$. 
		Thus $\dim \gr_{F_1,\ldots, F_r}(V) = \rk_R (V)$.
		
		Next we prove $(2)\implies (3)$, as well as the reverse implication of the final assertion. 
		Assume that $N:=\dim_\kappa {\rm gr}_{F_1,\ldots, F_r}(V) = \rk_R (V)$. 
		Thus there exist  $s_1,\ldots, s_N \in V$ such that $\overline{s}_1,\ldots,  \overline{s}_N$ form a basis for $\gr_{F_1,\ldots, F_r}(V)$.
		Consider the submodule 
		\[
		M':=As_1x^{-d\la^1}+\cdots +A s_N x^{-d\la^N} \subset M,
		\]
		where $\la^i := \ord_{F}(s_i)$. 
		Since $s_1 x^{-d\la^1},\ldots, s_N x^{-d\la^N}$ generate $M/(x_0,\ldots, x_r) M$ as a $\kappa$-vector space, $M= M'+(x_0,\ldots, x_r) M$. Thus $Q:= M/M'$ satisfies that $Q= (x_0,\ldots, x_r)Q$. 
		Now Nakayama's Lemma implies that there exists $a\in (x_0,\ldots, x_r) \subset A$ such that $(1+a) Q=0$. 
		Since $Q$ is a $\bZ^r$-graded module over the $\bZ^r$-graded ring $A$, the degree $0$ component $a_{\bf 0}$ of $a$ satisfies $(1+a_{\bf 0}) Q=0$. 
		As the degree ${\bf 0}$ component $A_{\bf 0} \subset A$ satisfies $A_{\bf 0} \cap (x_0,\ldots, x_r)= \pi R$ and $R$ is a DVR, $1+a_{\bf 0}$ is invertible in $A$. 
		Thus $Q=0$, and so $M'=M$.
		Therefore, for $\la \in \bR^r$,
		\[
		F^{\la} V  = R \pi^{ a_1(\la)} s_1 + \cdots + R \pi^{a_N(\la)}s_N
		, \quad \text{ where }  \quad  a_i(\la) := \max_{1\leq j \leq r} \{ 0, \lceil \la_j - \la_j^i \rceil \}.
		\]
		For $1\leq j \leq r$, we have that 
		\[
		F_j^{\la_j} V = F^{\la_j e_j - c ({\bf 1} - e_j)} V
		\]
		for $c\gg0$, where ${\bf 1} := (1,\ldots, 1)$.
		Therefore  
		\[
		F_j^{\la_j} V = R \pi^{a_{j1}(\la)} s_1 + \cdots + R \pi^{a_{jN}(\la)} s_N,\quad \text{ where }\quad  a_{ji}(\la) := \max \{ 0 , \lceil \la_j - \la^i_j \rceil \}.
		\]
		Furthermore, by choosing $\la$ sufficiently negative in each coordinate, we see that $s_1,\ldots, s_N$ span $V$. As $N:= \rk_R V$, $s_1,\ldots, s_N$ form an $R$-module basis for $V$. 
		Thus the previous sum can be rewritten as a direct sum. 
		Thus $s_1,\ldots, s_N$ form a diagonalizing basis for $F_1,\ldots, F_r$ by Remark \ref{r:diag}. 
		This proves $(2)\implies(3)$, as well as the reverse implication in the last statement.
		
		Finally, we verify $(3)\implies (1)$ and the forward implication of the last statement. 
		If there exists a basis $s_1,\ldots, s_N$ of $V$ that diagonalizes  $F_1,\ldots, F_r$, then 
		\[
		F^{\la} V  = R \pi^{a_1(\la)} s_1 \oplus \cdots \oplus R \pi^{a_N(\la)}s_N
		, \quad \text{ where }  \quad  a_i(\la) := \max_{1\leq j \leq r} \{ 0, \lceil \la_j - \la_j^i \rceil \},
		\]
		and $\la^i := \ord_{F}(s_i) \in \bR^r$. 
		Therefore $s_1x^{-d\la^1} ,\ldots, s_N x^{-d\la^N}$
		form a free $A$-module basis for $M$.
		In particular $M$ is a flat $A$-module. 
		As well, $\overline{s}_1,\ldots, \overline{s}_N$ are a basis for  ${\rm gr}_{F_1,\ldots, F_r} (V)$. 
	\end{proof}

	\begin{lem}\label{l:uniquenessbasisDVR}
		Let $(F_j)_{j \in J}$ be a collection of $\bR$-filtrations of $V$.
		If $(s_1,\ldots, s_N)$ is a free $R$-module basis of $V$ that simultaneously diagonalizes $(F_j)_{j\in J}$, then 
		\[
		{\rm card}\{ i\in \{1,\ldots, N\} \,\vert\, \ord_{F}(s_i) = \la \} = \dim_\kappa {\rm gr}^\la_{F}(V)
		\]
		for any $\la \in \bR^J$.
		In particular, the cardinality of the set on the left does not depend on the choice of simultaneously diagonalizing basis.
	\end{lem}
	
	\begin{proof}
		If  $s_1,\ldots, s_N$ form a simultaneously diagonalizing basis for  $(F_j)_{j\in J}$, then ${\rm gr}_{F}^\la(V)$ 
		has a basis given by the elements $\overline{s}_i $ with  $\ord_{F}(s_i) = \la$. 
		Thus the relevant cardinality and the dimension agree.
	\end{proof}

	\begin{prop}\label{p:dvrdiagfinitecol}
		Assume that the DVR $R$ is complete. 
		If $(F_j)_{j\in \bN}$ is a collection of $\bR$-filtrations of $V$ such that $F_0,\ldots, F_r$ are simultaneously  diagonalizable for each $r\geq 1$, then $(F_j)_{j\in \bN}$ are simultaneously diagonalizable. 
	\end{prop}
	
	\begin{proof}
		For each $j \geq0$ and $\la \in \bR$, let $\overline{F}_j^\la \overline{V}: = {\rm im} (F_j^\la V\to \overline{V})$, where $\overline{V}:=V/\pi V$. 
		By Lemmas \ref{l:completebounded} and \ref{l:boundedfiltofvs}, $\overline{F}_j$ is a filtration of $\overline{V}$. 
		For any $r\geq 0$,  $F_0,\ldots, F_r$ are simultaneously diagonalizable by assumption. 
		By  reducing the diagonalizing basis modulo $\pi$, we see that $\overline{F}_0, \ldots, \overline{F}_r$ are simultaneously diagonalizable.
		By Proposition \ref{p:vsfinitediminfinitecol}, $(\overline{F}_j)_{j\in \bN}$ are simultaneously diagonalizable. 
		Thus there  exists a  basis $(t_1,\ldots, t_N)$ of $\overline{V}$ that  diagonalizes $\overline{F}_j$ for each $j \geq 0$.

		For each $\la= (\la_j)_{j \in \bN} \in \bR^{\bN}$, 
		we set
		$F^\lambda V:=\cap_{j\geq0}F_j^{\lambda_j}V$
		and $ \overline{F}^\lambda\overline{V} :=\cap_{j\geq0}\overline F_j^{\lambda_j}\overline{V}$.
		We now show that the quotient map $\rho:V\to \overline{V}$ satisfies  
		\begin{equation}\label{eq:filtrest}
			\rho(F^\la V) = \overline{F}^\la \overline{V}
		\end{equation}
		for each $\la \in \bR^{\bN}$.
		The inclusion $\subseteq$ holds trivially.
		To verify the reverse inclusion, note that
		\[
		\rho(\cap_{j=0}^r F_j^{\la_j} V) =  \cap_{j=0}^r \overline{F}_j^{\la_j} \overline{V}
		\]
		by a computation with a basis that diagonalizes  $F_0,\ldots, F_r$.
		Now observe that 
		\[
		\cap_{j=0}^1 F_j^{\la_j} V\supseteq \cap_{j=0}^2 F_j^{\la_j} V \supseteq \cdots 
		\]
		is a decreasing sequence of submodules of $V$ with intersection $F^\la V$. 
		Thus Chevalley's Lemma \cite[Theorem 14.3]{24local} implies that 
		\[
		\cap_{j=0}^r F_j^{\la_j} V \subset F^\la V + \pi V
		\]
		for $r\gg0$.
		Now we compute that 
		\[
		\overline{F}^\la \overline{V}
		=
		\cap_{r \geq 0}  (\cap_{j=0}^r \overline{F}_j^{\la_j} \overline{V})
		=
		\cap_{r\geq 0}\rho(\cap_{j=0}^r F_j^{\la_j} V) 
		\subset 
		\rho(F^\la V).
		\]
		Therefore  \eqref{eq:filtrest} holds. 
		
		Now set $\la^i = \ord_{\overline{F}}(t_i)\in \bR^{\bN}$. 
		By \eqref{eq:filtrest}, there exists $s_i \in F^{\la^i} V$ such that $\overline{s}_i=t_i$. 
		Since $\ord_{F}(s_i) \geq \la^i = \ord_{\overline{F}}(t_i)$, $\ord_{F_j}(s_i) \geq \ord_{\overline{F}_j}(t_i)$ for all $j\geq 0$.
		Thus Lemma \ref{l:restrictiondiagbasis}  implies that, for each $j\geq0$, $(s_1,\ldots,s_N)$ is a basis of $V$ that diagonalizes $F_j$.
		Therefore $(F_j)_{j\in \bN}$ are simultaneously diagonalizable.
	\end{proof}

	\subsection{Approximation}
	We now analyze the behavior of diagonalizability under approximation. 
	\medskip

	\begin{defn}
		Let $F$ be an $\bR$-filtration of $V$. 
		\begin{enumerate}
			\item For each $n\geq 1$ and $\la \in \bR$, we write $F^\la (V/\pi^nV) := \im (F^\la V\to V/\pi^nV) \subset V/\pi^nV$. 
			
			\item We say that a free $R$-module basis $s_1,\ldots,s_N$ of $V$ diagonalizes $F$ modulo $\pi^n$ if there exist $c_1,\ldots , c_N\in \bR $ such that 
			\begin{equation}\label{e:diagordern}
				F^\la (V/\pi^nV) =   \pi^{a_1(\la)} 
				\overline{s}_1 
				(R/\pi^n R) 
				\oplus \cdots \oplus  \pi^{a_N(\la) } \overline{s}_N (R/\pi^n R) 
				\subset (V/\pi^nV) 
			\end{equation}
			for each $\la \in \bR$, where  $a_{i}(\la) = \max\{ 0 , \lceil \la  - c_i \rceil \}$.
		\end{enumerate}
	\end{defn}
	
	Above $\overline{s}_i$ denotes the images of $s_i$ in $(V/\pi^nV) $ and the direct sum is taken in the category of $R/\pi^n R$-modules.

	\begin{defn}
		A collection of $\bR$-filtrations  $(F_j)_{j\in J}$ of $V$ is said to be \emph{simultaneously diagonalizable} modulo $\pi^n$ if there exists a free $R$-module basis $s_1,\ldots, s_N$ for $V$ such that the basis diagonalizes $F_j$ modulo $\pi^n$ for all $j\in J$.
	\end{defn}
	
	Below, we prove that simultaneous diagonalizability can be characterized by a sufficiently large approximation.
	
	\begin{prop}\label{p:diagordn}
		Let $F_1,\ldots, F_r$ be bounded $\bR$-filtrations of $V$. 
		\begin{enumerate}
			
			\item If $F_1,\ldots, F_r$ are simultaneously diagonalizable modulo $\pi^n$ for all integers $n\geq 1$, then $F_1,\ldots, F_r$ are simultaneously diagonalizable. 
			
			\item Let $n\geq 1$ be an integer such that $F_j^0 V = V $ and $F_j^n V\subset \pi V$ for all $1\leq j \leq r$.
			If $F_1,\ldots, F_r$ are simultaneously diagonalizable modulo $\pi^n$, then $F_1,\ldots, F_r$  are simultaneously  diagonalizable.
		\end{enumerate}
	\end{prop}
	
	\begin{proof}
		Since  $F_1,\ldots, F_r$ are bounded, 
		there exists an integer $c$ such that $F^c_j V \subset \pi V$ for all $j$. 
		Furthermore, by the definition of a filtration, there exists an integer $b $ such that $F^b_j V = V$ for all $j$.
		Now set $n:=c-b$.
		To prove (1) and (2), it suffices to prove that if $F_1,\ldots, F_r$ are simultaneously diagonalizable modulo $\pi^n$, then $F_1,\ldots, F_r$ are simultaneously diagonalizable. 
		Indeed, in case (2), we may set $b=0$ and $c=n$.
		
		Now assume that $F_1,\ldots, F_r$ are 
		simultaneously diagonalizable modulo $\pi^n$. Thus there exists an $R$-module basis  $s_1,\ldots, s_N$ of $V$ that diagonalizes $F_1,\ldots,F_r $ modulo $\pi^n$.
		Let $t_1,\ldots, t_N$ denote their images under the map $V\to V/\pi^nV$.
		Thus there exist real numbers $c_{ij}$ such that 
		\[
		F_j^\la (V/\pi^nV) = (R/\pi^n R) \pi^{a_{1j}(\la)} t_1 \oplus \cdots \oplus (R/\pi^nR) \pi^{a_{Nj}(\la)} t_N
		\]
		for each $\la \in \bR$ and $j\in \{1,\ldots, r\}$, where $a_{ij}(\la) = \max \{ 0, \lceil \la - c_{ij} \rceil \}$.
		Since $t_i \in F_j^{c_{ij}}(V/\pi^nV)$ and 
		\[
		F_j^{c} (V/\pi^n V)\subset \pi (V/\pi^n V),\]
		$c_{ij} \leq c$ holds.
		As well, by reducing modulo $\pi$, we see that 
		$c_{ij}= \ord_{\overline{F}_j}(\overline{s}_i)$, where $\overline{F}_j$ is the induced filtration of $\overline{V}:= V/\pi V$ in Lemma \ref{l:boundedfiltofvs}.
		
		We now show that $\ord_{F_j}(s_i) \geq c_{ij}$ for each $1\leq i \leq N$ and $1\leq j \leq r$.
		Since $t_i \in F_j^{c_{ij}} (V/\pi^nV)$, there exists $s'_{ij} \in F_j^{c_{ij}} V$ such that $s_{i} - s'_{ij} \in \pi^{n} V$. In particular, $s_i - s'_{ij} = \pi^{n} r_{ij}$ for some $r_{ij}\in V$. 
		Now note that 
		\[
		\pi^{n} V = \pi^{n} F_j^b V
		\subset F_j^{n+b}V
		= 
		F_j^{c}V.
		\]
		Therefore
		\[
		s_i  =s'_{ij} + \pi^{n} r_{ij}
		\in 
		F_j^{c_{ij}}V + \pi^{n} V
		\subset 
		F_j^{c_{ij}}V + F_j^{c} V
		\subset 
		F_j^{c_{ij}}V .
		\]
		Thus $\ord_{F_j}(s_i)\geq c_{ij}$ as desired. 
		
		Since $\ord_{F_j}(s_i) \geq c_{ij} = \ord_{\overline{F}_j}(\overline{s}_i)$, Lemma \ref{l:restrictiondiagbasis} implies that $(s_1,\ldots, s_N)$ diagonalizes $F_j$. 
		As this holds for each $j\in \{1,\ldots, r\}$, $F_1,\ldots, F_r$ are simultaneously diagonalizable as desired.
	\end{proof}

	\section{Degenerations of boundary polarized CY pairs}\label{s:bpcys}

	In this section, we discuss degenerations of boundary polarized CY pairs over DVRs and certain multi-degenerations that relate them.
	\medskip
	
	Throughout, let $R$ be a DVR essentially of finite type over $\bk$.\footnote{The assumption that $R$ is essentially of finite type is needed to use results from the lc MMP  in \cite{HX13}.}
	We write $K:= {\rm Frac}(R)$ for the fraction field, $\pi \in R$ for a fixed uniformizer, and $0 \in \Spec(R)$ for the closed point.
	
	\subsection{Boundary polarized CY pairs over DVRs}

	\begin{defn}
		A \emph{boundary polarized CY pair $(X,B+D)$  over $R$}  is the data of 
		a scheme $X$  that is projective over $R$ and  effective $\bQ$-divisors $B$ and $D$ on $X$ such that
		\begin{enumerate}
			\item $(X,B+D+X_{0,\red})$ is an lc pair, 
			\item $K_{X}+B+D+X_{0,\red}\sim_{\bQ}0$, and 
			\item  $D$ is relatively ample. 
		\end{enumerate}
	\end{defn}

	\begin{rem}
		If $(X,B+D)$ is a boundary polarized CY pair over $R$ and $X_0$ is reduced, then Proposition \ref{p:kollarsmoothslcfam} implies that $(X,B+D) \to \Spec\,R$ is a family of slc pairs. Thus $(X, B+D) \to \Spec\,R$ is a family of boundary polarized CY pairs. 
	\end{rem}

	Extensions of a boundary polarized CY pair over $K$ to such pairs over $R$   are related to divisorial points in the essential skeleton. 
	
	\begin{prop}\label{p:EinSKtodegen}
		If $(X_K,B_K+D_K)$ is an lc boundary polarized CY pair over $K$ and $v \in {\rm Sk}(X_K,B_K+D_K)$ is a divisorial valuation, then there exists a boundary polarized CY pair $(X,B+D)$ over $R$ with an isomorphism 
		\[
		(X_K,B_K+D_K)\cong (X,B+D)\times_R K
		\]
		over $K$ such that $E:=X_{0,\red}$ is a prime divisor  and  
		$ v = (b_E)^{-1} \ord_E$ with $b_E={\rm coeff}_{E}( X_0)$.
	\end{prop}

	\begin{proof}
		By Proposition \ref{lem:semistable-reduction}, $(X_K,B_K+D_K)$ extends to a projective lc CY pair $(X^{\lc},B^{\lc}+ D^{\rm lc} +X_{0,\red}^{\lc})$ over $R$.
		Write $v= (\ord_E(\pi))^{-1} \ord_E$, where $E$ is a divisor over $X^{\lc}$.
		Since $\ord_E$ has log discrepancy $0$ along $(X^{\rm lc},B^{\lc}+D^{\lc}+X^{\rm lc}_{0,\red})$, Proposition \ref{p:dltmodexist} implies that there exists a  $\bQ$-factorial dlt modification
		\[
		f:(Y,B_Y+D_Y+Y_{0,\red}) \to (X^{\rm lc}, B^{\rm lc} +D^{\rm lc}+ X^{\lc}_{0,\red})
		\]
		such that $E$ appears as a prime divisor on $Y$, $B_Y:= f_*^{-1} B^{\lc}$, and $D_Y$ is the $\bQ$-divisor on $Y$ such that the above morphism of pairs is crepant.

		Since $Y$ is $\bQ$-factorial and $(Y,B_Y+ D_Y+Y_{0,\red})$ is a dlt CY pair over $R$, and 
		\[
		K_{Y}+B_Y+D_Y+Y_{0,\red}-E\sim_{\bQ} -E\sim_{\bQ}
		(b_E)^{-1}Y_0
		- E
		\geq 0
		,
		\]
		\cite[Theorem 1.6]{HX13} implies that a 
		$(K_{Y}+B_Y+D_Y+Y_{0,\red}-E)$
		-MMP with scaling terminates and produces a birational contraction to a good minimal model
		\[
		g:Y \dashrightarrow Z.
		\]
		Write $B_Z:= g_* B_Y$, $D_{Z}:=g_* D_Y$, and $E_Z:= g_* E$.
		Since the above MMP is also a $(Y_0
		- b_E E )$-MMP, $g$ can only contract divisors in $\Supp(Y_0-b_E E)$. 
		Furthermore, since $-E_Z$ is nef over $R$,  the MMP must contract all divisors in $\Supp(Y_0- b_E E)$ as otherwise, by the connectedness of $Z_0$, there would exist a curve $C\subset Z_0$  such that $C\cdot E_Z>0$.
		Therefore $E_Z=Z_{0,\red}$. 
		
		We now seek to construct a model on which $D_Z$ is ample. 
		Since 
		\[
		(Y, B_Y+D_Y+Y_{0,\red} )
		\]
		is an lc CY pair and $g$ is a birational contraction, \[
		(Z,B_Z+D_Z+Z_{0,\red})
		\]
		is an lc CY pair.
		Since $D_{K}$ is ample and $(X_K,B_K+D_K)$ is lc, Bertini's Theorem implies that there exists an effective $\bQ$-divisor $H_K \sim_{\bQ}D_K$ such that $(X_K,B_K+D_K+H_K)$ is lc. In particular, $\Supp(H_K)$ does not contain lc centers of $(X_K,B_K+D_K)$. 
		Let $H_Y$ denote the closure of  $f_K^* H_K $ in $Y$ and $H_{Z} := g_* H_Y$.
		Thus 
		\[
		(Y, B_Y+D_{Y}+ \varepsilon H_{Y}+Y_{0,\red})
		\]
		is dlt away from $Y_0$ for $0<\varepsilon <1$.
		Since $g$ is an isomorphism away from $Y_0$,
		\[
		(Z,  B_Z+D_{Z}+ \varepsilon H_{Z} +Z_{0,\red})
		\]
		is dlt away from $Z_0$ for $0<\varepsilon<1$. 
		As $(Z, B_Z+D_{Z} +Z_{0,\red})$ is lc, $(Z,B_Z+D_Z)$ has no lc centers contained in $Z_0$. 
		Since $\Supp(H_Z)$ contains no lc centers of $(Z,B_Z+D_Z)$, there exists $0<\varepsilon \ll1$ such that
		\[
		(Z, B_Z+D_{Z} +\varepsilon H_{Z} )
		\]
		is dlt and has no lc centers contained in $Z_0$.
		Note that there is an isomorphism  and a crepant morphism 
		\[
		(Z, B_Z+D_{Z}+ \varepsilon H_{Z})_K
		\cong (Y, B_Y+D_{Y}+ \varepsilon H_Y)_K
		\to 
		(X_K,B_K+ D_K+ \varepsilon H_K)
		,\]
		where the rightmost pair is a canonical model.
		Therefore \cite[Theorem 1.1]{HX13} implies that there exists a birational contraction
		\[
		h:Z\dashrightarrow X
		\]
		to a $K_{Z} +B_Z+ D_Z+ \varepsilon H_Z$-canonical model over $R$. 
		Set $B:= h_*B_Z$, $D := h_* D_{Z}$ and $H:= h_* H_{Z}$. 
		As $Z_K \to X_K$ is the morphism to a canonical  model, the birational map 
		\[
		(X^{\rm lc},B^{\lc}+D^{\lc})\dashrightarrow (X,B+D)
		\]
		is an isomorphism over $K$. 
		As $Z_0 = b_E E_Z$, which is irreducible, $Z\dashrightarrow X$ does not  contract $E_Z$. 
		Thus $X_{0,\red}$ is irreducible and  $(b_E)^{-1}\ord_{X_{0,\red}} = v$.
		
		It remains to verify that $(X,B+D)$ is a boundary polarized CY pair over $R$. Since 
		\[
		(Z,B_Z+D_Z+Z_{0,\red})
		\]
		is an lc CY pair and $h$ is a birational contraction, 
		\[
		(X,B+D+X_{0,\red})
		\]
		is an lc CY pair. 
		Since $H\sim_{\bQ} D$ over $X\setminus X_0$ and  $X_0$ is a principal divisor with $X_{0,\red}$ prime,  $H\sim_{\bQ} D$. 
		Thus 
		\[
		\varepsilon D\sim_{\bQ}
		h_* (K_Z+B_Z+D_Z + \varepsilon H_Z)
		.\]
		Since the last term is relatively ample by construction,  $D$ is relatively ample. 
	\end{proof}

	\subsection{Filtrations}\label{sss:bpcyfiltrations}
	
	\begin{defn}\label{d:filtbpcy/R}
		Let $(X,B+D)$ and $(X',B'+D')$ be boundary polarized CY pairs over $R$ with an isomorphism 
		\[
		(X_K,B_K+D_K) \cong (X'_K,B'_K+D'_K)
		\]
		over $K$.
		For each integer $m\geq 0$ such that $mD_K$ is integral, we define an $\bR$-filtration $F$  of $H^0(X,mD)$ by 
		\[
		F^\la H^0(X,mD) := H^0(X,mD) \cap H^0(X',mD'-  \la X'_0    )
		\]
		for $\la \in \bR$, 
		where the intersection is taken by viewing the terms as submodules of $H^0(X_K,mD_K)$.
	\end{defn}
	
	This filtration is an analogue of the filtration induced by a test configuration in Definition \ref{d:testconfigfilt}.  As in the test configuration case, it has a valuative interpretation.

	\begin{lem}\label{l:inducedfilt}
		Keep the notation from Definition \ref{d:filtbpcy/R}. The $\bR$-filtration $F$ satisfies
		\[
		F^\la H^0(X,mD)
		=\bigcap _{E\subset X'_0}\{ s\in H^0(X,mD)\, \vert\, (b_E)^{-1}\ord_{E}(s)\geq \la \},
		\]
		where $b_E := {\rm coeff}_E(X'_0)$.
		Furthermore, the filtration is bounded and a $\tfrac{1}{d}\bZ$-filtration, where $d:={\rm lcm}(b_E \, \vert\, E\subset X'_0)$.
	\end{lem}
	
	In the above formula,  $\ord_{E}(s)$ denotes the vanishing of $s$ along $E$, where we view $s\in K(X)$. 
	
	\begin{proof}
		Fix $s\in H^0(X,mD)$, which we view as an element of $K(X)$. 
		We have that $s\in F^\la H^0(X,mD)$ if and only if 
		${\rm div}_{X'}(s)+mD' -  \la X'_0 \geq 0 $. 
		As $s_K \in H^0(X_K,mD_K)$, 
		the latter holds if and only if 
		\[
		\ord_{E}(s)  - {{\rm coeff}_{E}}( \la X'_0 ) \geq 0
		\]
		for each irreducible component $E\subset X'_0$. 
		This translates to the condition that  $\ord_{E}(s) \geq  b_E \la$ for all $E\subset X_0'$. Therefore the equality holds.
		
		The equality  implies that  $F$ is a $\tfrac{1}{d}\bZ$-filtration.
		To prove boundedness, note that 
		$V:=H^0(X,mD)$ and $V':=H^0(X',mD')$ are both finitely generated $R$-submodules of $V_{K} := H^0(X_K,mD_K)$ satisfying
		\[
		K \cdot V = V_K=K \cdot V'
		.\]
		Hence there exists an integer $n\geq 0$ such  that
		$\pi^n V'\subset V$.
		If $\la\geq n+1$,  then
		\[
		F^\la V \subset F^{n+1} V 
		\subset V\cap \pi^{n+1} V' 
		\subset \pi V.
		\]
		Therefore $F$ is bounded.
	\end{proof}

	\subsection{Base change}\label{sss:basechange}
	We now  analyze the normalized base change of a boundary polarized CY pair defined by taking an extension of $R$ that adjoins a root of a uniformizer.
	Such covers  are needed to make the fiber over $0$ reduced.

	Fix an integer $d>0$. We set 
	\[
	\tR:= R[\pi^{1/d}] = R \oplus  R \pi^{1/d}\oplus \cdots  \oplus R\pi^{(d-1)/d}
	,\]
	which is a DVR essentially of finite type over $\bk$, and $\tK := {\rm Frac}(\tR)= K [\pi^{1/d}]$.
	There is a natural $\mu_d$-action on $\Spec\,\tR$ such that $\xi \in \mu_d(\bk)\subset \bk^\times$ acts on $ R\pi^{i/d}$ as multiplication by $\xi^i$.
	
	Given an lc boundary polarized CY pair $(X_K,B_K+D_K)$ over $K$, we write 
	\[
	(X_{\tK},B_{\tK}+D_{\tK}):= (X_{K},B_K+D_K)\times_K  \tK
	,\]
	which is an lc boundary polarized CY pair over $\tK$ with a $\mu_d$-action. 
	By the flat base change theorem, for each integer $m\geq0$ there are natural $\mu_d$-equivariant isomorphisms
	\begin{equation}\label{e:isomH^0mD_K}
		H^0(X_{\tK}, mD_{\tK}) \cong 
		H^0(X_{K},mD_{K}) \otimes_{K} \tK 
		\cong 
		\bigoplus_{i=0}^{d-1} \pi^{i/d} H^0(X_K,mD_K),
	\end{equation}
	where $\xi \in \mu_d(\bk)\subset \bk^\times$ acts on the summand 
	$\pi^{i/d} H^0(X_K,mD_K) $ by multiplication by $\xi^i$.

	Given an lc boundary polarized CY pair $(X,B+D)$ over $R$, we define an lc boundary polarized CY pair $(\tX,\tB+\tD)$ over $\widetilde{R}$ as follows.
	Let $\rho$
	denote  the composition 
	\[
	\begin{tikzcd}
		\tX\arrow[r,"n"]\arrow[rr,bend left,"\rho"] & X\times_ R \tR \arrow[r]& X 
	\end{tikzcd}
	\]
	of the base change and the normalization morphism.
	We define $\bQ$-divisors 
	by 
	\[
	\tB := \rho^* B \quad \text{ and } \quad 
	\tD:= \rho^* D
	\]
	as in \cite[2.40]{Kol13}.
	Since $\mu_d \times \tX$ is normal, the $\mu_d$-action on $X\times_R \tR$ lifts to a $\mu_d$-action on $\tX$ that fixes $\tB$ and $\tD$.
	We call $(\tX,\tB+\tD)$ the base change of $(X,B+D)$ with respect to $R\hookrightarrow \tR$.
	
	\begin{lem}\label{l:basechangebpcy}
		If $(X,B+D)$ is a boundary polarized CY  pair over $R$, then the base change $(\tX,\tB+\tD)$  is a boundary polarized CY pair over $\tR$.
		Furthermore, 
		\begin{enumerate}
			\item the isomorphism  \eqref{e:isomH^0mD_K} restricts to an  isomorphism
			\[
			H^0(\tX, m\tD) \cong  \bigoplus_{i=0}^{d-1}  \pi^{i/d} H^0(X,mD+ (i/d) X_0)
			\]	
			that is  $\mu_d$-equivariant, where $\xi \in \mu_d(\bk)\subset\bk^\times$ acts on the right as multiplication by $\xi^i$ on the $i$-th summand.
			
			\item If all coefficients of $X_0$ divide $d$, then  $\tX_0$ is reduced. In particular, $(\tX,\tB+\tD)\to \Spec\,\tR$ is a family of boundary polarized CY pairs. 
		\end{enumerate}
	\end{lem}
	
	\begin{proof} 
		Since $(X,B+D+X_{0,\red})$ is an lc  CY pair over $R$ and 
		\[
		K_{\tX}+\tB+\tD+\tX_{0,\red}
		= 
		\rho^*(K_X+B+D+X_{0,\red})
		,\]
		$(\tX,\tB+\tD+\tX_{0,\red})$ is lc by \cite[Corollary 2.43]{Kol13} and CY over $\tR$.
		Since $\rho$ is finite and $D$ is ample over $R$, $\tD=\rho^*D$ is ample over $\tR$. Thus $(\tX,\tB+\tD)$ is a boundary polarized CY pair over $\tR$.
		
		For statement (1), fix $s=\sum_{i=0}^{d-1} \pi^{i/d}s_i$, where $s_i \in H^0(X_K,mD_K)$. Using the isomorphism \eqref{e:isomH^0mD_K}, we may view $s$ as an element of $H^0(\tX_{\tK},m\tD_{\tK})$. 
		By the $\mu_d$-action on $H^0(\tX,m\tD)$, $s\in H^0(\tX,m\tD)$ if and only if $\pi^{i/d}s_i \in H^0(\tX,m\tD)$ for all $0 \leq i \leq d-1$.
		Since 
		\[
		{\rm div}_{\tX}(\pi^{i/d}s_i) +m \tD= i\tX_0+ \rho^*( {\rm div}_{X}(s_i) +mD)
		\]
		and $\rho^* X_0 = d \tX_0$, we see that $\pi^{i/d} s_i\in H^0(\tX,m\tD)$ if and only if $
		{\rm div}_{X}(s_i)+mD + (i/d) X_0 \geq 0$, which is equivalent to the condition that  $s_i \in H^0(X,mD+(i/d)X_0)$. 
		Thus the isomorphism in (1) holds. Statement (2) follows immediately from \cite[Lemma 2.53]{Kol23}.
	\end{proof}

	We now analyze the filtration in Section \ref{sss:bpcyfiltrations} under base change.

	\begin{lem}\label{l:basechangefilt}
		Let  $(X,B+D)$ and $(X',B'+D')$  be  boundary polarized CY pairs over $R$ with an isomorphism 
		$(X_K,B_K+D_K)\cong (X'_K,B'_K+ D'_K)$
		over $K$.
		Write 
		\[
		(\tX,\tB+\tD)\quad \text{and } \quad (\tX',\tB'+\tD')
		\]
		for their base changes via $R\hookrightarrow \tR := R[\pi^{1/d}]$. Write  $F$ and $\tF$ for the filtrations of 
		$H^0(X,mD)$ and $H^0(\tX,m\tD)$
		induced by $(X',B'+D')$ and $(\tX',\tB'+\tD')$.
		Then the induced filtration $\tF$ of $H^0(\tX,m\tD)$ is $\mu_d$-invariant and the isomorphism  in Lemma \ref{l:basechangebpcy} sends the weight 0-part of $\tF^\la H^0(\tX,m\tD)$ to  $F^{\la/d} H^0(X,mD)$. 
	\end{lem}
	
	\begin{proof}
		Since the $\mu_d$-action on $(\tX',\tB'+\tD')$  agrees with the $\mu_d$-action on $(\tX,\tB+\tD)$ over $K$, the filtration $\tF$ of $H^0(\tX,m\tD)$ is $\mu_d$-invariant. Thus there is a decomposition into weight spaces
		\[
		\tF^{\la} H^0(\tX,m\tD) = \bigoplus_{i=0}^{d-1} (\tF^\la H^0(\tX,m\tD)\cap H^0(\tX,m\tD)_i ).
		\]
		Now fix an element  $\widetilde{s} \in H^0(\tX,m\tD)$ of pure weight 0. 
		By   Lemma \ref{l:basechangebpcy}.1, there exists $s\in H^0(X,mD)$ such that $\widetilde{s}= \rho^*s$. 
		Now $\widetilde{s} \in \tF^\la H^0(\tX, m\tD)$ if and only if 
		\[
		{\rm div}_{\tX'}(\widetilde{s})  + m \tD' \geq \la \tX'_0
		.\] 
		Since ${\rm div}_{\tX'} (\tilde{s}) +m\tD' =\rho'^*( {\rm div}_{X'}(s)+mD')$ and $\rho'^*(X'_0 )= d \tX'_0$, where $\rho': \tX' \to X' $ is the composition of the projection and normalization, 
		the latter holds if and only if  
		\[
		{\rm div}_{X'}(s)  +mD'\geq (\la/d) X'_0
		,\] 
		which is equivalent to the condition that $s\in F^{\la/d} H^0(X,mD)$ as desired. 
	\end{proof}
	
	\subsection{Higher rank degenerations}
	We now construct certain multi-degenerations that relate a collection of extensions of boundary polarized CY pairs over $R$ that agree over $K$.
	
	For an integer $r\geq 1$,  let
	\[
	\mathrm{S}^r
	:= \Spec R[x_0, \ldots, x_r]/(x_0 x_1 \cdots x_r - \pi)
	,\]
	which admits a $\bT:= \bG_m^r$-action that fixes $R$ and acts on $x_0$ with weight $(-1,\ldots, -1)$ and $x_i$ with weight $e_i$ for $1\leq i \leq r$, where $e_i\in \bZ^r$ is the $i$-th unit vector.
	For $0\leq i \leq r$, we set
	\[
	U_i:= D(x_0\cdots \widehat{x}_{i} \cdots x_r) \cong \Spec(R[x_0^{\pm1},\ldots, \widehat{x}_{i}^{\pm1}, \ldots, x_r^{\pm1}]) \cong \Spec(R)\times \bT
	\] 
	and 
	\[
	U:= D(x_0 \cdots x_r) 
	\cong \Spec(K[x_1^{\pm1},\ldots, x_r^{\pm1}])
	\cong \Spec(K)\times \bT
	.\]
	Note that $U= U_i \cap U_j$ for $i\neq j$.

	\begin{thm}\label{t:highrankScomplete}
		Fix a collection of boundary polarized CY pairs over $R$
		\[
		(X^0,B^0+D^0),\ldots, (X^r, B^r+D^r)
		\]
		such that each $X^i_0$ is reduced and there is  a fixed isomorphism
		\[
		(X^0_K,B^0_K+D^0_K) \cong (X^i_K,B^i_K+D^i_K),
		\]
		over $K$ for each $1\leq i \leq r$.
		Then there exists a $\bT:= \bG_m^r$-equivariant family of boundary polarized CY pairs 
		\[
		(\cX,\cB+\cD) \to \mathrm{S}^r
		\]
		with a $\bT$-equivariant isomorphism 
		\[
		\phi_i:(\cX,\cB+\cD)_{U_i}\to (X^i,B^i+ D^i)\times_R U_i 
		\]
		over $U_i$ for each $0\leq i \leq r$ such that the birational map 
		$\phi_i \circ \phi_{0}^{-1}$
		is the product of the  birational map $X^0\dashrightarrow X^i$ and the birational map $U_0 \dashrightarrow U_i$ over $R$.
	\end{thm}
	
	The proof follows the same strategy as the proof of  Proposition \ref{p:highranktheta}.
	
	\begin{proof}
		To simplify notation, let $S:= \mathrm{S}^r$ and  $(X,B+D) := (X^0,B^0+D^0)$. 
		Since $X_0$ is reduced, $(X,B+D) \to \Spec\,R$ is a family of boundary polarized CY pairs.
		Consider the $\bT$-equivariant family of boundary polarized CY pairs
		\[
		(X_S,B_S+D_S)  = (X,B+D)\times_{R}S \to S.
		\]	
		Note that over $U_0$ the family has the desired form. 
		\medskip

		\emph{Step 1. We  construct an appropriate dlt modification $f:Y \to X_S$ .}
		We first describe the  divisors over $X_S$ that need to be on $\cX$.
		The  isomorphism  $(X_K,B_K+D_K)\cong (X_K^i,B_K^i+ D_K^i)$ induces a birational map
		\[
		(X,B+D) \dashrightarrow (X^i,B^i+ D^i)
		\]
		over $R$. 
		Let $X^i_{01} ,\ldots, X^i_{0n_{i}}$ 
		denote the irreducible components of $X^i_0$.
		These are prime divisors in $X^i$ and, hence,  induce prime  divisors over $X$ with center contained in $X_0$.
		Let $E_{ij}:= X^i_{0j}\times \bT$.
		Using the rational maps
		\[
		X^i \times \bT 
		\cong 
		X^i \times_R U_i 
		\dashrightarrow X\times_R U_i 
		\hookrightarrow X_{S} 
		,
		\]
		we may  view each $E_{ij}$ as a divisor over $X_S$. 
		Note that 
		\begin{multline*}
			A_{X_S,B_S+D_S+ V(x_0\cdots x_r)}(E_{ij})
			= 
			A_{X\times \bT,B\times \bT+ D\times \bT+ X_0 \times \bT}(E_{ij})\\
			=
			A_{X,B+D+X_0}(X^i_{0j})
			=
			A_{X^i,B^i+D^i+X^i_0}(X^i_{0j})
			=
			0,
		\end{multline*}
		where the third equality holds by \cite[Lemma 2.10]{ABB+}.

		Since $(X_S,B_S+D_S) \to S$ is a family of slc pairs,
		Propositions \ref{p:kollarsmoothslcfam}  and \ref{p:kollarsmoothslcfamgen} imply that 
		\[
		(X_S,B_S+D_S+V(x_0 \cdots x_r))
		\]
		is lc and the lc centers of $(X_S,B_S+D_S)$ dominate $S$.
		Now \cite[Corollary 1.38]{Kol13} implies that there exists a $\bQ$-factorial dlt modification 
		\[
		f:(Y, B_Y+D_Y+Y_{0,\red}) \to (X_S, B_S+D_S +V_{X_S}(x_0 \cdots x_r))
		,\]
		such that $B_Y:= f^{-1}_* B_S$ and  each $E_{ij}$ is a prime divisor on $Y$.\footnote{Note that $Y_0 = V_Y(\pi=0) = V_{Y}(x_0\cdots x_r) $. To simplify notation, we usually write $Y_0$ for this divisor.}
		Furthermore, by choosing the log resolution in the proof of \cite[Corollary 1.38]{Kol13} to be $\bT$-equivariant and using that $\bT$ is a connected group scheme to see that the MMP in the proof is necessarily $\bT$-equivariant, we may assume that the  $\bT$-action on $X_S$ extends to a $\bT$-action on $Y$.
		
		Let $E:=\sum_{i=0}^r \sum_{j=1}^{n_i} E_{ij}$.
		Note that   
		\[
		D_Y  =f^*D_S+G - \textstyle \sum_{i,j} a_{ij}E_{ij},
		\]
		where $a_{ij}:=A_{X_S,B_S +V_{X_S}(x_0\cdots x_r)}(E_{ij})$
		and
		$\Supp(G)$ is a union of exceptional divisors that either dominate $S$ or are contained in $\Supp(Y_0-E)$. 
		\medskip
		
		\emph{Step 2: We now construct a birational contraction $g: Y \dashrightarrow Z$ that will contract $\Supp(Y_0 -E)$.}
		Note that  
		\[
		(Y,B_Y +D_Y+Y_{0,\red})
		\]
		is dlt,
		\[
		K_{Y}+B_Y+D_{Y}+Y_{0,\red} \sim_{\bQ,S} 0,\]
		and 
		\[
		K_{Y}+B_Y+D_{Y}+Y_{0,\red} -E \sim_{\bQ,S} -E \sim_{\bQ,S} Y_0 -E,
		\]
		which is effective. 
		Thus  \cite[Theorem 1.6]{HX13} implies that there exists a   $-E$-MMP over $S$ that terminates with a good minimal model. Write 
		\[
		g:Y \dashrightarrow Z
		\]
		for the birational contraction to the minimal model, $B_Z:= g_* B_Y$,  $D_{Z} = g_* (D_Y)$, and $E_Z : = g_* E$. 
		Thus  $K_{Z}+B_Z+D_{Z} +Z_{0,\red}-E_Z$ is nef. 
		Since $(Y,B_Y+ D_{Y}+Y_{0,\red} )$ is lc and CY over $S$, $(Z,B_Z+D_Z+Z_{0,\red})$ is lc and CY over $S$.
		
		We claim that $g$ contracts precisely the prime divisors in the support of $Y_{0} - E$.
		Indeed, since  $-E \sim_{\bQ} Y_0 - E$ and  the latter divisor is effective, $g$ only contracts divisors contained in $Y_0 -E$.
		Since $-E_Z$ is nef over $S$ and the map $\Supp(E_Z) \to V_{S}(x_0\cdots x_r)$ is surjective, all such divisors must get contracted. Indeed, otherwise by the connectedness of the fibers of $Z\to S$, there would exist a curve $C\subset Z_0= V_{Z}(x_0\cdots x_r)$ contracted by $Z\to S$ such that $C $ is not contained in $E_Z$ but intersects  it. Hence $C\cdot E_Z >0$, which contradicts the relative nefness of $-E_Z$.
		As a consequence 
		\[
		E_Z=  V_Z(x_0 \cdots x_r)_{\red}  = V_{Z}(x_0\cdots x_r) = Z_0,
		\]
		where the second equality uses that  $\ord_{E_{ij}}(x_0\cdots x_r) = \ord_{E_{ij}}(\pi)=1$.
		\medskip

		\emph{Step 3: We now construct $h: Z\dashrightarrow \cX$.}
		Since $D$ is ample and $(X,B+D+X_0)$ is lc, there exists a $\bQ$-divisor $0 \leq H\sim_{\bQ}D$ such that $(X,B+D+H+X_{0})$  is lc. 
		Thus $(X,B+D+H)\to \Spec\,R$ is a family of slc pairs. 
		Therefore  $(X_S,B_S+D_S+H_S)\to S$ is a family of slc pairs, where $H_{S}:= H\times_R S$. 
		Thus  \cite[Theorem 4.54]{Kol23} implies that
		\[
		(X_S, B_S+D_S+ H_S +V_{X_S}(x_0\cdots x_r))
		\]
		is lc. In particular,  $\Supp(H_S)$ contains no lc centers of $(X_S,B_S+D_S+ V_{X_S}(x_0\cdots x_r))$.
		Therefore 
		\[
		(Y, B_Y+D_{Y} +  \varepsilon f^*H_{S}+Y_{0,\red} ) \to (X_S,B_S+D_S+\varepsilon H_S+V_{X_S}(x_0\cdots x_r))
		\]
		is  crepant  for any $0<\varepsilon \leq1$. 
		Let $H_{Z}:= g_* f^*H_S$.
		Fix a rational number $a>0$ such that 
		\[
		A_Z:= g_* f^*(H_S) + a E_Z -\textstyle \sum_{i,j} a_{ij} g_*E_{ij}
		\]
		is effective.
		Note that 
		\[
		A_{Z} \sim_{\bQ} D_Z - g_*G 
		\quad \text{ and } \quad
		(D_{Z}-g_*G)\vert_U = g_*(f^*D_S)\vert_U
		.\]
		
		We claim  that  for $0<\varepsilon \ll 1$, the pair 
		\[
		(Z,B_Z+D_Z+ \varepsilon A_Z)
		\]
		(i) is a  $\bQ$-factorial dlt pair with no lc centers contained in $Z_0$
		and  (ii) its restriction to $U$ is a good minimal model over $U$.
		Indeed, (i) holds since
		$(Z,B_Z+D_{Z})_U$ is dlt with no lc centers contained in $\Supp(H_Z\vert_U)$ and $(Z,B_Z+D_{Z})$ is $\bQ$-factorial with no lc centers contained in $Z_0$ as $(Z,B_Z+D_Z+Z_0)$ is lc.
		For (ii), note that the birational map
		\[
		(Z ,B_Z+ D_Z + \varepsilon A_Z)_U 
		\dashrightarrow 
		(X,B+D+  \varepsilon H) \times_R U
		\]
		is a morphism, crepant, and $K_{X}+B+D+\varepsilon H$ is ample.
		Since (i) and (ii) hold, 
		\cite[Theorem 1.1]{HX13} implies that   $(Z,B_Z+D_Z+ \varepsilon A_Z)$ admits a canonical model over $S$ for any $0<\varepsilon \ll 1$. 
		Write
		\[
		h:Z \dashrightarrow \cX
		\]
		for the birational contraction,  $\cB:= h_* B_Z$,  $\cD:= h_* (D_{Z})$, and $\cH := h_* H_{Z}$.
		By the justification for (ii), $\cX\dashrightarrow X_S$ is an isomorphism over $U$. 
		As a consequence, $h_* g_* G=0$.
		Since
		\begin{multline}\label{e:Dequiv}
			K_{Z}+B_Z+D_Z+ \varepsilon A_Z
			\sim_{\bQ}
			K_{Z}+B_Z+D_Z +Z_{0}+ \varepsilon A_Z
			\sim_{\bQ,S}
			\varepsilon  A_Z
			\sim_{\bQ} \varepsilon (D_{Z}-g_* G),
		\end{multline}
		$\cX$ is the ample model of $D_Z-g_*G$ over $S$ and $\cD$ is ample over $S$.
		
		\medskip

		\emph{Step 4. We now verify that  the constructed $(\cX,\cB+\cD) \to S$ is a $\bT$-equivariant family of boundary polarized CY pairs of the desired form.}
		Since $\bT$ is a connected group scheme, the MMP steps $Y \dashrightarrow Z$ and $Z\dashrightarrow \cX$ are $\bT$-equivariant.
		Thus the $\bT$-action on $(X_S,B_S+D_S)$ extends to a $\bT$-action on $(\cX,\cB+\cD)$.
		Next, since 
		\[
		(Z,B_Z+D_Z+ \varepsilon A_Z)_{U} \dashrightarrow (X,B+D+  \varepsilon H) \times_R U
		\]
		is the morphism to the canonical model, 
		the  induced birational map 
		\[
		(\cX,\cB+\cD) \dashrightarrow (X_S, B_S+ D_S)
		\]
		is an isomorphism over $U$.
		Thus the induced birational map  
		\[
		\cX  \dashrightarrow X^i\times_R S
		\]
		is also an isomorphism over $U$.
		Consider the birational map 
		\[
		h_i:
		Z\vert_{U_i} 
		\dashrightarrow X^i \times \bT
		\]
		Note that  $h_i\vert_U$ is a morphism and $h_i$ is a birational contraction that only contracts divisors that intersect $Z_U$.
		Additionally, 
		\[
		(D_Z-g_*G)_{U_i}
		=
		(h_i)^* ( D^i \times \bT)
		,\]
		where $h_i^*$ denotes the birational pullback, as both $\bQ$-divisors agree with $g_* f^*D_S$ over $U$ and have no components of their support contained in $V_Z(x_i)$. 
		Thus the relative section rings of $(D_Z-g_*G)_{U_i}$ and $D^i\times \bT$ over $U_i$ are isomorphic. 
		Since $D^i\times \bT$ is ample over $U_i$, the induced birational map 
		\[
		\cX\vert_{U_i} \dashrightarrow X^i\times \bT \cong X^i \times_R U_i
		\]
		is an isomorphism as desired. 
		
		Finally, we verify that $(\cX,\cB+\cD) \to S$ is a family of boundary polarized CY pairs. Since 
		$(Z,B_Z+D_{Z} +Z_0)$
		is an lc CY pair over $S$  and \[
		h_*(B_Z+D_{Z}+Z_{0})=\cB+ \cD+\cX_0,\]
		the pair
		$(\cX,\cB+\cD+\cX_0)$ is lc and CY over $S$.
		Thus $(\cX, \cB+\cD) \to S$ is a family of slc pairs in an open neighborhood of the point $o \in S$ given by the vanishing of $x_0,\ldots, x_r$ by Proposition \ref{p:kollarsmoothslcfam}. 
		Let $W\subset S$ denote the largest open set over which $(\cX,\cB+\cD)\to S$ is a family of slc pairs. 
		Note that $W$ is $\bT$-invariant and contains both $o$ and $U$.
		Since $o$ lies in the closure of every $\bT$-orbit of a point in $S\setminus U \cong V(x_0\cdots x_r) \subset \bA^{r+1}_{\kappa}$, $W=S$. 
		Thus $(\cX,\cB+\cD) \to S$ is a family of slc pairs. 
		Since $K_{\cX}+\cB+\cD \sim_{\bQ,S}0$ and $\cD$ is relatively ample, it is a family of boundary polarized CY pairs. 
	\end{proof}
	
	We now describe the  relative section ring of the family over ${\rm S}^r$ constructed in Theorem \ref{t:highrankScomplete} as a multigraded Rees module.

	\begin{corollary}\label{c:bpcyReesmodule}
		Keep the notation from the assumptions and conclusion of Theorem \ref{t:highrankScomplete}. 
		Let $(X,D) := (X^0,D^0)$ and $F_i$ denote the filtration of $V_m:= H^0(X,mD)$ induced by $(X^i,D^i)$ as defined in Section \ref{sss:bpcyfiltrations}.
		
		For each integer $m\geq 0$ such that $mD_K$ is a $\bZ$-divisor, the following hold: 
		\begin{enumerate}
			\item The $R[x_0,\ldots, x_r]/(x_0 \cdots x_r-\pi)$-module $H^0(\cX,m\cD)$ is flat and the natural map 
			\[
			H^0(\cX,m\cD)\otimes k(s)\to  H^0(\cX_s, m\cD_s) 
			\]
			is an isomorphism for all $s\in \mathrm{S}^r$.
			\item There is a commutative diagram of $\bZ^r$-graded $R[x_0,\ldots, x_r]/(x_0 \cdots x_r-\pi)$-modules
			\[
			\begin{tikzcd}
				{\rm Rees}_1(F_1,\ldots, F_r;V_m)\arrow[r,hook] \arrow[d,"\cong"]  & V_m[x_1^{\pm1},\ldots, x_r^{\pm 1}]\arrow[d,"\cong"] \\
				H^0(\cX,m\cD) \arrow[r,hook] & H^0(\cX_{U_0}, m\cD_{U_0})  
			\end{tikzcd}
			,\]
			where the horizontal maps are the natural inclusions, the vertical maps are isomorphisms, and the right vertical isomorphism is induced by pulling back via the isomorphism $\cX_{U_0} \to X \times \bT$.
		\end{enumerate}
	\end{corollary}
	
	\begin{proof}
		Statement (1) follows from exactly the same argument as Proposition \ref{p:familyArRees}.
		To prove (2), first observe that the isomorphism 
		\[
		(\cX_{U_0},\cD_{U_0})\cong (X,D) \times_R \Spec(R[x_1^{\pm1},\ldots, x_r^{\pm1}]) \cong  (X,D) \times \bT
		\]
		induces an isomorphism
		\[
		\phi:V_m[x_1^{\pm1},\ldots, x_r^{\pm1 }] 
		\to
		H^0(\cX_{U_0}, m\cD_{U_0}).
		\]
		Since $(\cX,\cB+\cD ) \to \mathrm{S}^r$ is a family of slc pairs and $\mathrm{S}^r$ is regular,  $(\cX,\cB+\cD +V_{\cX}(x_0\cdots x_r))$ is slc and every lc center of $(\cX,\cB+\cD)$ dominates $S$ by \cite[Theorem 4.54 and Corollary 4.56]{Kol23}.
		Thus if the non-normal locus of $\cX$ is non-empty, then it dominates $S$. As $\cX_{U} \cong X_K \times \bT$, which is normal, $\cX$ must be normal. 
		Thus we may view $H^0(\cX,m\cD)$ as an 
		\[
		R[x_0,\ldots, x_r]/(x_0\cdots x_r-\pi)
		\]
		submodule of $H^0(\cX_{U_0}, m\cD_{U_0})$.
		It remains to verify  that $\phi^{-1}(H^0(\cX,m\cD))={\rm Rees}_1(F_1,\ldots, F_r;V_m)$.

		Fix an element $s\in H^0(\cX_{U_0},m\cD_{U_0})$. 
		By the isomorphism $\phi$, we can write
		$s=\sum_{\la\in \bZ^r}s_{\la} x^{-\la}$ where $x^{-\la}:= x_1^{-\la_1}\cdots x_r^{-\la_r}$ and 
		each $s_\la \in H^0(X,mD)$ with only finitely many terms  nonzero. 
		Since
		$\cX\to \mathrm{S}^r$ is flat and 
		\[
		{\rm codim}_{ \mathrm{S}^r}(\mathrm{S}^r \setminus (U_0 \cup \cdots \cup U_r))\geq 2,
		\]
		it follows that 
		\[
		{\rm codim}_{\cX}(\cX\setminus (\cX_{U_0}\cup \cdots \cup \cX_{U_r})) \geq 2
		.\]
		Thus  $s\in H^0(\cX,m\cD)$ if and only if 
		$s \in H^0(\cX_{U_i}, m\cD_{U_i})$ for all $0\leq i \leq r$.
		Using the isomorphism
		\[
		\cX_{U_i} \cong X^i\times_R U_i \cong X^i\times_R  \Spec(R[x_0^{\pm1}, \ldots, \widehat{x}_i^{\pm1} ,\ldots, x_r^{\pm1}]) \cong  X^i \times \bT
		\]
		and that
		\[
		x_i =\pi/(x_0 \cdots \widehat{x}_i \cdots x_r)
		\in 
		\cO_{{\rm S}^r} (U_i)
		,\]
		we see that the latter holds if and only if 
		\[
		\sum_{\la \in \bZ^r} 
		s_\la \big( \pi  /( x_0\cdots \widehat{x_i}\cdots x_r)\big)^{-\la_i}
		x_1^{-\la_1}\cdots \widehat{x}_i^{-\la_i} \cdots x_r^{-\la_r} \in H^0(X^i\times \bT, mD^i \times \bT) 
		\]
		for all $1\leq i \leq r$.
		This is equivalent to the condition that 
		$\pi^{-\la_i} s_{\la} \in H^0(X^i,mD^i)$
		for all $\la \in \bZ^r$ and $1\leq i \leq r$. 
		The latter holds if and only if   $s_{\la} \in F_i^{\la_i} H^0(X,mD)$ for all $\la \in \bZ^r$ and $1\leq i \leq r$ or, equivalently, $s\in {\rm Rees}_1(F_1,\ldots, F_r;V_m)$.
		Therefore $\phi^{-1}(H^0(\cX,m\cD) ) = {\rm Rees}_1(F_1,\ldots, F_r; V_m)$ as desired.
	\end{proof}
	
	\subsection{Diagonalizability}
	We now apply the previous results to simultaneously diagonalize the filtrations in Section \ref{sss:bpcyfiltrations}.
	
	\begin{prop}\label{p:diagfiltbpcy}
		Let $(X,B+D)$ be a boundary polarized CY pair over $R$ and $b$ be a positive integer such that $b D$ is  an integral divisor.
		Let $v_1,\ldots, v_r \in{\rm Sk}(X_K,B_K+D_K)$  be divisorial valuations, and write $F_j$ for the filtration of $V_{m}:= H^0(X,mbD)$ defined by 
		\[
		F_j^\la V_m := \{s \in V_m \, \vert\, v_j(s)\geq \la \}.
		\]
		The following hold: 
		\begin{enumerate}
			\item  For each integer $m \geq 0$, the filtrations $F_1,\ldots, F_r$ of $V_m$ are simultaneously diagonalizable. 
			\item If $X_0$ is reduced, the graded $\kappa$-algebra
			$\bigoplus_{m \in \bN} {\rm gr}_{F_1,\ldots, F_r}(V_m)$
			is reduced. 
		\end{enumerate}
	\end{prop}

	\begin{proof}
		Let $(X^0,B^0+D^0):=(X,B+D)$.
		By Proposition \ref{p:EinSKtodegen}, for each $1\leq j\leq r$,  
		there exists a boundary polarized CY pair $(X^j,B^j+D^j)$ over $R$  with an isomorphism
		\[
		(X^0_K,B^0_K+D^0_K) \cong (X^j_K,B^j_K+D^j_K)
		\]
		over $K$ such that $E_j := X_{0,\red}^j$ is irreducible and $v_j= ( {\rm mult}_{E_j}(X_0^j))^{-1} \ord_{E_j}$.
		By Lemma \ref{l:inducedfilt}, the filtration  $F_j$ of $V_m$ equals  the filtration of $V_m$ induced by $(X^j,B^j+D^j)$ in Definition \ref{d:filtbpcy/R}.

		If $X_0^0,\ldots, X^r_0$ are all reduced divisors, then (1) follows immediately from Lemma \ref{l:Reesmodflat->diag} and Corollary \ref{c:bpcyReesmodule}. 
		To prove the full result in general, we will take a ramified base change. 
		
		Fix a positive integer $d$ that is divisible by all coefficients of $X^0_0,\ldots, X^r_0$.
		Following Section \ref{sss:basechange}, let $\tR := R[\pi^{1/d}]$ and $\tK:={\rm Frac}(\tR)= K[\pi^{1/d}]$,
		which admit $\mu_d$-actions.  
		Let 
		\[
		(\tX^0, \tB^0+\tD^0),\ldots, (\tX^r,\tB^r+ \tD^r) 
		\]
		denote the base changes of the original collection of pairs over $R$ via  $R\hookrightarrow \tR$.
		These are boundary polarized CY pairs over $\tR$ admitting $\mu_d$-actions that extend the action on $\tR$. 
		Furthermore, the isomorphisms over $K$ induce isomorphisms
		\[
		(\tX^0, \tB^0+\tD^0)_{\tK}
		\cong (\tX^i,\tB^i+ \tD^i)_{\tK}
		.\]
		Let $\tV_m = H^0(\tX^0,mb\tD^0)$.
		Write $\tF_i$ for the filtration of $\tV_m$ induced by $(\tX^i,\tB^i+\tD^i)$. 
		By Lemmas \ref{l:basechangebpcy} and \ref{l:basechangefilt}, the 
		$\mu_d$-action on $(\tX^0,\tB^0+\tD^0)$ induces a decomposition into weight spaces
		$\tV_{m} = \oplus_{i=0}^{d-1} \tV_{m,i}$
		such that
		\[
		\tF_j^\la \tV_m = (\tF_j^\la \tV_{m} \cap \tV_{m,0}) \oplus \cdots \oplus (\tF_j^\la\tV_m \cap  \tV_{m,d-1})
		.\]
		Thus the Rees module admits a direct sum decomposition
		\[
		{\rm Rees}_1(\tF_1,\ldots, \tF_r;\tV_m) = \bigoplus_{\la \in \bZ^r}  \tF^{\la} \tV_m x^{-\la} 
		= 
		\bigoplus_{\la \in \bZ^r} \bigoplus_{i=0}^{d-1} (\tF^{\la} \tV_{m} \cap \tV_{m,i}) x^{-\la} 
		,\]
		which endows it with the structure of a $\bZ/d\bZ$-graded module over the $\bZ/d\bZ$-graded ring
		\[
		\tA:= \tR[x_0 ,\ldots, x_r]/\big(x_0 \cdots x_r -\pi^{1/d}\big) 
		,\]
		where $R$ and $x_1,\ldots, x_r$ have weight 0 and $\pi^{1/d}$ and $x_0$ both have weight $1$.
		The weight zero summand satisfies
		\[
		{\rm Rees}_1(\tF_1,\ldots, \tF_r;\tV_m)_0
		=
		{\rm Rees}_{d}(F_1,\ldots, F_r;V_m)
		\]
		by Lemma \ref{l:basechangefilt}.
		Furthermore we have an isomorphism of $\bZ/d\bZ$-graded rings
		\begin{align*}
			\tA 
			&\cong 
			R[t,x_0,\ldots, x_r]/ (x_0 \cdots x_r -t, t^d- \pi) \\
			&\cong 
			R[x_0,\ldots, x_r]/(x_0^d\cdots x_r^d-\pi),
		\end{align*}
		where $R$ and $x_1,\ldots, x_r$ have weight $0$ and $x_0$ and $t$ have weight  $1$.
		Thus the weight $0$ component satisfies 
		\begin{align*}
			(\widetilde{A})_0 &\cong R[x_0^d, x_1,\ldots, x_r]/(x_0^d x_1^d \cdots x_r^d-\pi) \\
			&\cong 
			R[y_0,x_1,\ldots, x_r]/(y_0x_1^d\cdots x_r^d -\pi)\\
			&=:A.
		\end{align*}
		Thus $\tilde{A}$ is a free module over $A$ with a basis $1,x_0,\ldots, x_0^{d-1}$.

		We now analyze the flatness of the multigraded Rees modules and the diagonalizability of the filtrations. 
		Since $\tX_0$ and $\tX^j_0$ are reduced  by Lemma \ref{l:basechangebpcy}, 
		${\rm Rees}_1(\tF_1,\ldots, \tF_r;\tV_m)$
		is a flat $\tA$-module by Corollary \ref{c:bpcyReesmodule}.
		Since $A \to \tA$ is free and, in particular, flat, ${\rm Rees}_1(\tF_1,\ldots, \tF_r;\tV_m)$ is flat as an $A$-module. 
		Thus the degree $0$-summand 
		${\rm Rees}_d(F_1,\ldots, F_r;V_m)$
		is flat as an $A$-module.
		Since $F_1,\ldots, F_r$ are bounded $\tfrac{1}{d} \bZ$-filtrations  by Lemma \ref{l:inducedfilt}, Lemma    \ref{l:Reesmodflat->diag} now implies that $F_1,\ldots, F_r$ are simultaneously diagonalizable. This completes the proof of statement (1).
		
		We now prove statement (2). 
		Assume that $X_0$ is reduced.
		By Corollary \ref{c:bpcyReesmodule}, there is an $\tA$-module isomorphism
		\[
		\bigoplus_{m \in \bN} {\rm Rees}_1(\tF_1,\ldots, \tF_r; \tV_m) 
		\overset{\simeq}{\longrightarrow} \bigoplus_{m \in \bN} H^0(\tcX, mb\widetilde{\cD})
		,\]
		where $(\tcX, \tcB + \widetilde{\cD})  \to \Spec \, \tA$ is a family of boundary polarized CY pairs satisfying the conclusion of Theorem \ref{t:highrankScomplete} for the collection of pairs 
		\[
		(\tX^0,\tB^0+\tD^0),\ldots, 
		(\tX^r,\tB^r+\tD^r).
		\]
		Furthermore, since the map respects multiplication, it is an $\tA$-algebra isomorphism. 
		By  Lemma  \ref{l:Reestquotient=gradedDVR} and Corollary \ref{c:bpcyReesmodule}.1, the base change via $\tA\to \kappa$ that sends each $x_i\mapsto 0$  induces a graded algebra isomorphism 
		\[
		\bigoplus_{m \in \bN} 
		{\rm gr}_{\tF_1,\ldots, \tF_r}(\tV_m) 
		\overset{\simeq}{\longrightarrow} \bigoplus_{m \in \bN} H^0(\tcX_o, mb\widetilde{\cD}_o)
		,\]
		where $o\in \Spec(\tA)$ is the point at which $x_0,\ldots, x_r$ vanish.
		Since $\tcX_o$ is reduced and $\widetilde{\cD}_o$ is a Mumford divisor, the graded rings are both  reduced. 
		
		We now compare the graded modules for $\tF_1,\ldots, \tF_r$ and $F_1,\ldots, F_r$.
		Observe that 
		\[
		\tV_{m,i}\cong \pi^{i/d} H^0(X,mbD + (i/d) X_0 ) = \pi^{i/d} V_m
		\]
		for $0\leq i\leq d-1$, where the  first isomorphism holds by Lemma \ref{l:basechangebpcy} and the second equality uses that $\lfloor (i/d)X_0\rfloor=0$ as $X_0$ is reduced. 
		Therefore
		\[
		\pi^{1/d} \tV_m = \pi \tV_{m,0} \oplus \pi^{1/d} \tV_{m,0} \oplus \cdots \oplus \pi^{(d-1)/d}  \tV_{m,0}
		.\]
		Using this equality and  Lemma \ref{l:basechangefilt}, we compute that 
		\[ 
		({\rm gr}_{\tF}^\la \tV_m)_0
		\cong 
		\frac{F^{\la/d}V_m }{F^{>\la/d} V_m +  F^{\la/d} V_m \cap \pi V_m }
		=: {\rm gr}_{F}^{\la/d} V_m
		\]
		for $\la \in \bZ^r$,
		where the $0$ in the subscript on the left denotes the degree $0$-component with respect to the $\bZ/d\bZ$-grading.
		This gives an
		isomorphism 
		\[
		\Big(\bigoplus_{m \in \bN}{\rm gr}_{\tF_1,\ldots, \tF_r} \tV_{m}\Big)_0
		\overset{\simeq}{\longrightarrow}
		\bigoplus_{m \in \bN}{\rm gr}_{F_1,\ldots, F_r} V_{m}
		\]
		of $\kappa$-vector spaces.
		As the isomorphism respects multiplication, it is an isomorphism of $\kappa$-algebras. 
		Since the graded ring $\bigoplus_{m \in \bN}{\rm gr}_{\tF_1,\ldots, \tF_r} \tV_{m}$ is reduced, 
		the graded $\kappa$-algebra on the right side is also  reduced.
	\end{proof}
	
	\subsection{CY pairs}
	
	We now use the cone construction to deduce a version of Proposition \ref{p:diagfiltbpcy} for  CY pairs polarized by an ample divisor.  
	
	\begin{prop}\label{p:diagfiltcypol}
		Let $(X_K,B_K)$ be an lc CY pair over $K$ with an lc model $(X,B)$ and $H$ be a relatively ample $\bQ$-Cartier $\bZ$-divisor on $X$.
		Let $w_1,\ldots, w_r\in \Sk(X_K,B_K)$ be divisorial valuations and $G_j$ denote the filtration of $W_m:= H^0(X,mH)$ defined by 
		\[
		G_j^\la W_m := \{s \in W_m \, \vert\, w_j(s) \geq \la\}
		.\]
		Then the following hold: 
		\begin{enumerate}
			\item The filtrations $G_1,\ldots, G_r$ of $W_m$ are simultaneously diagonalizable for each integer $m\geq 0$. 
			\item If $X_0$ is reduced, then the graded ring $\bigoplus_{m \in \bN} {\gr}_{G_1,\ldots, G_r}(W_m)$ is reduced. 
		\end{enumerate}  
	\end{prop}
	
	In the definition of $G_j$, the value $w_j(s)$ is defined by viewing $s\in H^0(X,mH) \subset K(X)$.
	
	\begin{proof}
		We may assume that $\dim X_K>0$ as otherwise the statement holds trivially.
		Fix an integer $d>0$ such that $dH$ is a Cartier divisor. 
		Let 
		\[
		Y := C_p(X,H):= \Proj \bigoplus_{m \in \bN} \bigoplus_{i=0}^m H^0(X,iH) x^{m-i} 
		\]
		and 
		\[
		Y^d := C_p(X,dH) :=\Proj \bigoplus_{m \in \bN} \bigoplus_{i=0}^m H^0(X,idH) y^{m-i}
		.\]
		Let $D:= V(x) \subset Y$ and $D^d := V(y) $, which are ample divisors isomorphic to $X$.
		Note that $Y^d$ is a $\mu_d$-quotient of $Y$ induced by the $\mu_d$-action on the graded ring in the definition of $Y$ that acts on the summand $H^0(X,iH) x^{m-i}$ with weight $\overline{i} \in \mathbb{Z}/d\mathbb{Z}$. Write $\rho:Y  \to Y^d$ for the quotient map. 
		
		Recall that we have inclusions 
		\[
		X_H\hookrightarrow C_a(X,H)  \hookrightarrow C_p(X,H)=Y \quad \text{ and } \quad 
		X_{dH}\hookrightarrow C_a(X,dH)  \hookrightarrow C_p(X,dH)=:Y^d
		,\]
		and morphisms $\sigma:X_{H}\to X$ and  $\sigma_d:X_{dH} \to X$ with the latter a $\bG_m$-bundle.
		Let $G^d$ denote the $\bQ$-divisor on $Y^d$ that is the closure of the preimage of $B$ under  $X_{dH} \to X$. 
		By the same argument as \cite[Proof of Proposition 2.20]{ABB+}, 
		\[
		(Y^d, G^d+D^d+Y^d_{0,\red})
		\]
		is an lc CY pair. As $D^d$ is relatively ample, $(Y^d,G^d+D^d)$ is a boundary polarized CY pair over $R$.
		Since $X$ is normal, $H$ is a Cartier divisor away from a codimension two subset. 
		Thus $X_{H} \to X_{dH}$ is \'etale in codimension one. 
		Thus if we write $G:= \rho^* G^d$, then 
		\begin{equation}\label{eq:YY^d}
			(Y,G+D+Y_{0,\red}) \to (Y^d,G^{d}+D^d+Y^d_{0,\red})
		\end{equation}
		is crepant, where we are additionally using that $\rho^*D^d = d D$ and that both  $D$  and $D^d$ have coefficient one. 
		Hence $(Y,G+D+Y_{0,\red})$ is also an lc CY pair and  $(Y,G+D)$ is a boundary polarized CY pair over $R$.

		Observe that for $f= \sum_{i\geq0} f_i\in H^0(Y\setminus D , \cO_Y)= \bigoplus_{i\geq 0} H^0(X,iH)$  with $f_i \in H^0(X,iH)$, 
		\[
		\ord_{D}(f) = \min \{ - i \, \vert\, f_i \neq 0 \}.
		\]
		Therefore 
		\begin{align*}
			V_m:=H^0(Y,mD) &= H^0(X,0H) \oplus H^0(X,1H) \oplus \cdots \oplus H^0(X,mH)\\
			& =W_0 \oplus W_1 \oplus \cdots \oplus W_m
			.\end{align*}
		
		We now show that there exist divisorial valuations $v_1,\ldots, v_r \in {\rm Sk}(Y_K,G_K+D_K)$ such that 
		\begin{equation}\label{e:w_jf}
			v_j(f) =\min \{w_j(f_i) +w_j(iH)\, \vert\, f_i \neq 0\}
		\end{equation}
		for any $f= \sum_{i\geq 0} f_i \in H^0(Y\setminus D, \cO_Y)$ where  $f_i\in H^0(X,iH)$.
		By Proposition \ref{p:dltmodexist}, there exists   a dlt modification 
		\[
		g:(X',B'+X'_{0,\red}) \to (X,B+X_{0,\red})
		\]
		such that there exist prime divisors $E_1,\ldots, E_r$ on $X'$ and $w_j = (\mult_{E_j} (X'_0))^{-1} \ord_{E_j}$. 
		Let $H'=g^*H$.
		Consider the Seifert $\bG_m$-bundle and $\bG_m$-bundle
		\[
		\sigma': X'_{H'}  \to X' \quad \text{ and } \quad 
		\sigma'_d : X'_{dH'} \to X'
		.\]
		Consider the prime divisors 
		\[
		C_j :=\sigma'^{-1}(E_j)_{\red}\quad \text{ and } \quad C_j^d := {\sigma'_d}^{-1} (E_j)_{\red}
		.
		\]
		Set 
		\[
		v_j:= (\ord_{C_j}(\pi))^{-1} \ord_{C_j}\quad\text{ and } \quad
		v^d_j:=(\ord_{C^d_{j}}(\pi))^{-1} \ord_{C_j^d}
		.\]
		Using that ${\rm Ram}(C_j/ C_j^d) =  \ord_{C_j}(\pi) / \ord_{C^d_j}(\pi)
		$, we see that $v_j^d = \rho_* v_j$.
		By a computation on a local chart of the $\bG_m$-bundle $X'_{dH'} \to X'$, we see that 
		\[
		v^d_j(f_{id}) = w_j(f_{id}) +w_j(idH)
		\]
		for any $f_{id}\in H^0(X,id H)$. 
		
		We now show that $v_j$ is in the relevant skeleton.
		Since we have a crepant birational morphism and an open embedding 
		\[
		(X'_{dH'},{\sigma'_d}^{-1}(B')+ (X'_{dH'})_{0,\red})
		\to 
		(X_{dH}, {\sigma_d}^{-1}(B) + (X_{dH})_{0,\red})
		\hookrightarrow
		(Y^d,G^d+D^d +Y^d_{0,\red}),
		\]
		$v_j^d \in \Sk(Y^d_K,G^d_K+D^d_K)$.
		Since \eqref{eq:YY^d} is crepant, $v_j \in{\rm Sk}(Y_K,G_K+D_K)$ as desired.
		
		We now verify that \eqref{e:w_jf} is satisfied.
		Now fix any $f=\sum_{i\geq0} f_i\in H^0(Y\setminus D , \cO_Y)$ where $f_i \in H^0(X,iH)$. 
		As $v_j$ is a $\bG_m$-invariant valuation, its valuation ideals are $\bG_m$-invariant. 
		Hence 
		\[
		v_j(f ) = \min \{v_j(f_i) \, \vert\, f_i \neq 0\}.
		\]
		Furthermore, we compute that 
		\[
		v_j(f_i) = \tfrac{1}{d} v_j(f_i^d) = \tfrac{1}{d} v_j^d(f_i^d)=
		\tfrac{1}{d} (w_j(f_i^d)+w_j(idH))
		=
		w_j(f_i)+w_j(iH)
		.\]
		Therefore 
		\eqref{e:w_jf} holds as desired.
		
		Let $F_j$ be the filtration of $V_m$ defined by $F_j^{\la} V_m := \{s\in V_m \vert\, v_j(s) \geq \la \}$. 
		Fix an integer $b>0$ such that $F_1,\ldots, F_r$ and $G_1,\ldots, G_r$ are all $\tfrac{1}{b}\bZ$-filtrations and $b w_1(H),\ldots, b w_r(H) \in \bZ$. 
		By \eqref{e:w_jf}, 
		\[
		F_j^\la V_m =G_j^{\la-w_j(0H)} W_0 \oplus \cdots \oplus G_j^{\la-w_j(mH)} W_m
		.\]
		Therefore there are direct sum decompositions
		\[
		{\rm Rees}_b(F_1, \ldots, F_r;V_m) = \bigoplus_{i=0}^m x^{-ic}{\rm Rees}_b(G_{1},\ldots, G_{r};W_{i})
		\quad \text{ and }\quad  
		{\rm gr}_{F_1,\ldots, F_r}(V_m)
		\cong 
		\bigoplus_{i=0}^m {\rm gr}_{G_1,\ldots, G_r}(W_i)
		\] 
		of $R[x_0,x_1,\ldots, x_r]/(x_0 x_1^b \cdots x_r^b-\pi)$-modules and $\kappa$-vector spaces, 
		where  $c:= b(w_1(H),\ldots, w_r(H))\in \Z^r$.
		Since $F_1,\ldots, F_r$ are simultaneously diagonalizable by Proposition \ref{p:diagfiltbpcy}, 
		Lemma \ref{l:Reesmodflat->diag} implies that 
		\[
		{\rm Rees}_b(F_1, \ldots, F_r;V_m)
		\]
		is flat.
		Thus each summand including 
		\[
		{\rm Rees}_b(G_{1},\ldots, G_{r};W_{m})
		\]
		is flat. 
		Thus $G_1,\ldots ,G_r$ are 
		simultaneously diagonalizable by Lemma \ref{l:Reesmodflat->diag}. 
		Additionally, since the previous isomorphisms respect multiplication, there is an isomorphism
		\begin{equation}\label{e:grFG}
			\bigoplus_{m \in \bN}{\gr}_{F_1,\ldots, F_r}(V_m)
			\cong
			\bigoplus_{m \in \bN}\bigoplus_{i=0}^m {\gr}_{G_1,\ldots, G_r}(W_i)
		\end{equation}
		as $\bN$-graded rings.
		
		For the final statement, assume that $X_0$ is reduced. We aim to show that $Y_0$ is reduced as well.
		Observe that $Y_0= \Proj S/\pi S$, where 
		\[
		S:= \bigoplus_{m\in \bN} \bigoplus_{i=0}^mH^0(X,iH) x^{m-i}
		.\]
		Fix a bihomogeneous element $f\in S$ of degree $(m,i)$ with $f^n \in \pi S$ for some integer $n\geq0$.
		Write $f= s x^{m-i}$ with $s\in H^0(X,iH)$.
		Since $f^n \in \pi S$,
		\[
		{\rm div}_{X}(s^n) +inH \geq X_0 .
		\]
		The latter implies that 
		\[
		{\rm div}_X(s) +  i H \geq X_0,
		\]and so $s \in \pi H^0(X,iH)$, which implies that $f\in \pi S$. Therefore $S/\pi S$ is reduced. Hence $Y_0$ is also reduced.
		By Proposition \ref{p:diagfiltbpcy}, the graded rings in \eqref{e:grFG} are reduced.
		As $\bigoplus_{m \in \bN} {\rm gr}_{G_1,\ldots, G_r}( W_m)$ is a subalgebra of the graded ring on the right, it is also reduced.
	\end{proof}

	\section{Valuative Independence}\label{s:valind}
	
	In this section, we prove Theorem \ref{thm:val-indep} on the existence of valuatively independent bases  for polarized CY degenerations.
	\medskip
	
	Throughout, let $R$ be a DVR containing our base field $\bk$. 
	Let $K:= {\rm Frac}(R)$, $\pi \in R$ be a fixed uniformizer, and $0 \in \Spec\,R$ denote the closed point. 
	
	\subsection{Definition}
	
	\begin{defn}[Valuative independence]\label{d:valindDVR}
		Let $(X_K,B_K)$ be a projective lc CY pair over $K$, $L_K$ be an ample line bundle on $X_K$, and $m\geq 0$ be an integer.
		A $K$-vector space basis $(\theta_1,\ldots,\theta_{N_m})$
		of $H^0(X_K,L_K^{\otimes m})$ is \emph{valuatively independent} if for some metric on $L_K$ we have
		\begin{equation}\label{d:valind=}
			v( a_1 \theta_1 + \cdots + a_{N_m} \theta_{N_m})
			= \min \{ v(a_1 \theta_1),\ldots, v(a_{N_m} \theta_{N_m}) \}
		\end{equation}
		for all $a_1,\ldots, a_{N_m} \in K$ and $v\in {\rm Sk}(X_K,B_K)$.
	\end{defn}

	\begin{rem}\label{r:indofmetric}
		The metric in Definition \ref{d:valindDVR} is used to evaluate sections
		along valuations, as explained in Section \ref{sss:metrics}.
		However, the condition in the definition  does not depend on
		the choice of metric.
		
		\begin{enumerate}
			\item 
			If the condition in Definition \ref{d:valindDVR} holds for one metric of $L_K$, then it holds for all metrics of $L_K$. 
			Indeed, let $(f^1,Y^1,M^1)$ 
			and $(f^2,Y^2,M^2)$ be two metrics of $L_K$. 
			After pulling back both metrics to a common normal birational model, we may assume that $Y:= Y^1 = Y^2$ and $f:=f^1= f^2$. 
			The induced isomorphism $M^1_K \cong  M^2_K$ extends to an isomorphism $M^1 \cong M^2 (D)$
			for some Cartier divisor $D$ on $Y$ with support contained in $Y_0$.
			Thus
			\[
			v_{M^1}(s)= v_{M^2}(s) +mv(D)
			\]
			for all $s\in H^0(X_K,L_K^{\otimes m})$,
			where the subscript denotes the metric used.
			Since the correction term $mv(D)$ is independent of $s$, the equation \eqref{d:valind=} holds for one metric if and only if it holds for every metric.
			
			\item Fix a Cartier divisor $H_K$ on $X_K$ with an isomorphism $L_K \cong \cO_{X_K}(H_K)$.
			We may alternatively evaluate a section $s\in H^0(X_K,L_K^{\otimes m})$ by using the identification 
			\[
			H^0(X_K,L_K^{\otimes m})
			\cong
			H^0(X_K,m H_K)
			\subset 
			K(X_K)
			\]
			and evaluating the corresponding rational function. We denote this value by  $v_H(s)$.
			
			Given a metric $(f,Y,M)$, there exists a Cartier divisor $H_Y$ on $Y$ such that $H_Y\vert_{Y_K} = f^*H_K$
			and the  isomorphism $\cO_{Y_K}(f^*H_K) \cong M_K$ extends to an isomorphism 
			$\cO_Y(H_Y)\cong M$. 
			Thus 
			\[
			v_H(s) + m v(H_Y) = v_M(s)
			\]
			for all $s\in H^0(X_K,L_K^{\otimes m})$.
			Thus \eqref{d:valind=} holds when sections are evaluated along valuations using rational functions via the isomorphism $L_K\cong \cO_{X_K}(H_K)$  if and only if it holds when sections are evaluated  using a metric.
		\end{enumerate}
	\end{rem}

	\subsection{Tropicalization}
	Let $(X_K,B_K)$ be a projective lc CY pair over $K$, and $L_K$ be an ample line bundle on $X_K$ with a metric. 
	
	For an integer $m\geq0$, the \emph{tropicalization} of a section $0\neq s\in H^0(X_K,L_K^{\otimes m})$ 
	is the function
	\[
	s^{\rm trop}: \Sk(X_K,B_K) \to \bR \quad \text{defined by } \quad s^{\rm trop}(v) = v(s).
	\]
	We now study properties of the tropicalization of a valuatively independent basis.

	\begin{lem}\label{l:valindDVRbasicprops}
		Fix an integer $m\geq0$. If $(\theta_1,\ldots, \theta_{N_m})$ is a valuatively independent basis  of $H^0(X_K,L_K^{\otimes m})$, then the multiset of functions 
		\[\theta_1^{\rm trop},\ldots, \theta_{N_m}^{\rm trop}
		\]
		is independent of the choice of valuatively independent basis up to  replacing
		$\theta_i^{\rm trop}$ with $\theta_i^{\rm trop}  + d_i$ for some $d_i \in \bZ$.
	\end{lem}
	
	Note that  replacing  $\theta_i^{\rm trop}$ with $\theta_i^{\rm trop}  + d_i$ corresponds to replacing $\theta_i$ with $\pi^{d_i} \theta_i$, which preserves valuative independence.

	\begin{proof}
		Fix $w\in {\rm Sk}(X_K,B_K)$. Consider the $R$-module
		\[
		V_m := \{s \in H^0(X_K,L_K^{\otimes m}) \, \vert\,  w(s) \geq 0\}
		.\]
		By valuative independence, there is a direct sum decomposition
		\[
		V_m= R  \pi^{-\lfloor w(\theta_1)\rfloor }\theta_1\oplus \cdots \oplus R
		\pi^{-\lfloor w(\theta_{N_m})\rfloor }
		\theta_{N_m}.
		\]
		Thus $V_m$ is a free $R$-module of rank $N_m$ and has a basis given by  $(\theta'_1,\ldots, \theta'_{N_m})$, where $\theta'_i := \pi^{ -\lfloor w(\theta_i)\rfloor} \theta_i$ and  ${\theta'_i}^{\trop}
		= \theta_i^{\trop} - \lfloor w(\theta_i) \rfloor
		$.

		For each valuation $v\in \Sk(X_K,B_K)$, consider the  filtration  $F_v $ of $V_m$ defined  by 
		\[
		F^\la_v V_m := \{ s\in V_m \, \vert\, v(s) \geq \la \}.
		\]
		Since $\ord_{F_v} (s)= v(s)$ for all $s\in V_m$,
		$(\theta'_1,\ldots, \theta'_{N_m})$ is a basis of $V_m$ that diagonalizes $F_v$. 
		Thus Lemma \ref{l:uniquenessbasisDVR} implies that the multiset of functions $\theta'^{\trop}_1,\ldots, \theta'^{\trop}_{N_m}$ does not depend on the choice of the valuatively independent basis.
		As $\theta'^{\trop}_i  = -\lfloor w(\theta_i)\rfloor + \theta^{\trop}_i$, we conclude that the multiset of functions $\theta_1^{\rm trop},\ldots, \theta_{N_m}^{\rm trop}$ does not depend on the choice of valuatively independent basis up to  replacing $\theta_i^{\rm trop}$ with $\theta_i^{\rm trop}  + d_i$ for some $d_i \in \bZ$.
	\end{proof}
	
	We now show that the tropicalizations of basis elements are distinct in the maximal degeneration case.

	\begin{lem}\label{l:CYdegenmaxdistinct}
		Assume that $\dim({\rm Sk}(X_K,B_K))= \dim X_K$ and the natural map $\bk \to R/\pi R$ is an isomorphism.
		If $(\theta_1,\ldots, \theta_{N_m})$ is a valuatively independent basis for $H^0(X_K,L_K^{\otimes m})$ for some integer $m\geq0$, then 
		\[
		\theta_i^{\trop} \neq \theta_j^{\trop}+d
		\]
		for all $d\in \bZ$ and $1\leq i <j \leq N_m$.
	\end{lem}
	
	\begin{proof}
		Fix an lc model $(X,B)$ of $(X_K,B_K)$. 
		Let $f:Y \to X$ be a log resolution of $(X,B+X_{0,\red})$. 
		After replacing $Y$ with a higher birational model, we may assume that the metric on $L_K$ is induced by a line bundle $M$ on $Y$ with an isomorphism $M_K \cong f_K^* L_K$.
		Define  $B_Y$ by the formula 
		\[
		K_Y+B_Y = f^*(K_X+B +X_{0,\red}).
		\]
		Let $E $ denote the coefficient one part of $B_Y$. 
		By Lemma \ref{l:LCplace}, 
		\[
		\Sk(X_K,B_K) = \QM(Y, E) \cap \{ v\in \Val_X\, \vert\, v(\pi)=1 \}.
		\]
		Since $n:=\dim X_K=\dim \Sk(X_K,B_K)$, 
		there exists an $n+1$ dimensional cone of $\QM(Y,E)$  whose interior meets the subset  $\{ v(\pi)>0\}$. 
		Thus there exist distinct prime divisors $E_0,\ldots, E_{n}\subset E$ meeting at a closed point $y \in E_0 \cap \cdots \cap E_{n}$ and with at least one $E_i$ contained in $Y_0$. 
		Since $R/\pi R=\bk$ and $\bk$ is algebraically closed, $y$ is a $\bk$-point.
		
		Choose a regular system of parameters
		$y_0,\ldots, y_n \in \cO_{Y,y}$ such that $y_i$ locally defines $E_i$ at $y$.
		Choose $\alpha=(\alpha_0,\ldots, \alpha_n) \in \bR_{>0}^{n+1}$ such that $\alpha_0,\ldots, \alpha_n$ are linearly independent over $\bQ$ and $\sum_{i=0}^n \alpha_i \mult_{E_i}(Y_0)  =1$. Thus $v:=v_\alpha \in \Sk(X_K,B_K)$.
		Aiming for a contradiction, assume that $1\leq i < j \leq N_m$ and
		$\theta_i^{\trop} =\theta_j^{\trop}+d$ for some $d\in \bZ$.
		After replacing $\theta_j$ with $\pi^d\theta_j$, we may assume that $d=0$.
		After trivializing $M$ near $y$ to view $\theta_i$ and $\theta_j$ as rational functions on $Y$,  the same argument as the proof of Lemma \ref{l:maxcythetaidistinct} shows that there exists $c\in \bk^\times$ such that 
		\[
		v(\theta_i-c\theta_j) > v(\theta_i) = v(\theta_j).
		\]
		Since $i\neq j$, this contradicts valuative independence.
	\end{proof}
	
	The valuatively independent bases we construct will often satisfy the following scaling property, which is motivated by mirror symmetry, see for example \cite[Section 5.1]{Li25}.
	
	\begin{defn}
		For each integer $m\geq0$, let $(\theta_{m,i})_{i=1}^{N_m}$ be a basis of $H^0(X_K,L_K^{\otimes m})$.
		The collection of bases is \emph{closed under tropical scaling} if, for any integers $k,l\geq 0$ and  $1\leq i \leq N_l$,  there exists an integer $1\leq j \leq N_{kl}$ such that
		$k \theta_{l,i}^{\trop} =  \theta_{kl,j}^{\trop}$.
	\end{defn}

	\subsection{Algebraic case}
	
	We now prove the existence of valuatively independent bases when the CY pair is the base change of a pair over the fraction field of a DVR essentially of finite type over $\bk$.
	In this case, the proof is slightly easier and applies to all powers of the line bundle.

	\begin{thm}\label{t:valindbasisalg}
		Assume that $R$ is essentially of finite type over $\bk$. 
		Let $(X_K,B_K)$ be a projective lc CY pair over $K$ and $L_K$ be an ample line bundle on $X_K$ with a metric. Let $\widehat{R}$ denote the completion of $R$ with respect to the maximal ideal and $\widehat{K}= {\rm Frac}(\widehat{R})$.
		Let 
		\[
		(X_{\widehat{K}},B_{\widehat{K}}) :=(X_K,B_K)\times_K \widehat{K}
		\]
		and $L_{\hK}$ denote the pullback of $L_K$ to $X_{\hK}$ equipped with the induced metric. 
		For each $m \in \bN$, there exists a valuatively independent basis  $(\theta_{m,1},\ldots, \theta_{m,N_m})$ of $H^0(X_{\widehat{K}},L_{\widehat{K}}^{\otimes m})$. 
	\end{thm}
	
	We will eventually deduce the result from  Proposition \ref{p:diagfiltcypol}. First we prove the following lemmas.
	
	\begin{lem}\label{l:skelcomp}
		Keep the assumptions and notation from Theorem \ref{t:valindbasisalg}.
		Then restriction of valuations 
		under the inclusion $K(X_K) \hookrightarrow K(X_{\widehat{K}})$
		induces a bijection  
		\[\Sk(X_{\widehat{K}},B_{\widehat{K}}) \to \Sk(X_K,B_K)
		\]
		that restricts to a bijection on divisorial valuations on the skeleta.
	\end{lem}
	
	\begin{proof}
		Fix an lc model 
		$(X,B)$ of $(X_K,B_K)$
		and a log resolution $f:Y \to X$ of $(X,B+X_{0,\red})$.
		Write $\Gamma$ for the $\bQ$-divisor on $Y$ such that $K_{Y} + \Gamma = f^*(K_X+B +X_{0,\red})$ and $E:= (\Gamma)^{=1}$.
		
		Now we consider the base changes. 
		First note that $(X_{\widehat{K}},B_{\widehat{K}})$ is a projective lc CY pair over $\widehat{K}$ as these properties are preserved under field extension. 
		Let $(\hX,\hB):=(X,B)\times_R \hR$ and $\widehat{Y}:= Y\times_R \hR$.
		Since $(X,B+X_{0,\red})$ is an lc CY pair over $R$ and $\hX \to X$ is regular,  $(\hX,\hB+\hX_{0,\red})$ is an lc CY pair over $\hR$.
		Thus $(\hX,\hB)$ is an lc model of $(X_{\hK},B_{\hK})$.
		Furthermore  $\widehat{Y}\to \widehat{X}$ is a log resolution of $(\hX,\hB+\hX_{0,\red})$ and 
		$\widehat{\Gamma}:= \Gamma\times_R \widehat{R}$ satisfies $K_{\hY} + \widehat{\Gamma} = \widehat{f}^*(K_{\hX}+\hB + \hX_{0,\red} )$.
		
		Now we consider the divisor $\widehat{E}:= E\times_R \hR = (\widehat{\Gamma})^{=1}$.
		We will show that restricting valuations via $K(X) \subset K(\hX)$ induces a bijection 
		\begin{equation}\label{e:QMv>0}
			\QM(\hY,\hE) \cap \{ v(\widehat{\pi}) >0\} \to \QM(Y,E) \cap \{ v(\pi)>0 \}
			.
		\end{equation}
		Since $R/(\pi) \to \hR/(\pi)$ is an isomorphism, the induced map $\hY_0 \to Y_0$ is an isomorphism. 
		This induces a bijection between strata $\hZ $ of $\hE$ contained in $\hY_0$ and strata $Z$ of $E$ contained in $Y_0$.
		For such strata $\hZ$ and $Z$, let $\widehat{\eta}\in \hY$ and $\eta\in Y$ denote the corresponding generic points. 
		We claim that restriction induces a bijection 
		\begin{equation}\label{e:QMint}
			\QM_{\heta}(\hY,\hE)^\circ \to \QM_{\eta}(Y,E)^\circ
			,\end{equation}
		where the superscripts denote the interior of the simplicial cones i.e. the subset of valuations such that the weights are all strictly positive. 
		To see this, let $E_1,\ldots, E_r$ denote the irreducible components of $E$ that contain $\eta$ and $y_1,\ldots, y_r\in \cO_{Y,\eta}$ be a regular system of parameters such that $y_i$ is a local defining equation for $E_i$ at $\eta$.
		Let $\hE_i$ denote the unique irreducible component of $ E_i \times_R \hR$ containing $\heta$ and $\widehat{y}_i$ denote the image of $y_i$ under the map $\cO_{Y,\eta}\to \cO_{\hY,\widehat{\eta}}$. 
		Thus $\hE_1,\ldots, \hE_r$ are the irreducible components of $\hE$ that contain $\widehat{\eta}$ and $\hy_1,\ldots, \hy_r\in \cO_{\hY,\widehat{\eta}}$ are a regular system of parameters such that $\widehat{y}_i$ is a local defining equation for $\hE_i$ at $\heta$.
		For $\alpha \in \bR_{>0}^r$, consider the corresponding valuations  $v_{\alpha} \in \QM_\eta(Y,E)^\circ$ and $\hv_{\alpha} \in \QM_{\heta}(\hY,\hE)^\circ$.
		Since the induced map on completions 
		$\widehat{\cO_{Y,\eta}} \to \widehat{\cO_{\hY,\widehat{\eta}}}$
		is an isomorphism by \cite[\href{https://stacks.math.columbia.edu/tag/0BG5}{Lemma 0BG5}]{stacks}, $v_\alpha(h) = \hv_{\alpha}(h)$ for any $h\in \cO_{Y,\eta} \hookrightarrow \cO_{\hY,\heta}$.
		As $K(X)$ and $K(\hX)$ are isomorphic to the fraction fields of the above rings, $\hv_\alpha$ restricts to $v_\alpha$.
		Therefore we get the  bijection in \eqref{e:QMint}.
		Since $\QM(\hY,\hE) \cap \{ v(\widehat{\pi}) >0\}$ and $\QM(Y,E) \cap \{ v(\pi)>0 \} $ are both disjoint unions of the interiors of the corresponding simplicial cones, restriction induces the desired bijection in \eqref{e:QMv>0}.
		Since $R\hookrightarrow \widehat{R}$ sends the uniformizer $\pi$ of $R$ to a uniformizer in $\widehat{R}$, this restricts to a bijection 
		\[
		\Sk(X_{\widehat{K}}, B_{\widehat{K}})\to \Sk(X_K,B_K)
		.\]
		As the bijection restricts to a bijection on the $\bQ$-points, it restricts to a bijection on divisorial valuations. 
	\end{proof}

	\begin{lem}\label{l:lcCYpolextensionviaBPCY}
		Keep the assumptions and notation from Theorem \ref{t:valindbasisalg}.
		If $H_K$ is an ample Cartier divisor on $X_K$, then  there exists an lc model $(X,B)$  of $(X_K,B_K)$  such that the closure of $H_K$ in $X$ is $\bQ$-Cartier and relatively ample over $R$. 
	\end{lem}
	
	The proof is  similar to the proof of Proposition \ref{p:EinSKtodegen}.

	\begin{proof}
		By Propositions  \ref{lem:semistable-reduction} and \ref{p:dltmodexist}, there exists an lc model $(X',B')$ of $(X_K,B_K)$ and a $\bQ$-factorial dlt modification 
		\[
		(Y,B_Y+ Y_{0,\red}) \to (X',B'+X'_{0,\red}).
		\]
		Since $(X',B'+X'_{0,\red})$ is a CY pair, so is $(Y,B_Y+Y_{0,\red})$. 
		
		Fix a prime divisor $E \subset Y_0$.
		Let $b_E: = {\rm coeff}_{E}(Y_0)$.
		Since $Y$ is $\bQ$-factorial, $(Y,B_Y+Y_{0,\red})$ is a dlt CY pair, and 
		\[
		K_Y+B_Y +Y_{0,\red} -E \sim_{\bQ} -E \sim_{\bQ} b_E^{-1} Y_0 - E \geq 0
		,\]
		\cite[Theorem 1.6]{HX13} implies that a $(K_Y+B_Y +Y_{0,\red} -E)$-MMP with scaling terminates and produces a birational contraction to a good minimal model $g:Y\dashrightarrow Z$. 
		Since $g$ is a birational contraction and $(Y,B_Y+Y_{0,\red})$ is an lc CY pair, $(Z,B_Z+Z_{0,\red})$ is also an lc CY pair, where $B_Z:= g_* B_Y$.
		
		We claim that $g$ contracts precisely the prime divisors in $Y_{0}- b_E E$.
		Indeed, since the above MMP agrees with the $Y_0 -b_E E$ MMP, $g$ only contracts divisors contained in $Y_{0}-b_E E$.
		On the other hand, as $-E_{Z}$ is relatively nef, all such divisors must get contracted.
		Indeed, otherwise, by the connectedness of $Z_0$,  there would exist a curve  $C\subset Z_0$ not contained in $E_Z$ but satisfying $C \cap E_Z \neq \emptyset$. 
		Such a curve would satisfy $E_Z \cdot C>0$, and hence contradict that $-E_Z$ is relatively nef.
		Thus $g:Y\dashrightarrow Z$ is an isomorphism over $Y\setminus Y_0$ and $E_Z:= g_* E = Z_{0,\red}$.

		We now construct the desired birational model $X$.
		Since $H_K$ is ample, there exists an effective $\bQ$-divisor $A_K \sim_{\bQ} H_K$ such that $\Supp(A_K)$ does not contain any lc centers of $(X_K,B_K)$.
		Let $A_Z$ denote the closure of the pullback of $A_K$ under the morphism $Z_K \cong Y_K \to 
		X_K$.
		Since $g$ is a composition of steps of the MMP and $Y$ is $\bQ$-factorial, $Z$ is $\bQ$-factorial. 
		Note that  $A_Z$ contains no horizontal lc centers of $(Z,B_Z+Z_{0,\red})$ and the pair 
		$(Z,B_Z+Z_{0,\red})$ is lc and dlt on the generic fiber.
		Thus for a sufficiently small rational number $\varepsilon>0$, the pair
		\[
		(Z,B_Z+\varepsilon A_Z)
		\]
		is dlt and has no lc centers contained in $Z_0$.
		On the generic fiber 
		\[
		(Z, B_Z+ \varepsilon A_Z)_{K} \to (X_K,B_K+ \varepsilon A_K)
		\]
		is crepant and $K_{X_K}+B_K+ \varepsilon A_K \sim_{\bQ} \varepsilon A_K$, which is ample. 
		Therefore $(Z, B_Z+ \varepsilon A_Z)_{K}$ is a good minimal model.
		Thus \cite[Theorem 1.1]{HX13} implies that there exists a birational contraction
		\[
		h:Z\dashrightarrow X
		\]
		to a $(K_{Z}+B_Z + \varepsilon A_Z)$-canonical model. 
		As canonical models are unique, the birational map $X\dashrightarrow X'$ is an isomorphism over $K$. 
		Let $B:= h_* B_{Z}$ and $A:= h_* A_Z$. 
		
		We now analyze $(X,B)$.
		Since 
		$(Z,B_Z+ Z_{0,\red})$
		is an lc CY pair and $h$ is a birational contraction, 
		$(X,B+X_{0,\red})$
		is an lc CY pair. 
		Since $Z_0$ is irreducible, $X_0$ is irreducible. 
		Thus
		\[
		K_{X}+ B + \varepsilon A \sim_{\bQ} 
		K_{X}+ B +X_{0,\red}+ \varepsilon A \sim_{\bQ} \varepsilon A, 
		\]
		and so 
		$A$ is relatively ample. 
		Now let $H$ denote the closure of $H_K$ in $X$. 
		Since $H_{X\setminus X_0} \sim_{\bQ} A_{X\setminus X_0}$,
		there exists a $\bQ$-divisor $F$ supported on $X_0$ such that $H \sim_{\bQ}  A+ F$. 
		Since  $X_{0, \red}$ is prime and $X_0\sim 0$, $F\sim_{\bQ}0$.
		Therefore $H\sim_{\bQ} A$, and so  $H$ is $\bQ$-Cartier and relatively ample.
	\end{proof}

	\begin{proof}[Proof of Theorem \ref{t:valindbasisalg}]
		Fix a Cartier divisor $H_K$ representing $L_K$. 
		By Lemma \ref{l:lcCYpolextensionviaBPCY}, there exists an lc model $(X,B)$ of $(X_K,B_K)$ such that $H$, which denotes the closure of $H_K$ in $X$, is $\bQ$-Cartier and relatively ample. 
		For each integer $m\geq0$,  consider the $K$-vector space and free $R$-module
		\[
		V_{m,K} := H^0(X_K,mH_K) \quad \text{ and }  \quad V_m:= H^0(X,mH).
		\]
		For $v\in \Sk(X_K,B_K)$, let $F_v$  denote the induced filtration of  $V_m$.
		
		We now analyze the completion.
		Let $(\hX, \hB):= (X,B) \times_R \hR$ and $\hH = \phi^* H$, where $\phi: \hX\to X$ is the natural map. 
		Similarly, we set
		\[
		\hV_{m,\hK} := H^0(X_{\hK},m H_{\hK}) \quad \text{ and }  \quad \hV_m:= H^0(\hX,m\hH).
		\]
		For each $\widehat{v}\in \Sk(X_{\hK},B_{\hK})$, let $F_{\widehat{v}}$  denote the induced filtration of $\hV_m$.
		We have a natural map 
		$\phi^* \cO_{X}(mH) \to \cO_{\hX}(m\hH)$, which is an isomorphism away from a codimension 2 closed subset. Since $\phi$ is regular and $\cO_{X}(mH)$ is $S_2$, $\phi^* \cO_{X}(mH)$ is also $S_2$.
		Thus the previous map is an isomorphism.
		Now, by flat base change, there are natural isomorphisms
		\[
		V_{m,K}\otimes_K \hK \overset{\simeq}{\longrightarrow}\hV_{m,\hK}\quad \text{ and } \quad 
		V_m \otimes_R \hR \overset{\simeq}{\longrightarrow} \hV_{m}.
		\]
		Furthermore, by Lemma \ref{l:skelcomp}, there is a natural bijection 
		\[
		\phi_*:\Sk(X_{\widehat{K}},B_{\widehat{K}}) \to \Sk(X_K,B_K)
		\]
		induced by restricting  valuations along $K(X) \hookrightarrow K(\hat{X})$. 
		For any divisorial valuation $\hv\in \Sk(X_{\hK},B_{\hK})$  with $v:= \phi_* \hv$ and $\la \in \bR$,
		\begin{equation}\label{eq:filtcomp}
			F_{\hv}^\la \hV_m=
			\im(F^\la _v V_m \otimes_{R} \hR \to \hV_m) ,
		\end{equation}
		by the flat base change theorem applied on a proper birational model $Y\to X$ for which $v= c\cdot \ord_E$ with $E$ a prime divisor on $Y$.
		\medskip
		
		\noindent \emph{Claim:} For any integer $m\geq0$ and   finite set of divisorial valuations $\widehat{v}_1,\ldots, \widehat{v}_r\in \Sk(X_{\widehat{K}},B_{\widehat{K}})$, the  filtrations 
		$\hF_1,\ldots,\hF_r$ of $\hV_m$
		are simultaneously diagonalizable, 
		where $\hF_i:= F_{\hat{v}_i}$ denotes the filtration of $\hV_m$ induced by $\hat{v}_i$.
		
		\begin{proof}[Proof of Claim]
			Let $v_i:= \phi_* \hv_i \in \Sk(X_K,B_K)$ and $F_i := F_{v_i}$ denote the filtration of $V_m$ induced by $v_i$. 
			By Proposition \ref{p:diagfiltcypol}, the filtrations $F_1,\ldots, F_r$ of $V_m$ are simultaneously diagonalizable.
			Thus there exists a basis $s_{m,1},\ldots,s_{m,N_m}$ of $V_m$ diagonalizing $F_1,\ldots, F_r$.
			Let $\widehat{s}_{m,i}\in \hV_{m}$ denote the image of $s_{m,i}\otimes 1$ under the isomorphism $V_m \otimes_R \hR \overset{\cong}{\to} \hV_m$.
			Thus $\widehat{s}_{m,1},\ldots, \widehat{s}_{m,N_m}$
			form a basis for $\hV_m$. 
			By \eqref{eq:filtcomp},  this basis diagonalizes $\hF_1,\ldots, \hF_r$. 
			Therefore $\hF_1,\ldots, \hF_r$ are simultaneously diagonalizable.
		\end{proof}
		\medskip 
		
		We now construct the valuatively independent basis. 
		Fix an integer $m\geq0$ and consider the set of filtrations $F_{\hv}$ of $\hV_m$ where $\hv \in \Sk(X_{\hK},B_{\hK})$ is divisorial. 
		The set of such filtrations is countable, since divisorial valuations correspond to rational points in finitely many rational polyhedra, and any finite subset is diagonalizable by the claim. Thus Proposition \ref{p:dvrdiagfinitecol} implies  that there exists a basis 
		$(\theta_{m,1},\ldots, \theta_{m,N_m})$
		of $\hV_m$ that   diagonalizes $F_{\hv}$ for all divisorial valuations $\hv\in \Sk(X_{\widehat{K}},B_{\widehat{K}})$. 
		
		We claim that the previous basis satisfies valuative independence. 
		Indeed, fix $\hv\in \Sk(X_{\hK},B_{\hK})$ and $a_1,\ldots, a_{N_m} \in \hK$. 
		We seek to verify that 
		\begin{equation}\label{e:valindproof}
			\hv(a_1 \theta_{m,1}+ \cdots + a_{N_m} \theta_{m,N_m}) = 
			\min \big\{ \hv(a_1 \theta_{m,1}), \ldots ,\hv(a_{N_m} \theta_{m,N_m}) \big\}.
		\end{equation}
		First assume that $\hv$ is divisorial. 
		In this case, choose $d\gg0 $ such that $b_i:=\pi^da_i \in \hR $ for all $i$. 
		Now 
		\begin{multline*}
			\hv(a_1 \theta_{m,1}+ \cdots + a_{N_m} \theta_{m,N_m})
			=
			\hv(b_1 \theta_{m,1}+ \cdots + b_{N_m} \theta_{m,N_m})
			-d\\
			=
			\min \big\{ \hv(b_i \theta_{m,i}) \,\vert\, 1\leq i \leq N_m\big\}
			-d
			=
			\min \big\{ \hv(a_i \theta_{m,i})\,\vert\, 1\leq i \leq N_m\big\},
		\end{multline*} 
		where the  second equality holds since the basis diagonalizes $F_{\hv}$.
		Thus \eqref{e:valindproof} holds for divisorial valuations in $\Sk(X_{\hK}, B_{\hK})$.
		Since divisorial valuations are dense in the skeleton and the tropicalization of a fixed nonzero element in  $\hV_{m,\hK}$ is continuous,
		\eqref{e:valindproof} holds for all  $\hv\in \Sk(X_{\hK}, B_{\hK})$.
		Thus $(\theta_{m,1},\ldots, \theta_{m,N_m})$  is valuatively independent.
	\end{proof}

	\subsection{Approximation}
	
	In this section, we prove an approximation result that will allow us to deduce Theorem \ref{thm:val-indep} from the algebraic case. 
	
	\begin{prop}\label{p:approxCY}
		Let $R:=\bk[\![t]\!] $ and $K:= {\rm Frac}(R)=
		\bk(\!(t)\!)$. 
		Let $(X_K,B_K)$ be an lc CY pair over $K$ with an lc model $(X,B)$ and $L$ be a relatively big and semiample line bundle on $X$.
		For every integer $l\geq 1$, there exists the following data:
		\begin{enumerate}
			\item  a DVR $R'$ that is the local ring of a smooth pointed curve over $\bk$ with a uniformizer also denoted by $t$ and an isomorphism $R\cong \widehat{R'}$ that sends $t$ to $t$ and
			\item  an lc CY pair $(X'_{K'},B'_{K'})$ over $K':= {\rm Frac}(R')$  with an lc model $(X',B')$ and an isomorphism 
			\[
			\phi: X_l:= X\times_{R} \Spec\, R/ (t^{l})  \overset{\simeq}{\longrightarrow} X'  \times_{R'} \Spec\, R'/ (t^{l}) =: X'_l
			\]
			over $R/\fm^{l} \cong R'/\fm'^{l}$, and
			\item a relatively big and semiample line bundle $L'$ on $X'$ with an isomorphism  
			\[
			L_l \cong \phi^*(L'_{l}) 
			,\]
			where  $L_{l}$ and $L'_l$ denote the pullback of $L$ and $L'$ to $X_l$ and $X'_l$, such that the  natural maps 
			\[
			H^0(X,L^{\otimes m})\otimes_{R} R/(t^{l})\to 
			H^0(X_l ,L_{l}^{\otimes m}) \quad \text{ and } \quad 
			H^0(X',L'^{\otimes m})\otimes_{R'} R'/(t^{l})\to 
			H^0(X'_l,{L'_l}^{\otimes m})
			\]
			are isomorphisms for all integers $m\geq 0$.
		\end{enumerate}
	\end{prop}
	
	A related, but less technical, statement is proven in \cite{NX16} using Greenberg approximation \cite{Gre66}. 
	As the proof of Proposition \ref{p:approxCY} involves approximating the   base change of $(X,B)$  via $\bk[\![t]\!] \subset  \bk[\![t^{1/d}]\!]$  and its $\mu_d$-action, we need the following equivariant approximation result, which follows from Artin approximation \cite{Art69}.
	
	\begin{lem}\label{l:equivgreenberg}
		Fix integers $d,l\geq 1$.
		Let $A$ be a finitely generated $\bk$-algebra with a $\mu_d$-action.
		Let $R= \bk [\![t]\!]$ and $\tR:=R[t^{1/d}] \cong \bk[\![t^{1/d}]\!]$ with $\mu_d$-action such that $\deg(t^{1/d}) = \overline{1}\in \bZ/d\bZ$.
		
		For any $\mu_d$-equivariant $\bk$-algebra homomorphism 
		$\phi: A\to \tR$,
		there exists the following data: 
		\begin{enumerate}
			\item a DVR $R'$ that is the local ring of a smooth pointed curve over $\bk$ with a fixed uniformizer of $R'$ also denoted by $t$, 
			\item an isomorphism $R \cong \widehat{R}'$ that sends the uniformizer  to the uniformizer,  and 
			\item a $\mu_d$-equivariant $\bk$-algebra homomorphism 
			$\phi': A \to  \tR' := R' [t^{1/d}]  $
			such that $\phi$ and $\phi'$ agree modulo $t^{l}$.
		\end{enumerate}
		
	\end{lem}

	\begin{proof}
		The $\mu_d$-action on $A$ induces a $\bZ/d\bZ$-grading on the ring $A$. 
		Since  $A$ is a finitely generated $\bk$-algebra, there exist finitely many homogeneous elements $a_1,\ldots, a_n\in A$ such that the $\bk$-algebra homomorphism
		\[
		\psi:\bk[x_1,\ldots, x_n] \to A\quad \quad  x_i \longmapsto a_i
		\]
		is surjective. 
		Let $b_i\in\{0,\ldots,d-1\}$ denote the degree of $a_i$.
		Thus $\psi$ is equivariant with respect to the $\mu_d$-action on $\bk[x_1,\ldots, x_n]$  such that $\deg(x_i) = \overline{b_i} \in \bZ/d\bZ$.
		Thus  $\ker(\psi)$ is a homogeneous ideal with respect to the $\bZ/d\bZ$-grading. 
		Thus  there exist polynomials  $f_1,\ldots,f_r\in \bk[x_1,\ldots,x_n] $ that are homogeneous with respect to the $\bZ/d\bZ$-grading such that $\ker(\psi) = (f_1,\ldots, f_r)$. 
		
		Now consider the homomorphism
		$\phi:A \to \tR$.
		Since it is $\mu_d$-equivariant and $\tR= \oplus_{i=0}^{d-1} t^{i/d} R$ is the induced $\bZ/d\bZ$-grading, there exists  $g_i \in R$ such that
		\[
		\phi(a_i)  = t^{b_i/d}  g_i
		.\]
		Since each $f_j$ is homogeneous, we may write 
		\[
		f_j(t^{b_1/d}y_1,\ldots, t^{b_n/d} y_n)
		=
		t^{c_j/d} F_j(t,y_1,\ldots,y_n)
		\in \bk[t^{1/d}, y_1,\ldots, y_n] 
		,\]
		where 
		$F_j(t,y_1,\ldots, y_n) \in \bk[t,y_1,\ldots, y_n]$ and $c_j = \deg(f_j) \in \{0,1,\ldots, d-1\}$. 
		Since $f_j \in \ker(\psi)$, 
		\[
		f_j(a_1,\ldots, a_n) = 0 \in A
		\]
		Therefore
		\[
		t^{c_j/d} F_j(t, g_1(t),\ldots, g_n(t)) = 0 
		\in \tR.
		\]
		As $\tR$ is a domain, the latter is equivalent to the condition that 
		\[
		F_j(t, g_1(t),\ldots, g_n(t)) = 0 .
		\]
		Equivalently,  $(g_1(t),\ldots, g_n(t)) \in (R)^n$ is a solution to the system of equations 
		\[
		F_1 = \cdots = F_r = 0 
		.\]
		To finish the proof, we will   apply Artin approximation to this system of equations.
		
		Consider the henselization $B:=(\cO_{\bA^1,0})^{\rm h}$.
		Since $B$ is an  excellent DVR with $\widehat{B} \cong R$,  \cite[Theorem 1.10]{Art69}
		implies that there exists 
		$(h_1(t),\ldots, h_n(t))\in B^n$  that solves
		\[
		F_1= \cdots = F_r=0
		\quad \text{ and } \quad g_i \equiv h_i \mod t^{l} .
		\]
		Since the henselization is the filtered colimit of local rings $\cO_{C,c}$ of  pointed \'etale maps $(C,c) \to (\bA^1_\bk,0)$ over $\bk$, 
		there exists a local ring of a smooth curve $R'$, with a uniformizer $t$ and an isomorphism $\widehat{R'}\cong R$ sending the uniformizer to the uniformizer, and  $(g'_1,\ldots, g'_n)\in (R')^n$ that solves
		\[
		F_1= \cdots = F_r=0
		\quad \text{ and } \quad g_i \equiv g'_i \mod t^{l}
		,\]
		where we view $F_j$ as a polynomial in $B[y_1,\ldots, y_n]$.
		Now consider the $\mu_d$-equivariant $\bk$-algebra homomorphism  
		\[
		\bk[x_1,\ldots, x_n] \longrightarrow \tR' := R'[t^{1/d}]
		\quad\quad  x_i \longmapsto t^{b_i/d} g'_i(t)
		.\]
		Since $(g'_1,\ldots, g'_n) \in (R')^n$ solves the above system, this homomorphism sends  each $f_j$ to zero  and therefore factors through a  $\mu_d$-equivariant $\bk$-algebra homomorphism
		\[
		\phi':A \longrightarrow \tR' 
		\quad\quad  a_i \longmapsto t^{b_i/d} g'_i(t)
		\]
		Since $g_i$ and $g'_i$ agree modulo $t^{l}$, $\phi$ and $\phi'$ agree modulo $t^{l}$.
	\end{proof}
	
	We also need the following special case of abundance.
	
	\begin{lem}\label{lem:HX}
		Let $R$ be a DVR essentially of finite type over $\bk$ and $(X,B) \to \Spec\, R$ be a projective family of slc CY pairs with $X_K$ normal.
		If $L$ is a line bundle on $X$ that is relatively big and nef and $L_K$ is semiample, then $h^0(X_K,L_K^{\otimes m}) = h^0(X_0,L_0^{\otimes m})$ for all integers $m \geq 0$.
	\end{lem}
	
	\begin{proof}
		By Propositions \ref{p:kollarsmoothslcfam} and \ref{p:kollarsmoothslcfamgen}, $(X,B)$ is an lc pair and no lc center of the pair is contained in $X_0$.
		By \cite[Corollary 1.36]{Kol13}, there exists a  $\bQ$-factorial dlt modification  
		\[
		f:(Y,B_Y) \to (X,B).
		\]
		Since $f$ is crepant, $(Y,B_Y)$ is an lc CY pair.
		Now choose an integer $r>0$ such that $L_K^{\otimes r}$ is base point free. 
		By Bertini's theorem, there exists a nonzero section $s_K \in  H^0(X_K,L_K^{\otimes r})$ whose zero set does not contain lc centers of $(X_K,B_K)$.
		After multiplying $s_K$ by a sufficiently large power of a uniformizer, it extends to a regular section of $L^{\otimes r}$. 
		Let $G$ be its zero divisor.
		Thus $\Supp(G)$ contains no lc centers of $(X,B)$ and $\cO_{X}(G) \cong L^{\otimes r} $. 
		
		We now show that $H^1(X, L^{\otimes m})$ is a torsion free $R$-module for all  $m\geq 0$.
		Since $\Supp(G)$ contains no lc centers of $(X,B)$ and $(Y,B_Y) \to (X,B)$ is crepant, $\Supp(f^*G)$ contains no lc centers of $(Y,B_Y)$. 
		Thus $(Y, B_Y+ \varepsilon f^*G)$ is dlt for $0<\varepsilon \ll1$. 
		Note that 
		\[
		K_Y+B_Y + \varepsilon f^* G \sim_{\bQ}\varepsilon f^*G ,\] which is relatively big and nef and $(f^*G)_{Y_K}$ is semiample. 
		Thus \cite[Theorem 1.1 and Lemma 2.4]{HX13} imply that the pair is a good minimal model over $R$. 
		Therefore $f^*G$ is relatively semiample over $R$.  Thus $G$ and $L$ are relatively semiample over $R$.
		Since $L $ is relatively semiample and $(X,B)$ is an lc CY pair, \cite[Theorem 1.12]{Fuj14} implies the desired torsion freeness statement.
		
		To finish the proof, fix an integer $m\geq 0$, and  consider the short exact sequence
		\[
		0 \to L^{\otimes m} \overset{\cdot t}{\to} L^{\otimes m} \to L_0^{\otimes m} 
		\to 0,
		\]
		where the first nonzero map is multiplication by a uniformizer $t\in R$. 
		Since $H^1(X,L^{\otimes m})$ is torsion free, multiplication  by $t$ on $H^1(X,L^{\otimes m})$ is injective. 
		Therefore the long exact sequence on cohomology implies that
		\[
		H^0(X,L^{\otimes m})/tH^0(X,L^{\otimes m})
		\cong H^0(X_0,L_0^{\otimes m}).
		\]
		Now we compute that 
		\[
		h^0(X_0, L_0^{\otimes m})
		=
		\dim_{R/\fm} H^0(X,L^{\otimes m})/ tH^0(X,L^{\otimes m}) 
		=
		\rk_{R} H^0(X,L^{\otimes m}) 
		=
		h^0(X_K, L_K^{\otimes m})
		.\]
		Above, the second equality holds since $H^0(X,L^{\otimes m})$ is a finite and torsion-free $R$-module, and hence a free $R$-module. 
		The third equality holds by flat base change.
	\end{proof}
	
	We are now ready to prove the main approximation result.
	
	\begin{proof}[Proof of Proposition \ref{p:approxCY}]
		We first apply a base change to produce a family of slc CY pairs.
		Let $d$ denote the least common multiple of the coefficients of $X_0$.  
		Let  $\tR: = \bk[\![t^{1/d}]\!]$ and $\tK : = \bk(\!(t^{1/d})\!)$.
		Let $\tX $ denote the normalization of $X\times_{R} \tR$,  $q: \tX \to X$ denote the composition of the normalization and projection and $\tB := q^* B$.
		By \cite[Lemma 2.53]{Kol23}, $\tX_0$ is reduced. 
		Since 
		\[
		(\tX, \tB+\tX_0) \to (X,B  +X_{0,\red})
		\]
		is crepant and $(X,B+X_{0,\red})$ is an lc CY pair over $R$, $(\tX,\tB+\tX_0)$ is  an lc  CY pair over $\tR$ by \cite[Corollary 2.43]{Kol13}.
		Thus $(\tX,\tB) \to \Spec\, \tR$ is a family of slc CY pairs by Proposition \ref{p:kollarsmoothslcfam}. 
		Additionally,  the  product of the  $\mu_d$-action on $\Spec \, \tR$ and the trivial action on $X$ lifts to a $\mu_d$-action on $\tX$ that fixes $\tB$ and such that $\tX\to \Spec\, \tR$ is equivariant. 
		
		We now  analyze the polarization $\tL := q^* L$.  
		Choose relatively very ample line bundles $L_1$ and $L_2$ on $X$ such that $L = L_1 \otimes L_2^{*}$. 
		By Bertini's Theorem, there exist  Cartier divisors $G_1$ and $G_2$ on $X$  in the linear equivalence class of $L_1$ and $L_2$, respectively, such that $G_1$ and $G_2$ do not contain an irreducible component of $X_0$ or the image of a codimension $1$ singular point of $\tX_{0}$. Thus $G_1$ and $G_2$ are relative Cartier divisors and $L \cong \cO_{X}(G_1-G_2)$. 
		Thus $\tL \cong \cO_{\tX}( \tG_1-\tG_2)$ where  $\tG_i : =q^* G_i$. 
		
		We now embed $\tf:(\tX, \tB)\to \Spec\, \tR$ in a relative projective space. 
		Fix a positive integer $r$ such that $r(K_{\tX}+\tB) \sim 0$, $\tL^{\otimes r}$ is relatively base point free, and
		$\cO_{\tX}(r \tG_1)$ is relatively strongly ample (i.e. $\cO_{\tX}(r \tG_1)$ is a relatively very ample line bundle and $R^i \tf_* \cO_{\tX}(mr\tG_1)=0$ for all integers $i,m>0$ \cite[Section 8.38]{Kol23}). 
		Since $\mu_d$ is linearly reductive, there exists a free $\tR$-module basis $s_0,\ldots, s_{N}$ of $H^0(\tX, r\tG_1)$ such that each $s_i$ is a  $\mu_d$-eigenvector i.e. there exist integers $0 \leq w_i \leq d-1$  such that $\xi \in \mu_d(\bk) \subset \bk^\times $ acts via $\xi \cdot s_i  = \xi^{w_i} s_i$. 
		The basis induces an embedding $\tX\hookrightarrow \bP^N_{\tR}$, which is equivariant with respect to the $\mu_d$-action on $\bP^N_{x_i}$ such that
		$\xi \cdot [x_0 : \cdots : x_N] = [ \xi^{-w_0} x_0: \cdots : \xi^{-w_N} x_N]$.
		
		Following \cite[Section 8.5]{Kol23}, consider the contravariant functor from $\bk$-schemes to sets that sends  a $\bk$-scheme $S$  to the set of all quadruples
		\[
		(Y \subset \bP^{N}_S , H^1, H^2, H^3)
		\]
		where $Y \subset \bP^N_S$ is a subscheme 
		and $H^i \subset Y$ are  K-flat relative Mumford divisors \cite[Definition 7.1]{Kol23}
		such that $g:Y \to S$ is flat, proper, $Y_s$ is pure dimensional, $H^0(Y_s, \cO_{Y_s}) \cong k(s)$ for all $s\in S$, $\cO_{Y}(1)$ is strongly $g$-ample, and the natural map $\cO_{S}^{\oplus N+1} \to g_*\cO_{Y}(1)$ is an isomorphism.
		By \cite[Section 8.46]{Kol23}, the functor is representable by a scheme $T$ whose connected components are finite type. 
		Let 
		\[
		(Y_{T} \subset \bP^{N}_{T}, H^1_{T}, H^2_{T}, H^3_{T})
		\to T
		\]
		be the corresponding universal family with morphism $h_{T}: Y_{T} \to  T$.
		Let $D_T:= \tfrac{1}{r} H^3_T$.
		The data of 
		\[
		(\tX\subset \bP^N_{
			\tR}, \tG_1, \tG_2, r \tB)
		\]
		induces a morphism $ \rho : \Spec \, \tR \to T$.
		Let $Z$ denote the scheme-theoretic image of $\rho$.
		Note that the $\mu_d$-action on $\bP^N$ induces a $\mu_d$-action on $T$. 
		As the image of $\rho$ and its closure are stable under the $\mu_d$-action, $Z$ is a $\mu_d$-invariant closed set.
		
		We now show that there exists  a $\mu_d$-invariant affine open subscheme $U\subset Z$ containing $\im(\rho)$ such that 
		\begin{enumerate}
			\item[(i)] the pullback $(Y_{T} , D_{T})_U \to U$ is a family of slc CY pairs, 
			\item[(ii)]  $H^1_{U}$ and $H^2_U$ are relatively Cartier divisors,
			\item[(iii)]  $M_U:=\cO_{Y_U}(H^1_U -H^2_U)$ is relatively big and semiample over $U$ and the formation of ${h_U}_* (M_U^{\otimes m})$ commutes with arbitrary base change for all $m \geq 0$. 
		\end{enumerate}
		By \cite[Corollaries 3.22 and 3.30 and Theorem 4.42]{Kol23},
		there exists a collection of disjoint locally closed subschemes $Z_1,\ldots, Z_m$ of $Z$ such that a morphism $S\to Z$ from a reduced $\bk$-scheme $S$ factors through $\sqcup_{i=1}^m Z_i \hookrightarrow Z$  if and only if $(Y_S, D_S)\to S$ is a family of slc pairs, $H^1_S$ and $H^2_S$ are both relative Cartier divisors, and $r(K_{Y_S/S}+D_{Y_S}) \sim_S 0$.
		Since $\rho: \Spec\, \tR \to Z$ satisfies the above condition, after reordering the $Z_i$ we may assume that it factors through $Z_1$.
		Since $\rho$ has dense image in $Z$, $U_1:=Z_1$ must be an open subset of $Z$.
		As the  decomposition satisfies a universal property, the $\mu_d$-action permutes the $Z_i$. As $Z_1$ is the unique subscheme in the decomposition that is an open set of $Z$,  $U_1=Z_1$  is a $\mu_d$-invariant open subscheme.
		
		Now we  show that 
		\begin{equation}\label{eq:invpl}
			h^0(\tX_0, \tL^{\otimes r}_0) = h^0(\tX_{\tK}, \tL^{\otimes r}_{\tK}).
		\end{equation}
		As $\tR$ is not essentially of finite type over $\bk$ this does not follow immediately from Lemma \ref{lem:HX}. Instead, using a divisorial valuation centered over $\rho(0)$, we can construct a  pointed morphism from the spectrum of an essentially finite type DVR $(C,c) \to (U_1,\rho(0))$ that sends generic point to generic point.
		Thus $(Y_{C}, D_C) \to C$ is a family of slc CY pairs with normal generic fibers. 
		Furthermore $M_{C}$ is relatively big and nef, since relative bigness  can be checked on the generic fiber and nefness can be checked on the special fiber.
		As well, $M_{K(C)}$ is semiample as this property may be checked after a field extension.
		Thus Lemma \ref{lem:HX} implies that $h^0(Y_c,M^{\otimes r}_c)  = h^0(Y_{K(C)}, M^{\otimes r}_{K(C)})$. As these fibers differ by field extension from \eqref{eq:invpl}, the  equality holds.

		Returning to the construction of $U$, let
		\[
		U_2:= \{ s\in U_1\, \vert\, h^0( Y_s , M_s^{\otimes r})= h^0(Y_{K(Z)}, M_{K(Z)}^{\otimes r}) \}
		,\]
		which is an open set of $U_1$ by the semicontinuity of cohomology \cite[28.1.1]{Vakil}.
		By the previous paragraph, the open set contains $\im(\rho)$. 
		By Grauert's Theorem \cite[28.1.5]{Vakil}, the formation of   $(h_{U_2})_*(M_{U_2}^{\otimes r})$ commutes with arbitrary base change.
		Thus the locus 
		\[
		U_3:=\{s \in  U_2 \, \vert\, M_s^{\otimes r } \text{ is base point free} \}
		\]
		is an open set of $U_2$ and contains $\im(\rho)$. 
		Note that $M_{U_3}$ is relatively semiample by the definition of $U_3$ and the previous  base change property.
		Let $U\subset U_3$ be an  affine open set containing $\rho(0)$. Note that $\rho(0)$ is $\mu_d$-fixed and $Z$ is separated. 
		Thus by replacing $U$ with the intersection of its orbit by $\mu_d$, we may additionally assume that $U$ is $\mu_d$-invariant.
		By construction $U$ satisfies (i) and (ii).
		Furthermore, $M_U$ is relatively semiample by the above discussion and relatively big as  the restriction to the generic fiber is big.
		
		To verify the base change part of (iii), by Grauert's Theorem \cite[28.1.5]{Vakil}, it suffices to show that  
		\[
		h^0(Y_s , M^{\otimes m}_s)= h^0(Y_{K(Z)}, M^{\otimes m}_{K(Z)})
		\]
		for all integers $m\geq 0$ and $s\in U$.
		To see the latter, fix a point $s\in U$. If $s$ is the generic point of $U$, this holds trivially. 
		Otherwise choose, as before, a pointed morphism from the spectrum of a DVR essentially of finite type over $\bk$ $h:(C,c)\to (U,s)$ that sends generic point to generic point.  
		Note that $(Y_C,D_C)\to C$ is a family of slc CY pairs with normal generic fiber, $M_C$ is relatively big, as this can be checked on the generic fiber, and relatively semiample. 
		Now 
		\[
		h^0(Y_s , M_s^{\otimes m})
		=
		h^0(Y_c , M_c^{\otimes m})
		=
		h^0(Y_{K(C)} , M_{K(C)}^{\otimes m}) 
		=
		h^0(Y_{K(Z)}, M_{K(Z)}^{\otimes m})
		\]
		where the first and third equalities hold by flat base change and the second by Lemma \ref{lem:HX}. 
		Therefore  (iii) holds.
		
		We now construct the approximation. 
		Consider the $\mu_d$-equivariant map 
		$\cO_U (U) \to \tR$
		induced by $\rho$.
		Let $U^{\rm lc}\subset U$ be the open locus parametrizing lc pairs.
		Since the generic fiber of $(\tX,\tB) \to \Spec\, \tR$ is normal,  the image of $\rho:\Spec\, \tR \to  U$ is not contained in $U\setminus U^{\rm lc}$.
		Thus there exists $g\in \cO(U)$ such that $g$ vanishes  on $U \setminus U^{\rm lc}$ and $\rho^*g\neq 0$.  
		Choose an integer $l'\geq l$ such that $\rho^*g\not\equiv 0$ modulo $t^{l'}$.
		By Lemma \ref{l:equivgreenberg} applied to the $\mu_d$-equivariant map $\cO_{U}(U) \to \tR$, there exists a DVR $R'$  with a uniformizer $t$ that is the local ring of a smooth pointed curve with an isomorphism $\widehat{R'} \cong R$ that sends uniformizer to uniformizer and a $\mu_d$-equivariant ring homomorphism
		\[ 
		\cO_U(U) \to \tR' := R'[t^{1/d}]
		\]
		that agrees with $\cO_{U}(U)\to \tR$ modulo $t^{l'}$.
		Write $\rho': \Spec\, \tR' \to U$ for the induced map.
		Since $\rho'^*g \equiv \rho^*g$ modulo $t^{l'}$ and the latter is not equivalent to zero, $\rho'^*g\neq 0$.
		Thus the image of $\rho'$  is not contained in $U\setminus U^{\lc}$.
		
		Now we construct a family over $\Spec\, \tR'$.
		Pulling back the universal family by  $\Spec\,\tR' \to U$ produces a $\mu_d$-equivariant  family of slc CY pairs $\tf':(\tX',\tB') \to \Spec\,\tR'  $ 
		and relative Cartier divisors $\tG'_1$ and $\tG'_2$ such that $\tL':= \cO_{\tX'}(\tG'_1-\tG'_2)$ is relatively big and semiample and  the formation of $\tf'_* \tL'^{\otimes m}$ commutes with arbitrary base change for all $m\geq 0$,  and
		a $\mu_d$-equivariant isomorphism 
		\[
		\widetilde{\phi}:
		\tX \times_{\tR} \Spec\,\tR/ (t^{l}) \overset{\cong}{\longrightarrow} \tX'\times_{\tR'} \Spec\,\tR'/ (t^{l})
		\]
		that sends the base change of $\tL$  to  that of $\tL'$. 
		Since  the image $\Spec \, \tR' \to U$ is not contained in $U\setminus U^{\rm lc}$,  the generic fiber of $(\tX',\tB') \to \Spec\, \tR'$ is lc. 
		Using Propositions \ref{p:kollarsmoothslcfam} and \ref{p:kollarsmoothslcfamgen}, we see that
		$(\tX',\tB'+ \tX'_0)$ is lc.
		Let 
		\[
		q': (\tX',\tB') \to (X',B') := (\tX',\tB') / \mu_d
		\]
		denote the quotient. 
		Since 
		\[
		(\tX',\tB'+ \tX'_{0,\red}) \to (X',B'+X'_{0,\red})
		\]
		is crepant,  $(X',B'+X'_{0,\red})$ is a projective lc  CY pair over $R'$.
		
		Now we compare $X$ and $X'$.
		Since $\mu_d$ is linearly reductive, taking invariants is exact and commutes with base change. 
		Thus the reductions of $X$ and $X'$ modulo $t^l$ are isomorphic to the $\mu_d$-quotients of the reductions of $\tX$ and $\tX'$. 
		Thus the  the above isomorphism induces an isomorphism 
		\[
		\phi: X_l:=X \times_R \Spec\ R/(t^l) \overset{\cong}{\longrightarrow} X'\times_{R'}\Spec\, R'/(t^l)=: X'_l
		\]
		as desired.
		
		Now we show that  $\tL'$ descends to a line bundle $L'$ on $X'$ that is isomorphic to $L$ modulo $t^l$.
		First note that there are natural  isomorphisms
		\[
		L 
		\cong 
		(q_* \cO_{\tX})^{\mu_d}  \otimes L 
		\cong 
		(q_* \cO_{\tX} \otimes  L)^{\mu_d} 
		\cong 
		q_*( \cO_{\tX} \otimes q^*L)^{\mu_d}
		\cong 
		(q_* \tL)^{\mu_d}
		.\]
		Let $L' := (q'_* \tL')^{\mu_d}$, which is a coherent sheaf on $X'$.
		Write $q_l$ and $q'_l$ for the reductions of $q$ and $q'$ modulo $t^l$. 
		Using again that  taking invariants  commutes with base change, we see that 
		\[
		L_l
		\cong 
		({q_l}_* ( \tL / t^l \tL))^{\mu_d}
		\cong 
		\phi^* ({q'_l}_*( \tL' /  t^l \tL'))^{\mu_d}
		\cong 
		\phi^*L'_l
		.\]
		Reducing modulo $t$, we see that the fibers of $(X,L)$ and $(X',L')$ are isomorphic modulo $t$. 
		Thus the fiber of $L'$ over $0\in \Spec\, R'$ is a line bundle.
		Since $L'$ is torsion free,  Nakayama's Lemma implies that $L'$ is a line bundle in a neighborhood of the fiber over $0\in \Spec\,R'$.
		Since $X'\to \Spec\,R'$ is proper and the locus where $L'$ fails to be a line bundle is a closed set of $X'$, $L'$ is a line bundle. 
		Thus $L'$ is a line bundle and $L_l \cong \phi^*L'_l$.

		By the definition of $L'$, there is natural  morphism $q'^*L' \to \tL'$.
		Since it is an isomorphism modulo $t$, it is an isomorphism in a neighborhood of the special fiber by Nakayama's Lemma. 
		Again using that  $\tX'\to \Spec \tR' $ is proper and the locus where it fails to be an isomorphism is closed, 
		we deduce that $q'^*L' \to \tL'$  is an isomorphism.

		We now prove the base change property. 
		By (iii),  $\tf'_* \tL'^{\otimes m}$ commutes with arbitrary base change. 
		Thus the natural map
		\[
		H^0(\tX', \tL'^{\otimes m})\otimes_{\tR'} \tR'/ (t^i)
		\to 
		H^0(\tX'_i, \tL_i'^{\otimes m})
		\]
		is an isomorphism for all integers $i\geq 1$ and $m\geq 0$.
		As this map is $\mu_d$-equivariant and the degree $0$ summand under the $\bZ/d\bZ$ grading is isomorphic to
		\[
		H^0(X', L'^{\otimes m})\otimes_{R'} R'/ (t^i)
		\to 
		H^0(X'_i, {L'_i}^{\otimes m})
		,\]
		the latter map is also an isomorphism.
		As well, the same argument implies the corresponding statement for $X'$ and $L'$ replaced by $X$ and $L$. 
		Additionally, since $\tX'\to X'$ is finite and $\tL'$ is relatively big, $L'$ is relatively big. 
		Using that $L'_0\cong L_0$, which is semiample and the above isomorphism with $i=1$, we see that $L'$ is relatively semiample.
	\end{proof}

	\subsection{Proof of main theorem}
	
	We now prove the existence of valuatively independent bases as a consequence of approximation and the algebraic case.

	\begin{thm}\label{t:semiamplevalind}
		Let $R :=  \bk[\![t]\!]$ and $K={\rm Frac}(R) \cong  \bk(\!(t)\!)$.
		Let $(X_K,B_K)$ be an lc CY pair over $K$ and $L_K$ be an ample line bundle on $X_K$ with a metric. 
		
		If there exists an lc model $(X,B)$ of $(X_K,B_K)$ such that $L_K$ extends to a line bundle $L$ on $X$  that is relatively semiample, then the following hold: 
		\begin{enumerate}
			\item There exists a  basis   $(\theta_i)_{i=1}^{N_m}$ of  $H^0(X_K, L_K^{\otimes m})$ satisfying valuative independence for each integer $m\geq0$. 
			\item If   $X_0$ is reduced, the bases in (1) can be chosen so that they are closed under tropical scaling. 
		\end{enumerate}
	\end{thm}

	\begin{proof}
		For an integer $m\geq 0$, let $V_m:= H^0(X,L^{\otimes m})$.
		By Remark \ref{r:indofmetric}, we may assume that the metric on $L_K$ is induced by the extension $L$.
		We now begin by stating the following claim that we will deduce from the algebraic case. 
		\medskip
		
		\noindent \emph{Claim:} For any collection of divisorial valuations $v_1,\ldots, v_r$ in $\Sk(X_K,B_K)$,  the following hold:
		\begin{enumerate}
			\item The filtrations  $F_1,\ldots, F_r$ of $V_m$ are simultaneously diagonalizable for each $m\geq0$.
			\item If $X_0$ is reduced, then $\bigoplus_{m \geq 0} {\rm gr}_{F_1,\ldots, F_r} V_m$ is reduced.
		\end{enumerate}
		Above $F_j := F_{v_j}$ denotes the filtration of $V_m$ induced by $v_j$.
		
		\begin{proof}[Proof of Claim:]
			Fix any  integer $d\geq 0$. 
			We will verify the  truncated versions of the claim:
			\begin{enumerate}
				\item[($1_d$)] The filtrations $F_1,\ldots, F_r$  of $V_m$ are simultaneously diagonalizable for each  $m \in \{0,\ldots, d \}$.
				\item[($2_d$)] If $X_0$ is reduced, then for any  nonzero homogeneous element $ s\in\bigoplus_{m\in \bN} {\rm gr}_{F_1,\ldots, F_r} V_m$
				and integer $k \geq 0$ such that $k\deg(s)\leq d$, $s^k \neq 0$.
			\end{enumerate}
			These statements follow  naturally  from approximation.
			
			Before proceeding, we construct a bound needed in the argument. Since  $F_1,\ldots, F_r$  are filtrations of $V_m $ and $R$ is complete, each filtration $F_i$ is bounded by Lemma \ref{l:completebounded}.
			Thus  there exists an integer $l\geq1$ such that 
			$F_j^{l} V_m \subset t V_m$
			for all $1\leq j \leq r$ and $m \in \{0,1,\ldots, d\}$.
			
			We now construct a birational model and an algebraic approximation. 
			Since $v_1,\ldots, v_r$ are divisorial, 
			Proposition \ref{p:dltmodexist} implies there exists a dlt modification 
			\[
			f:(Y,B_Y+Y_{0,\red}) \to (X,B +X_{0,\red})
			\]
			such that there exist prime divisors $E_1,\ldots, E_r$ on $Y$ such that $v_j =b_j^{-1}\ord_{E_j}$ with $b_j= {\rm mult}_{E_j}(Y_0)$. 
			Since the above morphism of pairs is crepant, $(Y,B_Y+Y_{0,\red})$ is a projective lc CY pair over $R$. 
			Since $f$ is birational, $L_Y:= f^*L$ is relatively big and semiample over $R$. 
			Thus Proposition \ref{p:approxCY} implies that there exists
			\begin{itemize}
				\item[(i)] a DVR $R'$ which is a local ring of a smooth curve with an isomorphism $\widehat{R'}\cong R $ that sends the image of a uniformizer $t \in R'$ to the uniformizer $t\in R$ and
				\item[(ii)]  an lc CY pair $(Y'_{K'}, B'_{Y'_{K'}}) $ over $K':= {\rm Frac}(R')$ with an lc model $(Y',B'_{Y'})$ with a relatively big and semiample line bundle $L'_{Y'}$ with an isomorphism 
				\begin{equation}\label{e:YY'isom}
					Y \times_{R} \Spec\, R/ (t^{l})  \longrightarrow 
					Y' \times_{R'} \Spec\, R'/ (t^{l})
				\end{equation}
				that sends the reduction of $L_Y$ modulo $t^l$ to the reduction of $L'_{Y'}$ modulo $t^l$.
			\end{itemize}
			Furthermore we may assume that  $L_Y^{\otimes m}$ and $L^{'\otimes m}_{Y'}$ satisfy the base change property in Proposition \ref{p:approxCY}.
			The isomorphism \eqref{e:YY'isom} induces a bijection between the irreducible components of $Y_0$ and $Y'_0$. 
			Write $E'_j\subset Y'_0$ for the irreducible component corresponding to $E_j \subset Y_0$ and $\eta'_j \in Y'_0$ for its generic point.
			Let $f':Y'\to X'$ denote the ample model of $L'_{Y'}$ and $B':= f'_* B'_{Y'}$. 
			Since $(Y',B'_{Y'}+Y'_{0,\red})$ is an lc CY pair over $R'$, 
			\[
			(Y',B'_{Y'}+Y'_{0,\red})\to 
			(X',B'+X'_{0,\red})
			\]
			is crepant and the pair on the right hand side is a projective lc  CY pair over $R'$. 
			Using that the map is crepant, we see that each  $v'_j:= (b_j)^{-1} \ord_{E'_j}
			$ is in $\Sk( X'_{K'},B'_{K'})$ where $K':={\rm Frac}(R')$.

			Now we analyze the line bundle $L'_{Y'}$ on $Y'$ and its space of sections. 
			Fix a Cartier divisor $H'_{Y'}$ in the linear equivalence class of $L'_{Y'}$ such that $H'_{Y'}$ does not contain an irreducible component of $Y'_0$, and let  $H':= f'_* H'_{Y'} $. By the same argument as in the Proof of Theorem \ref{t:affineCYposbound}, $H'_{Y'} = f'^* H'$ and $H'$ is relatively ample.

			Now we compare the polarizations.  
			First we have natural isomorphisms 
			\[
			V_m \cong H^0(Y, L_Y^{\otimes m}) \quad \text{ and } \quad 
			V'_m := H^0(X',mH')  \cong H^0(Y', L_{Y'}^{'\otimes m})
			.\]
			By the base change statement, we further have 
			\[
			V_m/ t^l V_m 
			\cong 
			H^0(Y_l, L_{Y_l}^{\otimes m}) \quad \text{ and } \quad
			V'_m/ t^l V'_m 
			\cong 
			H^0(Y'_l, L_{Y'_l}^{'\otimes m})
			.\]
			Here $Y_l:= Y \times_{R} R/(t^{l})$ and $Y'_l:= Y'\times_{R'} R'/(t^{l})$, while $L_{Y_l}$ and $L'_{Y'_l}$ are the pullbacks of $L_{Y}$ and $L'_{Y'}$ to these schemes.
			Thus the isomorphism in \eqref{e:YY'isom} induces an isomorphism 
			\[
			\phi_{k,m}: V_{m}/ t^k V_m \to V'_m /t^k V'_m
			\]
			for each $0\leq k\leq l $.
			
			We now describe the filtrations modulo $t^l$ geometrically.
			We claim that if $0\leq \la \leq l$, then
			\begin{equation}\label{e:F_jmodi}
				F_j^\la (V_m/ t^l V_m) = \ker (V_m / t^l V_m \longrightarrow (L_{Y_l}^{\otimes m})_{\eta_j} \otimes \cO_{Y_l, \eta_j}/ \fm_{l,\eta_j}^{\lceil \la b_j \rceil}),
			\end{equation}
			where $\eta_j\in Y$ is the generic point of $E_j$ and $\fm_{l,\eta_j}\subset \cO_{Y_l, \eta_j}$ denotes the maximal ideal. 
			To see this, observe that we have a commutative diagram 
			\[
			\begin{tikzcd}
				0\arrow[r]&F_j^\la V_m \arrow[r] \arrow[d,->>] & V_m \arrow[r] \arrow[d,->>]& (L_Y^{\otimes m})_{\eta_j}\otimes \cO_{Y,\eta_j}/ \fm_{\eta_j}^{\lceil \la b_j \rceil}\arrow[d,"\cong"]  \\
				0\arrow[r]&F_j^\la (V_m/t^{l}V_m) \arrow[r,hook] & V_m/t^{l} V_m \arrow[r] & (L_{Y_l}^{\otimes m})_{\eta_j}\otimes \cO_{Y_l,\eta_j}/ \fm_{l,\eta_j}^{\lceil \la b_j \rceil} 
			\end{tikzcd},
			\]
			where the first row is exact and the surjectivity and injectivity marked in the diagram follow from definitions.
			Furthermore, if $\la \leq l$, then $t^{l} \in  \fm_{\eta_j}^{\lceil \la b_j\rceil}$, and so the rightmost vertical arrow is an isomorphism. 
			Therefore, the bottom row is also exact, which implies that \eqref{e:F_jmodi} holds in this case.
			By the same argument,
			we have that 
			\begin{equation}\label{e:F_jmodiprime}
				F_j^{'\la} (V'_m/ t^l V'_m) = \ker \big(V'_m / t^l V'
				_m \longrightarrow  
				({L'}_{Y'_l}^{\otimes m})_{\eta'_j} \otimes \cO_{{Y'_l}, \eta'_j}/ {\fm'}_{l,\eta'_j}^{\lceil \la b_j \rceil}\big),
			\end{equation}
			for all  $0\leq \la \leq l$, 
			where we use similar notation as before, i.e. $F'_j:= F_{v'_j}$ denotes the filtration induced by $v'_j$ with respect to the metric induced by $L'_{Y'}$, $Y'_l = Y'\times_{R'} R'/(t^{l})$,  $L'_{Y'_l}$ is the pullback of $L'_{Y'}$ to $Y'_l$ and $\fm'_{l,\eta'_j} \subset \cO_{Y'_l, \eta'_j}$ is the maximal ideal. 
			Using \eqref{e:F_jmodi}, \eqref{e:YY'isom}, and \eqref{e:F_jmodiprime}, we see that the isomorphism  $\phi_{l,m}$ satisfies 
			\begin{equation}\label{e:F=F'}
				\phi_{l,m}(F^\la_j (V_m/t^{l}V_m)) =
				F'^\la_j(V'_m /t^{l}V'_m)
			\end{equation}
			for all $\la \leq l$ and $0\leq m \leq d$. 
			Now observe that 
			\[
			F_j^0  V_m = V_m
			\quad \text{ and } \quad
			F_j^{l} V_m \subset t V_m
			, \]
			where the inclusion is by our choice of $l$. 
			Thus 
			\[
			F^{'\,0}_j  V'_m = V'_m
			\quad \text{ and } \quad
			F^{'\,l}_j V'_m \subset t V'_m
			, \]
			where the  inclusion holds by applying \eqref{e:F=F'} in the case when $\la=l$.
			Using Definition \ref{d:filtration}, we then see that \eqref{e:F=F'} holds for all $\la \in \bR$ and $0 \leq m \leq d$.

			To finish the proof of the claim, note that  Proposition \ref{p:diagfiltcypol} implies that
			\begin{enumerate}
				\item[(1')] The filtrations $F'_1,\ldots, F'_r$ of $V'_m$ are simultaneously diagonalizable. 
				\item[(2')] If $X'_0$ is reduced, then $\bigoplus_{m \geq 0} {\rm gr}_{F'_1,\ldots, F'_r}(V'_m)$ is reduced.
			\end{enumerate}
			By (1') and reducing modulo $t^l$, we see that the filtrations $F'_1,\ldots, F'_r$ of $V'_m$ are simultaneously diagonalizable modulo $t^l$.
			By \eqref{e:F=F'},  $F_1,\ldots, F_r$  are also  simultaneously diagonalizable modulo $t^l$.
			By Proposition \ref{p:diagordn},  $(1_d)$ holds.
			As this statement holds for arbitrary $d\geq 0$, (1) holds.
			
			We now verify that (2) holds. 
			Assume that $X_0$ is reduced.
			We claim that $X'_0$ is also reduced. 
			Indeed, observe that there is a ring isomorphism and injective ring homomorphism
			\[
			\bigoplus_{m \in \bN} V'_m / t V'_m
			\overset{\cong}{\longrightarrow}
			\bigoplus_{m \in \bN} V_m / t V_m
			\lhook\joinrel\longrightarrow
			\bigoplus_{m \in \bN} H^0(X_0,L_0^{\otimes m})
			.\]
			Since $X_0$ is reduced, the ring on the right is reduced. 
			Thus the ring on the left is also reduced.
			Since $X'_0$ is Proj of the ring on the left, $X'_0$ is reduced. 
			Thus (2') implies that $\bigoplus_{m \in \bN} {\rm gr}_{F'_1,\ldots, F'_r}(V'_m)$ is reduced.
			
			For $0\leq m \leq d$, we have isomorphisms
			\begin{align*}
				{\rm gr}_{F_1,\ldots, F_r}(V_m) &\cong {\rm gr}_{F_1,\ldots, F_r}(V_m/tV_m)\\
				&\cong {\rm gr}_{F'_1,\ldots, F'_r}(V'_m/tV'_m)\\
				& \cong {\rm gr}_{F'_1,\ldots, F'_r}(V'_m)
			\end{align*}
			where the second and third objects are the graded vector spaces associated to the induced filtrations of the vector spaces $V_m/t V_m$ and $V'_m/tV'_m$. 
			Above the first and third isomorphisms use that the relevant collections of filtrations are simultaneously  diagonalizable and the second follows from reducing the equality in \eqref{e:F=F'} modulo $t$. 
			Since the isomorphisms respect multiplication, ($2_d$) holds. 
			Since statement ($2_d$) holds for arbitrary $d\geq 0$,  statement (2) holds. \end{proof}

		Using the claim, the exact same argument as in the  Proof of Theorem \ref{t:valindbasisalg}   implies that, for each integer $m\geq 0$,
		there exists a free $R$-module basis $(\theta_{m,1},\ldots, \theta_{m,N_m})$  of  $H^0(X,L^{\otimes m})$ that satisfies valuative independence.
		This proves (1).

		Finally, we show that the basis is closed under tropical scaling when $X_0$ is reduced. 
		Arguing by contradiction, assume that there exist  integers $m,k\geq 0$ and $1\leq i \leq N_m$ such that 
		\[
		k\theta_{m,i}^{\trop} \neq \theta_{mk,j}^{\trop}
		\]
		for all $ 1\leq j \leq N_{km}$.
		Since tropicalizations are continuous and divisorial valuations are dense in the skeleton, for each $ 1\leq j \leq N_{km}$, there exists a divisorial valuation $ v_j\in \Sk(X_K,B_K)$ such that 
		\[
		k \theta_{m,i}^{\trop}(v_j) \neq \theta_{mk,j}^{\trop}(v_j)
		.
		\]
		Now let $r:= N_{km}$ and consider the filtrations $F_1,\ldots, F_r$ where $F_j := F_{v_j}$.
		Since
		$(\theta_{m,i})_{i=1}^{N_m}$ diagonalizes the filtrations $F_1,\ldots, F_r$ of $V_m$,
		\[
		0\neq 
		\overline{\theta}_{m,i} 
		\in 
		{\rm gr}_{F_1,\ldots, F_r}(V_m)
		\]
		and is in degree $\ord_{F}(\theta_{m,i})\in \bR^r$. 
		If $X_0$ is reduced, then (2) implies that 
		\[
		0\neq 
		\overline{\theta}_{m,i}^k
		\in 
		{\rm gr}_{F_1,\ldots, F_r}(V_{km})
		\]
		and is in degree $k\ord_{F}(\theta_{m,i})\in \bR^r$.
		Since the images of $(\theta_{km,i})_{i=1}^{N_{km}}$
		in ${\rm gr}_{F_1,\ldots, F_r}(V_{km})$ form a homogeneous basis, 
		there exists $1\leq j \leq N_{km}$ such that $\ord_{F}(\theta_{km,j}) = k\ord_{F}(\theta_{m,i})$. 
		Taking the value in the $j$-th entry gives that 
		\[
		k\theta_{m,i}^\trop(v_j)=k v_j ( \theta_{m,i})
		=
		k \ord_{F_j}(\theta_{m,i})
		=
		\ord_{F_j}(\theta_{km,j})
		=
		v_j( \theta_{km,j})
		=
		\theta_{km,j}^{\trop}(v_j)
		,\]
		which is a contradiction.
		Thus the basis is closed under tropical scaling.
	\end{proof}

	\begin{proof}[Proof of Theorem \ref{thm:val-indep}]
		By Proposition \ref{lem:semistable-reduction}, there exists an lc model $(X,B)$ of $(X_K,B_K)$ and an integer $r>0$ such that $L_K^{\otimes r}$ extends to a relatively ample line bundle $M$ on $X$.
		By Theorem \ref{t:semiamplevalind} applied to the line bundle $L_K^{\otimes r}$, $H^0(X_K,L_K^{\otimes m})$ admits a valuatively independent basis $(\theta_{m,i})_{i=1}^{N_{m}}$
		for each integer $m\geq 0$ divisible by $r$. 
	\end{proof}

	\bibliographystyle{alpha}
	\bibliography{ref}
	
\end{document}